\documentclass[11pt,oneside,reqno]{amsart}
\usepackage{geometry}
\usepackage[utf8]{inputenc}
\usepackage{amstext,latexsym,amsbsy,amsmath,amssymb,amsthm,mathtools,relsize,geometry,enumerate}
\usepackage{hyperref}
\hypersetup{
  colorlinks   = true, 
  urlcolor     = blue, 
  linkcolor    = red, 
  citecolor   = red 
}
\usepackage{mathtools,enumerate,enumitem}
\usepackage{tikz} 
\usepackage{float}
\usepackage{graphicx}
\usepackage{epstopdf}
\setkeys{Gin}{width=\linewidth,totalheight=\textheight,keepaspectratio}
\graphicspath{{./images/}}

\usepackage{pgfplots}
\pgfplotsset{compat=1.10}
\usepgfplotslibrary{fillbetween}
\usepackage{asymptote}
\usetikzlibrary{patterns}

\usepackage{multicol}
\usepackage{booktabs}
\usepackage{mathrsfs}
\usepackage{color}

\newcommand{\close}{\!\!\!}

\newcommand{\rbb}{\mathbf{R}}
\newcommand{\RR}{\mathbf{R}}

\renewcommand{\L}{\mathcal{L}}

\newcommand{\B}{\mathcal{B}}
\newcommand{\A}{\mathcal{A}}
\newcommand{\Q}{\mathcal{Q}}

\newcommand{\R}{\mathcal{R}}
\newcommand{\PP}{\mathcal{P}}
\newcommand{\Rbf}{\boldsymbol{R}}
\newcommand{\la}{\langle}
\newcommand{\ra}{\rangle}
\newcommand{\N}{\mathcal{N}}

\newcommand{\xbb}{\boldsymbol{x}}
\newcommand{\ybb}{\boldsymbol{y}}
\newcommand{\ubf}{\boldsymbol{u}}
\newcommand{\xbf}{\mathbf{x}}
\newcommand{\ybf}{\mathbf{y}}
\newcommand{\zbf}{\mathbf{z}}
\newcommand{\xbftd}{\widetilde{\mathbf{x}}}

\newcommand{\xtil}{\widetilde{x}}
\newcommand{\ytil}{\widetilde{y}}

\newcommand{\ztil}{\widetilde{z}}
\newcommand{\ftil}{\widetilde{\varphi}}

\newcommand{\ghat}{\widehat{\psi}}
\newcommand{\ktil}{\widetilde{k}}

\newcommand{\Phitil}{\widetilde{\Phi}}

\newcommand{\alphah}{\alpha_h}
\newcommand{\gammah}{\widetilde{\gamma}}

\newcommand{\fO}{\Phi_\text{O}}

\newcommand{\RO}{\mathcal{R}_\text{O}}
\newcommand{\RI}{\mathcal{R}_\text{I}}

\newcommand{\mi}{\wedge}

\renewcommand{\d}{\text{d}}
\newcommand{\dt}{\text{dt}}
\newcommand{\dW}{\text{d}W}
\newcommand{\domain}{\mathcal{O}}

\newcommand{\f}{\varphi}
\newcommand{\g}{\psi}

\newcommand{\Sbf}{\mathbf{S}}

\newcommand{\E}{\mathbf{E}}
\renewcommand{\P}{\mathbf{P}}

\newcommand{\M}{\mathcal{M}}
\newcommand{\T}{\mathcal{T}}

\newcommand{\Wbf}{\mathbf{W}}

\numberwithin{equation}{section}
\theoremstyle{plain}
\newtheorem{theorem}{Theorem}[section]
\newtheorem{corollary}[theorem]{Corollary}

\newtheorem{lemma}[theorem]{Lemma}
\newtheorem{choices}[theorem]{Choice}

\newtheorem{proposition}[theorem]{Proposition}
\newtheorem{definition}[theorem]{Definition}

\newtheorem{remark}[theorem]{Remark}

\title[Turbulent transport of heavy particles in rough flows]{Stability and invariant measure asymptotics in a model for heavy particles in rough turbulent flows}
\author[ D.~Herzog and H.~Nguyen]{David P.~Herzog$^1$ and Hung D.~Nguyen$^2$}

\address{$^1$ Department of Mathematics, Iowa State University, Iowa, USA}
\address{$^2$ Department of Mathematics, University of Tennessee, Knoxville, TN, USA}

\begin{document}

\begin{abstract}
We study a system of Skorokhod stochastic differential equations (SDEs) modeling 
the pairwise dispersion (in spatial dimension $d=2$) of inertial particles transported by a rough turbulent flow with H\"{o}lder exponent $h\in (0,1)$. Under the assumption that $h>0$ is sufficiently small, we use Lyapunov methods and control theory to show that the Markovian system is nonexplosive and has a unique, exponentially attractive invariant probability measure.  Furthermore, our Lyapunov construction is radially sharp and gives partial confirmation on a predicted asymptotic behavior with respect to the H\"{o}lder exponent $h$ of the invariant probability measure.  A physical interpretation of the asymptotics is that intermittent clustering is weakened when the carrier flow is sufficiently rough.

\end{abstract}
\maketitle
\section{Introduction} \label{sec:intro}

In this paper, we rigorously analyze a model for the dynamics of two inertial (i.e. \emph{heavy}) particles transported by a random fluid velocity field.  Our goal is to understand the relative spatial statistics of the particles as they depend on the H\"{o}lder exponent $h\in (0,1]$ of the carrier fluid.  While the case of smooth fields, i.e. when $h=1$, has been extensively studied from a rigorous mathematical standpoint~\cite{athreya2012propagating, birrell2012transition,gawedzki2011ergodic,herzog2015noise}, our efforts focus on the case of a \emph{rough} field; that is, when the H\"{o}lder exponent satisfies $h\in (0,1)$, where little is known rigorously.

Unlike simple point-like tracers, inertial particles have finite size and mass different than that of the carrier fluid.  Consequently, they encounter viscous drag forces and have significantly different dynamics.  As a concrete example, inertial particles form strongly inhomogeneous spatial patterns in the transporting fluid, giving rise to the phenomenon known as \emph{intermittent clustering}.  Understanding this phenomenon is of paramount importance and has been studied in a number of settings under various assumptions~\cite{bec2005multifractal,
bec2007clustering,bec2007heavy, bec2008stochastic,duncan2005clustering, elperin1996self, falkovich2002acceleration,fouxon2008separation, horvai2005lyapunov, maxey1983equation, mehlig2004coagulation, piterbarg2002top, wilkinson2003path}.  

As opposed to smooth carrier flows, less attention from both physical and mathematical perspectives has been paid to understanding the two-point statistics of inertial particles advected by rough velocity fields~\cite{bec2007clustering}.  See, however, the few works~\cite{BDG_2004, CGV_2006, FSS_2005, SS_2002} where rough flows are treated.  In this setting, the carrier fluid is above the Kolmogorov length scale, which is the smallest active length scale of the flow, and the precise mechanisms that induce both separation and clustering between particles are less understood.  However, significant insight into clustering above the Kolmogorov length scale was made in the work~\cite{bec2007clustering}, where the model studied in this paper was first analyzed using heuristic scaling arguments.  There, as opposed to smooth flows, it was shown that in this model particles do not form fractal clusters.  Rather, clustering is actually weakened by the roughness of the fluid, as measured by power-law moment estimates on the steady-state marginal distribution of the relative velocity of two particles~\cite{bec2007clustering}.  Our goal in this paper will be to rigorously establish these moment estimates.

Extrapolating these results to realistic turbulent flows suggests similar power-law type behavior. However, while the model studied in this paper has remnants of real world turbulence, in the sense that the particles experience both ejection from eddies and dissipative forces, the specific dynamics here focuses more so on the latter and furthermore ignores gravitational effects.  Thus, depending on the delicate balance of these neglected effects, this particular power-law behavior may not be characteristic of real turbulent flows.

\subsection{The equations}
Assuming particles are
small and spherical and gravitational effects can be neglected, a simplified description
for an inertial particle's displacement is given by the Newtonian equation~\cite{bec2007clustering, bec2007heavy,  maxey1983equation, mehlig2004coagulation, wilkinson2003path}
\begin{equation} \label{eqn:turbulence.equation}
\ddot{\xbb} =-\frac{1}{\tau}[\dot{\xbb}- \ubf(\xbb,t)],
\end{equation}
where in the relation above, $\tau>0$ is the Stokes time and $\ubf$ is the fluid velocity field.  Note that~\eqref{eqn:turbulence.equation} is a modification of the Lagrangian equation \cite{bec2008stochastic} where the velocity field is coupled with a viscous drag force ($-\tau^{-1} \dot{\xbb}$) due to fluid bombardment.

Considering the behavior of two, non-interacting particles $\xbb$ and $\ybb$ satisfying~\eqref{eqn:turbulence.equation}, the \emph{particle separation} $\Sbf:=\ybb-\xbb$ obeys the equation
\begin{equation} \label{eqn:turbulence.distance.equation}
\ddot{\Sbf} =-\frac{1}{\tau}[\dot{\Sbf}-\delta \ubf(\Sbf,t)],
\end{equation}
where $\delta\ubf:=\ubf(\xbb+\Sbf,t)-\ubf(\xbb,t)$ is the fluid velocity difference.  Throughout, we assume that $\ubf$ is a stationary, homogeneous and isotropic Gaussian field satisfying \cite{bec2007clustering,kraichnan1968small}
\begin{equation} \label{eqn:turbulence.distance.equation:roughflow}
\la u_i(\xbb,t)u_j(\xbb',t')\ra = 2\mathrm{D}_0\delta^i_j-\mathrm{B}^i_j(\xbb-\xbb')\delta(t-t'),
\end{equation} 
where $\mathrm{D}_0$ is the velocity variance and the matrix function $\mathrm{B}$ is given by
\begin{equation} \label{eqn:turbulence.distance.equation:roughflow:B}
\mathrm{B}^i_j(\xbb)=D_1 |\xbb|^{2h}[(d-1+2h)\delta^i_j-2hx_ix_j/|\xbb|^2 ].
\end{equation}
In the above, $h\in(0,1)$ is the \emph{H\"{o}lder exponent} of the velocity field, as discussed above, $D_1$ is a nonzero constant measuring the intensity of the turbulence, and $\langle \cdots \rangle$ denotes the expectation.  Throughout the paper, we restrict to the case when the spatial dimension $d=2$; that is, $\Sbf=(S_1, S_2)$.

Observe that we may equivalently express relation~\eqref{eqn:turbulence.distance.equation} as an It\^{o} stochastic differential equation for the pair $(\Sbf, \dot{\Sbf})$ as
\begin{align}
\label{eqn:foSDEdiff}
\d\Sbf& = \dot{\Sbf} \, \d t, \\
\nonumber \d\dot{\Sbf}&= - \frac{\dot{\Sbf}}{\tau} \, \d t + \frac{\Sigma(\Sbf)}{\tau}\,  \d\Wbf_t ,
\end{align}
where $\Wbf_t$ is a standard, two-dimensional Brownian motion and $\Sigma(\Sbf)$ is a $2\times 2$ non-constant matrix given by 
\begin{align*}
\Sigma(\Sbf) = \begin{bmatrix}
\sqrt{2D_1} |\Sbf|^{h-1} S_1 & - \sqrt{2D_1} \sqrt{1+2h} |\Sbf|^{h-1} S_2 \\
\sqrt{2D_1} |\Sbf|^{h-1} S_2 & \sqrt{2D_1}\sqrt{1+2h} |\Sbf|^{h-1} S_1  
\end{bmatrix}.\end{align*}
In the above, $|\Sbf|$ denotes the Euclidean norm of $\Sbf$. Letting $\Rbf_t= \Sbf_{t\tau}$ and $\dot{\Rbf}_t=\dot{\Sbf}_{t\tau}$, we make the following change of variables as in~\cite{bec2007clustering}
\begin{equation} \label{eqn:change.variable:XYZ}
X =\frac{\tau}{L^2}\bigg(\frac{|\Rbf|}{L}\bigg)^{-(1+h)}\Rbf\cdot\dot{\Rbf},\qquad
Y =\frac{\tau}{L^2}\bigg(\frac{|\Rbf|}{L}\bigg)^{-(1+h)}(R_1 \dot{R}_2-R_2 \dot{R}_1),\qquad
Z=\bigg(\frac{|\Rbf|}{L}\bigg)^{1-h},
\end{equation}
where the dot denotes the time derivative, $X$ and $Y$ respectively denote the (dimensionless) longitudinal and transverse velocity differences and $L>0$ is a fluid boundary constant controlling the size of the particles' separation.  That is, we assume throughout that $|\Rbf| \leq L$ with reflective boundary conditions at $|\Rbf|=L$ (see the paragraph below and Section~\ref{sec:mainresult} for further details).  This is done because in unbounded carrier flows, the distance between two particles can grow indefinitely and thus never achieve stationarity \cite{bec2005multifractal,bec2007clustering,
bec2007heavy,bec2008stochastic}.  Although the added boundary conditions break self-similarity of the flow, clustering phenomena can be observed for $|\Rbf|\ll L$ by heuristic scaling arguments~\cite{bec2007clustering}.

Under the change of coordinates~\eqref{eqn:change.variable:XYZ}, the system~\eqref{eqn:foSDEdiff} becomes 
\begin{equation} \label{eqn:XYZ:no.reflection}
\begin{aligned}
\d X_t &=- X_t \, \d t -\frac{hX_t^2-Y_t^2}{Z_t} \d t+\sqrt{2a}\, \d W^1_t,\\
\d Y_t&= - Y_t \, \d s -(1+h)\frac{X_tY_t}{Z_t} \,\d s+\sqrt{2a(1+2h)}\d W^2_t,\\
\d Z_t&=(1-h)X_t\, \d t,
\end{aligned}
\end{equation}
where $a=a(L, D_1, \tau)>0$ is a positive constant and $(W^1_t,W^2_t)$ is a standard two-dimensional Brownian motion. Incorporating the imposed reflective boundary conditions at $|\Rbf|=L$, which in particular pushes the dynamics inward to the domain $\{|\Rbf|< L\}$ whenever the threshold $\{|\Rbf|= L\}$ is reached, translates to an additional reflective term at $\{Z=1\}$ in equation~\eqref{eqn:XYZ:no.reflection}.  This reflective term, in particular, supplements equation~\eqref{eqn:XYZ:no.reflection} to produce a Skorokhod SDE (see~\eqref{eqn:xyz} for further details).  Formally setting $h=1$ in the equation above and assuming $Z\equiv 1$, we arrive at the system, and variations thereof, studied rigorously in~\cite{athreya2012propagating,gawedzki2011ergodic, herzog2015noise, herzog2015noiseII}. In contrast to these works, our main focus here is to study the behavior of this Skorokhod equation; that is, equation~\eqref{eqn:XYZ:no.reflection} with the reflective boundary condition, when $h>0$ is small.

\subsection{History and results}
In the case of spatially smooth flows; that is, when $h=1$ and $Z\equiv 1$ in~\eqref{eqn:XYZ:no.reflection}, absence of finite-time explosion (for $(X,Y)$ in $\RR^2$) and the existence/uniqueness of an invariant probability measure was first established in~\cite{gawedzki2011ergodic}.  Similar results were also shown there in general dimensions $d\geq 3$, but in this setting the equation as well as the phase space changes slightly.  The arguments hinged on the existence of various Lyapunov functions, the constructions of which were involved since the underlying deterministic dynamics, i.e., when the noise coefficients are set to $0$,  has trajectories leading to finite-time explosion.  That is, if one starts the deterministic equation on the negative $X$ axis when $|X|$ is sufficiently large, the resulting system blows-up in finite time.  Thus, as a consequence of~\cite{gawedzki2011ergodic}, the noisy system stabilizes the underlying instability.

This work led to further studies on methods to better construct Lyapounov functions.  First, the work~\cite{athreya2012propagating} developed a propagation procedure, applied to the example of a smooth flow, which was used to prove that the unique invariant measure is exponentially attractive in a uniform sense with respect to the initial condition.  These methods were later improved in~\cite{herzog2015noise} and applied to a general family of planar flows yielding asymptotic estimates on the behavior of the invariant measure at the point at inifinity.  The paper~ \cite{birrell2012transition} analyzed a simplified model of the rough equation~\eqref{eqn:XYZ:no.reflection} where $Z$ was assumed to be a fixed constant but $h>0$ was assumed to be general.  Interestingly, the system is well behaved in the physical range of parameters, but if one slightly modifies the coefficients, as done there, the system explodes in finite time with positive probability.

An additional result of physical interest originally obtained in~\cite{gawedzki2011ergodic} is that, in the case of smooth flows, the invariant probability density (with respect to Lebesgue measure in $\RR^2$) $\rho(x,y)$ satisfies the asymptotic formula 
\begin{align}
\label{eqn:asyxy}
 \rho(x,y) = \frac{\rho_*}{r(x,y)^4} +o(r(x,y)^{-4}) \,\, \text{ as }\,\, r(x,y)\rightarrow \infty, 
\end{align}
for some constant $\rho_*>0$ where $r(x,y) =\sqrt{x^2+y^2}$.  The fact that this distribution is heavy-tailed with \emph{very} heavy tail $\rho_*/r(x,y)^4$ indicates the presence and strength of intermittent clustering.  Following the numerics and heuristics given in~\cite{bec2007clustering}, in the case of rough flows it is conjectured that the $X$-marginal $\tilde{\pi}$ of the invariant probability measure $\pi$ (supposing its existence) satisfies
\begin{align}
\label{eqn:tpi}
\tilde{\pi}(\d x) \propto |x|^{-2/h-1} \, \d x \,\, \text{ as } \,\, x\rightarrow -\infty.  
\end{align}   
Note that as $h>0$ becomes small in~\eqref{eqn:tpi}, the tail of the measure $\tilde{\pi}$ becomes \emph{less} heavy-tailed, indicating that intermittent clustering is weakened as the flow becomes rougher.  Also note that, after formally integrating out $y$ in~\eqref{eqn:asyxy}, we see that the formula~\eqref{eqn:tpi} is in agreement with~\eqref{eqn:asyxy} when $h=1$.  A partial resolution of this conjecture is one of the main results of this work; that is, we will show that for all $h>0$ small enough, $\pi$ exists, is exponentially attractive in total variation and 
\begin{align}
\label{eqn:mominvm}
\int_{\rbb^2\times (0,1]}\close r(x,y)^\lambda \, \pi( \d x, \d y, \d z) < \infty \,\, \text{ for all } \,\, 0 < \lambda < 2/h,
\end{align}     
cf. Theorem~\ref{thm:geometric.ergodicity} below.

While the main results in this paper represent progress in understanding the pairwise statistics of heavy particles advected by rough flows, it is important to highlight where our methods fall short in resolving similar conjectures for general H\"{o}lder exponents $h\in (0,1)$.  Different from previous works in similar settings, unique challenges are faced in the rough regime because of the additional $Z$ coordinate in equation~\eqref{eqn:XYZ:no.reflection} and its associated dynamics near $Z=0$ and $Z=1$.  While the reflection term at $Z=1$ in the equation presents its own set of difficulties, the dynamics near $Z=0$ is especially important from a physical perspective as it corresponds to the regime where $|\Rbf|\ll L$.  However, understanding this dynamics is nontrivial due to its singular behavior in the equation.  This is especially true in the region where $|(X,Y)|\sim Z^{1/3}$ and $Z>0$ is small, where almost all terms in equation~\eqref{eqn:XYZ:no.reflection} are in the same dominant balance.  In particular, when analyzing the equation in this regime, very few terms are negligible and thus the resulting dynamics is complicated.  This is where our restriction on $h>0$ is imposed.  The restriction that $h>0$ is sufficiently small allows us to construct the relevant Lyapunov function using various solutions to Laplace's equation, as the noise coefficients in~\eqref{eqn:XYZ:no.reflection} are approximately the same for $h>0$ small.   Away from this restriction, the construction in this region fails.  Resolving this issue remains an open problem, but doing so would ultimately yield the same conclusions in this paper for general $h\in (0,1)$.

\subsection{Remarks on reflective SDEs}
The area of reflective SDEs is arguably less traversed than usual It\^{o} diffusions.  However, we need to make use of such equations in this paper to deal with the reflective boundary condition at $|\Rbf|=L$ or, as in the case of~\eqref{eqn:XYZ:no.reflection}, at $Z=1$.     

In general, a Skorokhod SDE with a normal reflective boundary in a smooth domain $\domain\subset \rbb^d$ is an equation of the form
\begin{equation} \label{eqn:Skorokhod.eq}
\d\xbf_t = f(\xbf_t)\d t+\mathrm{A}(\xbf_t)\d W_t+g(\xbf_t)\d K_t,
\end{equation}
where for $\xbf\in\partial\domain$, $g(\xbf)$ is a vector field pointing inward towards $\domain$ and $K_t$ is a monotone increasing function that only increases when $\xbf_t\in\partial\domain$. Intuitively, $K_t$ is a control mechanism that is only activated when $\xbf_t$ touches the boundary, which in particular prevents the process from escaping the domain $\overline{\domain}$. For a mathematical introduction to SDEs of Skorokhod type, we refer the reader to \cite{andres2009diffusion,pilipenko2014introduction}.

SDEs with reflection arise naturally in a number of applied settings such as in queuing theory \cite{anantharam1993optimal,
dupuis1998skorokhod,dupuis2000multiclass,mandelbaum1995strong} and in optimal control of fluid systems \cite{vcudina2011asymptotically,menaldi1989optimal}. Historically, the Skorokhod problem in a domain was studied as early as in the works of \cite{skorokhod1961stochastic,skorokhod1962stochastic,wentzel1959boundary}.  Since then, various types of domains have been investigated, e.g., half space \cite{chaleyat1980reflexion,el1978probleme,mckean1963skorohod}, convex sets \cite{dupuis1999convex,lions1981construction,
tanaka2002stochastic}, smooth domains \cite{stroock1971diffusion}, nonsmooth domains satisfying a uniform exterior sphere condition \cite{costantini1992skorohod,lions1984stochastic} or more general conditions \cite{dupuis1993sdes}. A critical object of Skorokhod theory is reflected Brownian motion (RBM). Although this process is perhaps dynamically simple on the interior of the domain, the additional boundary effects for RBM can lead to a number of difficulties in its analysis depending on the structure of the boundary. Nevertheless, Lyapunov-type methods have been previously used to establish ergodic properties for RBM~\cite{atar2001positive,
budhiraja1999simple,chen1996sufficient,dupuis1994lyapunov}. Furthermore, characterizations of such stationary distributions were given in~\cite{harrison1987brownian,
kang2014characterization,sarantsev2017reflected}. In our situation~\eqref{eqn:xyz}, the dynamics itself is complicated but the boundary effects are relatively simple. However, because of the reflecting boundary, one has to modify the typical paths used to establish large-time properties of the system.  Specifically, in order to exhibit a suitable Lyapunov function for the Markov semigroup, one has to control both the \emph{interior generator}, corresponding to the dynamics on the interior of the domain $\domain$, as well as a boundary operator. Additionally, in order to verify the minorization condition on a sufficiently large compact set in $\domain$, obtaining existence and regularity of probability densities, as well as support properties of the corresponding transitions, is not obvious because of the boundary and our interior operator is weakly hypoelliptic. We work around these issues using hypoellipticity on the interior of the domain and establishing a support theorem specific to the system~\eqref{eqn:xyz}.    

\subsection{Organization of paper}
The rest of the manuscript is organized as follows. In Section \ref{sec:mainresult}, we introduce the Skorokhod SDE (cf. \eqref{eqn:xyz}) studied throughout the paper. There, we also state our main results, especially Theorem~\ref{thm:geometric.ergodicity}, and outline how we plan to prove them in later sections. Section~\ref{sec:intuition} provides some heuristic arguments used to build intuition about the dynamics~\eqref{eqn:xyz}. In Section~\ref{sec:minorization}, with the help of control theory, we establish the minorization condition, which is a key component in showing the system is geometrically ergodic. We then discuss and lay out the construction of our Lyapunov functions in Section~\ref{sec:Lyapunov}.  The existence of these functions allows us to both prove that the first return time to the \emph{center of space} has exponential moments and that the invariant probability measure has certain moments as in~\eqref{eqn:mominvm}. The full proof of all Lyapunov properties will be detailed later in Section~\ref{sec:Lyapunov:proof} and Section~\ref{sec:flux}. The proof of the main result, Theorem~\ref{thm:geometric.ergodicity}, will be presented in Section~\ref{sec:proof-thm-ergodicity}.

\section{Notation, Mathematical Setting and Main Results} \label{sec:mainresult}
Throughout, we let $(\Omega, \mathcal{F}, (\mathcal{F}_t)_{t\geq 0},  \P)$ be a filtered probability space satisfying the usual conditions \cite{karatzas2012brownian} and $(W^1_t,W^2_t)$ be a standard two-dimensional Brownian Motion on $(\Omega, \mathcal{F},\P)$ adapted to the filtration $(\mathcal{F}_t)_{t\geq 0}$.  We consider the following system of reflected stochastic differential equations of Skorokhod type \begin{equation} \label{eqn:xyz}
\begin{aligned}
\d x_t &= -\gamma x_t \, \d t -\frac{hx_t^2-y_t^2}{z_t} \, \d t+\sqrt{2\kappa_1}\, \d W^1_t,\\
\d y_t&= -\gamma y_t \, \d t -(1+h)\frac{x_t y_t}{z_t} \, \d t+\sqrt{2\kappa_2} \, \d W^2_t,\\
\d z_t&=(1-h)x_t\, \d t-\d k_t.
\end{aligned}
\end{equation}
The process $\xbf_t=(x_t, y_t, z_t)$ above evolves on the domain $\domain=\rbb\times\rbb\times(0,1]$ and the constant parameters $\gamma$, $h$ and $\kappa_i$ satisfy
$$\gamma>0,\quad h\in (0,1),\quad\text{and}\quad \kappa_i>0,\,i=1,2.$$   
The $(\mathcal{F}_t)-$adapted process $k_t$ is assumed to satisfy the following:
\begin{itemize}
\item $k_0=0$ and $k_t$ is continuous, non-decreasing;
\item $k_t$ only increases when $z_t=1$; that is, 
\begin{equation} \label{cond:k_t}
k_t< \infty,\qquad
\d k_t\geq 0,\quad\text{and}\quad   \int_0^\infty \mathbf{1} \{z_t<1\} \, \d k_t=0.
\end{equation}
\end{itemize}
In short, when on the interior of the domain $\domain$, the process $\xbf_t=(x_t, y_t, z_t)$ behaves according to system~\eqref{eqn:xyz} without the reflection term $d k_t$.  When the process hits the boundary $\{ z=1\}$, it is then reflected downward in the $z$-direction via the local time term $\d k_t$ in order to stay in $\domain$.  Given this explanation, however, several important issues remain evident when making sense of relation~\eqref{eqn:xyz}, the first of which is its well-posedness.  

To this end, for $R>0$ let 
\begin{align}
\label{eqn:OR}
\domain_R= \{ (x,y,z) \in \domain \, : \, |(x,y)|\leq  R\, \text{ and }\, z\geq 1/R\},
\end{align}
and define stopping times 
\begin{align}  \label{eqn:tau_R}
\tau_R = \inf \{ t\geq 0 \, : \, \mathbf{x}_t \notin \domain_R\}\qquad \text{ and } \qquad \tau= \lim_{R\rightarrow \infty} \tau_R.
\end{align}
We call $\tau$ the \emph{time of explosion} for system~\eqref{eqn:xyz}. By now standard arguments~\cite{khasminskii2011stochastic, lions1984stochastic,tanaka2002stochastic}, a pair of continuous $\mathcal{F}_t-$adapted processes $(\xbf_t,k_t)$ satisfying ~\eqref{eqn:xyz} and~\eqref{cond:k_t} exists pathwise and is unique for all times $t\geq 0$ with $t< \tau$. This stems from the fact that up to time $\tau_R$, the system \eqref{eqn:xyz} agrees with a system that has globally Lipschitz coefficients, allowing us to employ \cite[Theorem 4.1]{tanaka2002stochastic} to establish local well-posedness. Our first result shows that, in fact, such solutions $(\xbf_t,k_t)$ exist and are unique for all finite times $t\geq 0$ almost surely, provided the positive noise parameters $\kappa_i$, $i=1,2$, satisfy further conditions.   
\begin{proposition}
\label{prop:ne}
There exists a constant $0<c_*< 1$ sufficiently small such that for all $\kappa_1> 0$ and $\kappa_2>0$  satisfying 
\begin{equation} \label{cond:c_*}
\Big|\frac{\kappa_1}{\kappa_2}-1\Big|<c_*,
\end{equation}
the locally defined solution $(\xbf_t, k_t)$ is \textbf{nonexplosive}; that is, for all initial conditions $\mathbf{x}=(x, y, z) \in \domain$ we have 
\begin{align*}
\P_{\mathbf{x}}\{ \tau< \infty\} =0.  
\end{align*}
\end{proposition}  

The argument for Proposition \ref{prop:ne} will be briefly explained in Appendix~\ref{sec:wellposed}.

\begin{remark} \label{rem:wel-posed}
Proposition~\ref{prop:ne} will be a consequence of the proof of one of the main results below (Proposition~\ref{prop:Lyapunov:xyz}) but it should be noted that it by itself is nontrivial.  Indeed, this is because the solution in the absence of noise, i.e., $\kappa_1=\kappa_2=0$ above, when started on the domain $\{ x< -K \} \cap \{ y=0\}$ explodes in finite time for $K>0$ sufficiently large.  This can be shown by using a comparison argument after noting that, with the reflection term, $z_t \leq 1$, and thus any such solution with these initial conditions has 
\begin{align*}
\dot{x}_t \leq - x_t - h x^2_t.  
\end{align*}  
As a consequence, the random perturbation makes large excursions when it is sufficiently close to the region $\{x<-K\}\cap \{y=0\}$.  
\end{remark}

By the uniqueness established in Proposition~\ref{prop:ne}, using a similar argument as in the case of SDEs without reflection \cite{albeverio2008spde,oksendal2003stochastic}, it can be shown that $\xbf_t$ and $(\xbf_t,k_t)$ are both Markov.  See, for example, \cite[Theorem 1.2.2]{
pilipenko2014introduction}. We note however that in general, $k_t$ itself is not Markov.  See \cite[Section 2]{pilipenko2014introduction} for a more detailed discussion. We can thus introduce the Markov transitions of $\xbf_t$ by
\begin{align*}
\PP_t(\xbf,A):=\P_{\xbf}(\xbf_t\in A),
\end{align*}
defined for times $t\geq 0$, initial states $\mathbf{x} \in \domain$ and Borel sets $A\subseteq \domain$. Letting $\B_b(\domain)$ denote the set of bounded Borel measurable functions $f:\domain \rightarrow \rbb$, the associated Markov semigroup $\PP_t:\B_b(\domain)\to\B_b(\domain)$ is defined and denoted by
\begin{align*}
\PP_t f(\xbf)=\E_{\xbf}[f(\xbf_t)], \,\, f\in \B_b(\domain).
\end{align*}
A probability measure $\mu$ on Borel subsets of $\domain$ is called an {\bf invariant probability measure} for the semigroup $\PP_t$ if for every $f\in \B_b(\domain)$
\begin{align*}
\int_{\domain}\PP_t f(\xbf)\mu(\d\xbf)=\int_\domain f(\xbf)\mu(\d\xbf).
\end{align*}

We next turn to the issue of large-time properties of equation~\eqref{eqn:xyz}. For a measurable function $\Psi:\domain\to(0,\infty)$, we introduce the following weighted supremum norm
\begin{align*}
\|\f\|_\Psi:=\sup_{\xbf\in\domain}\frac{|\f(\xbf)|}{1+\Psi(\xbf)}.
\end{align*}
With this norm, we can now state the main result of the paper:
\begin{theorem} \label{thm:geometric.ergodicity}
Under the hypotheses of Proposition~\ref{prop:ne}, we have the following results:
\begin{itemize}
\item[\emph{(a)}]  $\PP_t$ admits a unique invariant probability measure $\pi$.
\item[\emph{(b)}]  There exist a function $\Psi\in C(\domain; [1, \infty))$ and constants $C_*>0$, $\rho_*\in(0,1)$ such that the following estimate holds
\begin{align*}
\|\PP_tf-\pi(f)\|_{\Psi}\leq C_* \rho_*^{t}\|f-\pi(f)\|_{\Psi},
\end{align*}  
for all times $t\geq 0$ and for every measurable $f:\domain\rightarrow \rbb$ with $\|f\|_{\Psi}<\infty$.  In the above, $\pi(f):= \int_\domain f \, \emph{d}\pi$.  
\item[\emph{(c)}]  Letting $r(x,y) =\sqrt{x^2+y^2}$, the $(x,y)$-marginal of $\pi$, denoted by $\bar{\pi}$, satisfies 
\begin{align} \label{ineq:int.r^(p+1).pi(dx)}
\int_{\RR^2}r(x,y)^\lambda \, \bar{\pi}( \emph{d}x, \emph{d}y) < \infty,
\end{align}
for all $0<\lambda<\tfrac{2}{h}$. 
\end{itemize}
\end{theorem}

\begin{remark} \label{rem:pi}

We remark that the bound $0<\lambda<\tfrac{2}{h}$ in Theorem~\ref{thm:geometric.ergodicity} (c) is conjectured to be optimal. This stems from computer simulations supported by heuristics that under the particular instance $\kappa_1/\kappa_2 = 1+ 2h$ with $h\in(0,1)$ as in~\eqref{eqn:XYZ:no.reflection}, the $x$-marginal $\tilde{\pi}$ of $\pi$ is numerically shown to exhibit relation \eqref{eqn:tpi} \cite{bec2007clustering}. Translating~\eqref{ineq:int.r^(p+1).pi(dx)} to the physical model~\eqref{eqn:XYZ:no.reflection}, if~\eqref{ineq:int.r^(p+1).pi(dx)} holds true for some $\lambda >0$, then it should be the case that $\lambda < \tfrac{2}{h}$. As mentioned in the introduction, this scaling also agrees with the established scaling \eqref{eqn:asyxy} in the case of smooth flows~\cite{gawedzki2011ergodic}, i.e., when $h=1$.

\end{remark}

\begin{remark}
As mentioned in the introduction, ergodic properties of reflective SDEs with non-degenerate noise have been established in a number of settings~\cite{budhiraja1999simple,cattiaux1992stochastic,
cattiaux2017hitting,cattiaux2013poincare,
dieker2013positive,
dupuis1994lyapunov}. On the other hand, to the best of the authors' knowledge, analogous results for reflective SDEs with degenerate noise appear to be rarer. See, however, the works \cite{dieker2013positive,lipshutz2021sensitivity}, which treat cases of reflective SDEs with degenerate noise.  While \cite{lipshutz2021sensitivity} treats reflected Brownian motion in a convex polyhedral cone,
reference \cite{dieker2013positive} discusses ``piecewise" Ornstein-Uhlenbeck processes using quadratic Lyapunov structure. For the system~\eqref{eqn:xyz}, there is no obvious starting guess for a global Lyapunov function  
because of the underlying instability discussed in Remark \ref{rem:wel-posed}. The stability in our system is dictated by a complicated dynamical structure which combines fast exits from the unstable region with subsequent transportations to the stable region, from which the Markov 
process returns to the ``center" of space.
\end{remark}

The proof of Proposition~\ref{prop:ne} and Theorem~\ref{thm:geometric.ergodicity} (a)-(b) will follow by establishing the existence of an appropriate Lyapunov functional $\Psi$ as in~Proposition~\ref{prop:Lyapunov:xyz} below, showing that the dynamics makes frequent returns to the \emph{center} of space, and by proving that initial conditions in this center minorize, Proposition~\ref{prop:minorization} below, allowing one to couple a solution pair once both have arrived in the center~\cite{hairer2011yet, mattingly2002ergodicity, meyn1992stochastic}.

\begin{proposition} \label{prop:Lyapunov:xyz}
Under the same hypothesis as in Proposition~\ref{prop:ne}, there exists a function $\Psi\in C(\domain; [1,\infty))$, which we call throughout a \textbf{Lyapunov function}, satisfying the following properties:
\begin{itemize}
\item[\emph{(a)}]  For $R>0$, let 
\begin{align*}\Psi_R= \inf_{ \mathcal{O}\setminus\domain_R} \Psi(x,y,z),
\end{align*}
where $\domain_R$ is given as in~\eqref{eqn:OR}. Then $\Psi_R\rightarrow \infty$ as $R\rightarrow \infty$.  
\item[\emph{(b)}]  For every $\xbf \in \domain$ and every $t\geq 0$ there exists a constant $c(x,t)>0$ such that 
\begin{align*}
\E_\xbf \Psi(\xbf_{t\wedge \tau_R}) \leq c(x,t),
\end{align*}
for every $R>0$.  
Consequently, $\xbf_t$ is non-explosive and the Markov semigroup $\PP_t$ is defined.  
\item[\emph{(c)}]There exist constants $\varepsilon_*\in(0,1),\, D_*> 0$ such that for all $\xbf\in \domain$,
\begin{align*}
\PP_t\Psi(\xbf)\le \varepsilon_*^t \Psi(\xbf)+D_*.
\end{align*}
\end{itemize}

 \end{proposition}

Following~\cite{hairer2011yet}, we obtain Theorem~\ref{thm:geometric.ergodicity} parts (a) and (b) by combining the previous result with:

\begin{proposition}[Minorization]\label{prop:minorization} For $R$ sufficiently large, there exists a time $t_*>0$, a constant $c>0$, and a probability measure $\nu$ such that $\nu(\domain_R)=1$ and such that for every $\xbf\in \domain_R$ and any Borel set $A\subset \domain$
\begin{align*}
\PP_{t_*}(\xbf,A)\geq c\nu(A).
\end{align*}
\end{proposition}

The proof of Proposition~\ref{prop:minorization} will be given in Section~\ref{sec:minorization}.  The construction of the Lyapunov function will be carried out in Section~\ref{sec:Lyapunov}, Section~\ref{sec:Lyapunov:proof} and Section~\ref{sec:flux}.  We will use facts deduced in these sections to establish Proposition~\ref{prop:Lyapunov:xyz} in Section~\ref{sec:proof-thm-ergodicity}. Combining these results, we will be able to conclude the convergence rate to equilibrium, Theorem~\ref{thm:geometric.ergodicity} (b), whose proof is classical \cite{hairer2011yet} and thus is omitted. The remaining assertion concerning the moments of the marginal $\bar{\pi}$, Theorem~\ref{thm:geometric.ergodicity} part (c), will essentially follow from our Lyapunov construction. All of this will be explained in detail in the proofs of Theorem~\ref{thm:geometric.ergodicity} and Proposition~\ref{prop:Lyapunov:xyz} in Section~\ref{sec:proof-thm-ergodicity}.

 \section{Heuristics} \label{sec:intuition}
 
 Before diving into the precise details of the main results stated in Section~\ref{sec:mainresult}, in this section we build some intuition about the system~\eqref{eqn:xyz} using heuristics.  Doing this will furthermore allow us to make key observations and discuss the main difficulties faced when proving these results.

We first make a convenient change of coordinates of equation~\eqref{eqn:xyz}.  Let  
 \begin{align}
 \label{eqn:cocoord}
 u = x z^{-1/3}, \,\,\, v = y z^{-1/3}, \,\,\, \text{ and } \,\,\, z=z,
 \end{align}
 and note that, in this coordinate system, equation~\eqref{eqn:xyz} transforms as
 \begin{equation}\label{eqn:uvz}
  \begin{aligned}
  \d u_t  &= -  \gamma u_t \, \d t  -\frac{\alpha_h u^{2}_t - v_t^{2}}{z^{2/3}_t} \, \d t +\frac{\sqrt{2\kappa_1}}{z^{1/3}_t} \, \d W^1_t+\frac{u _t}{3z_t}\, \d k_t, \\
 \d v_t &= -  \gamma v_t \, \d t -  \frac{(\alpha_h +1) u_t v_t}{z^{2/3}_t}\, \d t + \frac{\sqrt{2\kappa_2}}{z^{1/3}_t}\, \d W^2_t  +\frac{v_t}{3z_t}\, \d k_t,\\
 \d z_t&= (1-h)  u_tz^{1/3}_t \, \d t- \d k_t   ,
\end{aligned} 
 \end{equation}
where 
\begin{align}
\label{def:alpha}
\alpha_h = \tfrac{1}{3}+ \tfrac{2}{3}h.
\end{align}

The motivation behind this substitution is as follows.  If either $|(u,v)|$ is large or $z$ is small and we ignore boundary effects, then under a random time change the system~\eqref{eqn:uvz} asymptotically decouples the process $(u_t,v_t)$ from $z_t$.  To be more precise, note that in this new coordinate system in the absence of boundary effects, i.e., setting $k_t\equiv 0$ in~\eqref{eqn:uvz}, the generator has the form
\begin{align}
\nonumber \L&= - \gamma u \partial_{u} -\gamma v \partial_{v} - \frac{\alpha_h u^{2} -v^{2}}{z^{2/3}}\partial_{u} - (\alpha_h+1) \frac{u v}{z^{2/3}}\partial_{v} +(1-h) u z^{1/3} \partial_z+ \frac{\kappa_1}{z^{2/3}}\partial^2_{u} + \frac{\kappa_2}{z^{2/3}}\partial_{v}^2\\
 &=: z^{-2/3} \M  . \label{def:M}
\end{align}
Assuming that the terms $-\gamma z^{2/3} u \partial_u - \gamma v z^{2/3} \partial_v$ in the operator $\M$ are negligible when $|(u,v)|\gg 1$ or $z\approx 0$ (see Section \ref{sec:Lyapunov:large(u,v)} below), then for such values
\begin{align}
\label{eqn:Mapprox}
\M \approx - (\alpha_h u^2 -v^2) \partial_u - (\alpha_h + 1) uv \partial_v + \kappa_1 \partial_u^2 + \kappa_2 \partial_v^2 + (1-h) u z \partial_z=:\N .
\end{align}
Note that the $(u,v)$ dynamics defined by $\N$ moves independent of $z$, and that the corresponding $z$ dynamics depends on the $(u,v)$ dynamics in a relatively simple manner.  Moreover, the $(u,v)$ dynamics is similar to those studied previously in~\cite{athreya2012propagating,
birrell2012transition,gawedzki2011ergodic}.  Thus in order to understand the behavior of the system~\eqref{eqn:uvz} in the absence of boundary effects, one needs to understand this $(u,v)$ dynamics and how this determines the behavior of the $z$ process.    

Continuing with this informal discussion, we introduce the following SDE
\begin{equation} \label{eqn:UVZ}
\begin{aligned}
\d U_t &= -(\alpha_h U^2_t- V^2_t) \, \d t + \sqrt{2\kappa_1} \, \d W_t^1,\\
\d V_t &=- (\alpha_h+1)U_t V_t \, \d t + \sqrt{2\kappa_2} \, \d W_t^2 ,  \\
\d Z_t &= (1-h) U_t Z_t \, \d t ,
\end{aligned}
\end{equation}
and note that this process has generator $\N$ as in~\eqref{eqn:Mapprox}. Also, by setting 
\begin{align} \label{def:A}
 \A=  - (\alpha_h u^2 -v^2) \partial_u - (\alpha_h + 1) uv \partial_v + \kappa_1 \partial_u^2 + \kappa_2 \partial_v^2 ,\end{align}
we observe that $\A$ is the generator for the process $(U_t,V_t)$ which is decoupled from $Z_t$, and that the $Z$-process satisfies
$$Z_t=Z_0\exp\Big\{(1-h)\int_0^t U_s\, \d s\Big\}.$$
Now when $|(U_t,V_t)|$ is large, the system~\eqref{eqn:UVZ} is dominated by the dynamics along the generator $\A$.  However, when $Z_t$ is near zero, whether $Z_t$ decreases or increases on average is dictated by the sign of the time-average of the $U$-process. To keep the dynamics from hitting $Z=0$ in finite time, it is necessary that the process $(U_t,V_t,Z_t)$ should spend most of the time in the region where $U>0$ rather than where $U\leq 0$. Therefore, if we are hoping for well-posedness and ergodicity of the original process~\eqref{eqn:uvz}, i.e., for the $Z_t$ process to eventually increase, it should be the case that, for all $t>0$ large enough, 
$$\mu(U):=\int_{\rbb^2}\close U\,\mu(\d U,\d V )\approx\frac{1}{t}\int_0^t U_s\, \d s>0, \qquad t\gg 0,$$
where $\mu$ denotes the ergodic invariant probability measure for the generator $\A$.  Note that such a measure $\mu$ exists and is unique by the main result in~\cite{birrell2012transition}. 

To see intuitively why one should expect 
$\mu(U)> 0$, we provide some numerics in Figure~\ref{figure:1} indicating that this should be the
case for various values of the noise parameters $\kappa_1=1>0$, $\kappa_2=1+2h$ and H\"{o}lder exponent $h>0$. 
\begin{figure}[H]
\begin{tikzpicture}

            \node [inner sep=0pt,above right] 
                {\includegraphics[width=10cm]{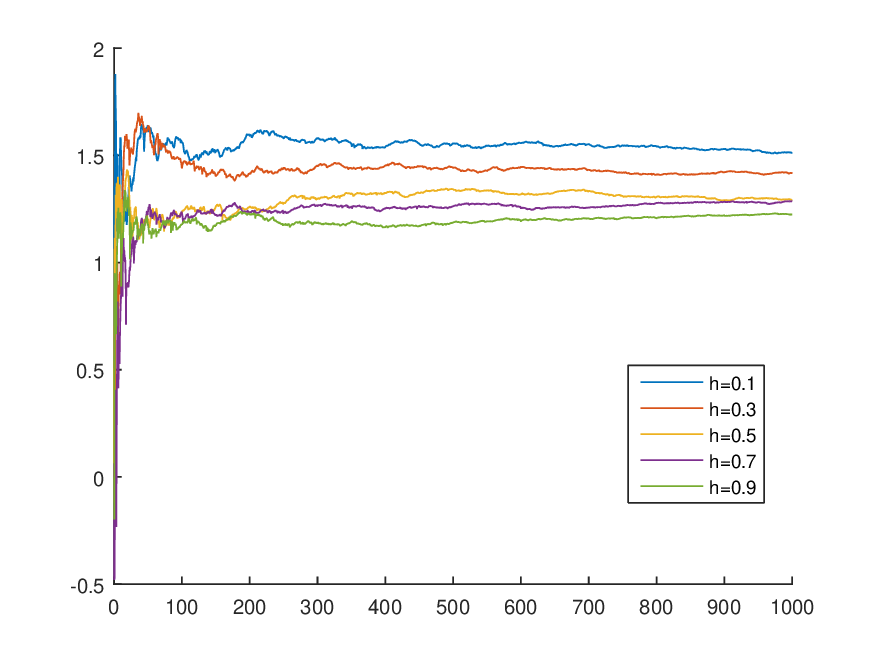}};
           \draw[->] (9.05,0.828)--(9.3,0.828) node[right]{$t$};
          \draw[->] (1.312,6.95)--(1.312,7.2)  node[above]{$\frac{1}{t}\int_0^tU_s\d s$} ;
        \end{tikzpicture}

\caption{Large-time average $\frac{1}{t}\int_0^tU_s\d s$ with initial condition $U_0=V_0=0$ and noise parameters $\kappa_1=1$, $\kappa_2=1+2h$.}
          \label{figure:1}
\end{figure}

Even though the numerics clearly indicate the positive sign of $\mu(U)$, establishing this
fact rigorously and analyzing its consequences are more involved. The intuitive reason for
this is that there is a region (which cannot be avoided) in which all of the terms in the operator $\A$ are on the same order. Thus, different from the construction of the Lyapunov function for large
radial values as in Section~\ref{sec:outer-region}, we cannot neglect any term in $\A$ to produce the right type of test function.
Additionally, this operator is sufficiently complicated in its total form so as to not yield obvious explicit solutions to these problems. Ultimately, under the assumption $\kappa_1\approx\kappa_2$, cf.~\eqref{cond:c_*}, we will be able to prove $\mu(U)>0$, thereby establishing ergodicity of the original process~\eqref{eqn:uvz} and our main results. All of the details will be carefully described in Section~\ref{sec:inner-region}. 

\section{Minorization} \label{sec:minorization}

Throughout the rest of the paper,  $c,D$ denote positive constants that might change from line to line. The important parameters that they depend on are indicated in parenthesis, e.g., $c(\kappa_1)$ depends on $\kappa_1$.

In this section, we prove Proposition~\ref{prop:minorization}. The argument giving this result will employ a series of lemmata below, whose proofs are deferred to the end of the section. First, letting $\mathring{\domain}$ denote the interior of $\domain$ and, offering a slight abuse of notation, lettting
\begin{align}
\label{eqn:Lxyz}
\L= - \gamma x \partial_x - \gamma y \partial_y - \frac{h x^2-y^2}{z}\partial_x -(1+h) \frac{xy}{z}\partial_y + (1-h) x \partial_z + \kappa_1 \partial_x^2 + \kappa_2 \partial_y^2,
\end{align}
we show that the differential operators $\partial_t \pm \L$, $\partial_t \pm \L^*$, $\L$ and $\L^*$, where $\L^*$ denotes the formal $L^2(\RR^3, \d x \d y \d z)$ adjoint, are all hypoelliptic on the respective domains $(0, \infty) \times\mathring{\domain}$, $(0,\infty) \times \mathring{\domain}$, $\mathring{\domain}$ and $\mathring{\domain}$.  Applying H\"{o}rmander's hypoellipticity theorem~\cite{hormander1967hypoelliptic}, this will allow us to determine the existence and smoothness of all probability densities of interest when restricted to the interior $\mathring{\domain}$ of the full domain $\domain$.  In particular, with what follows we will not be able to conclude the existence nor any regularity properties of the densities on the boundary $\{ z=1\}$.  One might be able to do so if the system were elliptic or strongly hypoelliptic \cite{cattiaux1992stochastic}, but the problem is made more difficult because the diffusion operator in the absence of reflection is weakly hypoelliptic (i.e., one has to make use of the drift vector field to generate all directions in the tangent space).  Regardless, however, we will not need such properties on the boundary to conclude minorization.

\begin{lemma} \label{lem:Hormander} For $t>0$ and $\xbf\in \mathring{\domain}$, the measure $\PP_t(\xbf, \, \cdot \,)$ when restricted to Borel subsets of $\mathring{\domain}$ is absolutely continuous with respect to Lebesgue measure on $\mathring{\domain}$.  We denote the density of this measure by $\ybf\mapsto p(t, \xbf, \ybf):\mathring{\domain}\rightarrow [0, \infty)$.  Moreover, the mapping $(t, \xbf, \ybf) \mapsto p(t, \xbf, \ybf): (0, \infty)\times \mathring{\domain}\times \mathring{\domain}\rightarrow [0, \infty)$ is $C^\infty$.    
\end{lemma}

With the smoothness of the interior density $p(t,\xbf,\ybf)$ obtained in the previous lemma, we next show that for all times $T>0$ there exists a nonempty open subset of $\mathring{\domain}\times \mathring{\domain}$ on which $(\xbf, \ybf)\mapsto p(T,\xbf,\ybf)$ is uniformly positive. This crucial estimate will be employed in the proof of Proposition~\ref{prop:minorization} below. 

\begin{lemma} \label{lem:inf.p(x,y)>0}  Let $\xbf_*=(0,0,1/2)$ and $B(\ybf,\varepsilon)$ be the open ball in $\rbb^3$ centered at $\ybf$ with radius $\varepsilon$.  Then for any time $T>0$, there exist $\ybf_*\in \mathring{\domain}$ and $\varepsilon_1,\,\varepsilon_2>0$ such that  $B(\xbf_*,\varepsilon_1)\times B(\ybf_*,\varepsilon_2)\subseteq  \mathring{\domain}\times  \mathring{\domain}$ and 
\begin{align*}
\inf_{\xbf\in B(\xbf_*,\varepsilon_1),\ybf\in B(\ybf_*,\varepsilon_2)}\close\close p(T,\xbf,\ybf)>0.
\end{align*}
\end{lemma}

Having obtained the crucial bound above, in order to finish showing Proposition~\ref{prop:minorization}, we also need to show that the probability of starting the process from arbitrary $\xbf\in \domain_R$ and returning to $B(\xbf_*,\varepsilon_1)$ is uniformly bounded from below. Together with Lemma~\ref{lem:inf.p(x,y)>0}, the result of Lemma~\ref{lem:inf.P(x,C.epsilon)>0} will be employed to prove Proposition~\ref{prop:minorization}. 
\begin{lemma}\label{lem:inf.P(x,C.epsilon)>0} Let $\domain_R$ be the cylinder as in~\eqref{eqn:OR}. For all $R>0$ sufficiently large, there exists a time $T>0$ such that 
\begin{align}\label{ineq:inf.x.in.C_R}
\inf_{\xbf\in\domain_R}\PP_T\big(\xbf, B(\xbf_*,\varepsilon_1)\big)>0,
\end{align} 
where $\xbf_*,\,\varepsilon_1$ are as in Lemma~\ref{lem:inf.p(x,y)>0}.
\end{lemma}

The arguments giving the above lemmata are deferred to the end of this section. Having laid out the necessary ingredients, we are now in a position to prove Proposition~\ref{prop:minorization}. The proof follows along the lines of the proof of~\cite[Lemma 2.3]{mattingly2002ergodicity}.  See also \cite{mattingly2012geometric}. Since the proof is short, we include it here for the sake of completeness.
\begin{proof}[Proof of Proposition~\ref{prop:minorization}] Let $\xbf_*,\, \ybf_*$ and $B(\xbf_*,\varepsilon_1), B(\ybf_*,\varepsilon_2)$ be as in Lemma~\ref{lem:inf.P(x,C.epsilon)>0} and $p(T,\xbf,\ybf)$ be the restricted transition density as in Lemma~\ref{lem:inf.p(x,y)>0}.  Define the following probability measure $\nu$ on Borel subsets of $\domain$ by 
\begin{align*}
\nu(A)=\frac{|A\cap B(\ybf_*,\varepsilon_2)|}{|B(\ybf_*,\varepsilon_2)|} , 
\end{align*}
where $|\cdot|$ denotes Lebesgue measure on $\RR^3$.  
Consider $\xbf\in \domain_R$ where $\domain_R$ is the compact set as in Lemma~\ref{lem:inf.P(x,C.epsilon)>0}.  By the Markov property, Lemma~\ref{lem:inf.p(x,y)>0} and Lemma~\ref{lem:inf.P(x,C.epsilon)>0}, we have \begin{align*}
\PP_{2T}(\xbf,A)&= \int_{A}\int_{\domain}\PP_T(\xbf,d\ybf)\PP_T(\ybf,d\zbf)\\
&\geq \int_{A\cap B(\ybf_*,\varepsilon_2)}\int_{B(\xbf_*,\varepsilon_1)}\close\close\PP_T(\xbf,d\ybf)\PP_T(\ybf,d\zbf)\\
&= \int_{A\cap B(\ybf_*,\varepsilon_2)}\int_{B(\xbf_*,\varepsilon_1)}\close\close\PP_T(\xbf,d\ybf)p(T,\ybf,\zbf)d\zbf\\
&\geq |A\cap B(\ybf_*,\varepsilon_2)| \PP_T(\xbf,B(\xbf_*,\varepsilon_1))\inf_{\xbf\in B(\xbf_*,\varepsilon_1), \ybf\in B(\ybf_*,\varepsilon_2)}\close\close p(T,\xbf,\ybf)\\
&\geq |A\cap B(\ybf_*,\varepsilon_2)| \inf_{\xbf\in\domain_R}\PP_T\big(\xbf, B(\xbf_*,\varepsilon_1)\big)\inf_{\xbf\in B(\xbf_*,\varepsilon_1), \ybf\in B(\ybf_*,\varepsilon_2)}\close\close p(T,\xbf,\ybf)\\
&= c\,\nu(A)|B(\ybf_*,\varepsilon_2)|,
\end{align*}
where the last implication follows from the fact that the two infima above are positive.
\end{proof}

We now turn to the proofs of Lemma ~\ref{lem:Hormander}, Lemma~\ref{lem:inf.p(x,y)>0} and Lemma~\ref{lem:inf.P(x,C.epsilon)>0}. First, we make use of the theory of hypoellipticity to prove that $p(t,\xbf,\ybf)$ exists and is smooth on $(0,\infty)\times \mathring{\domain}\times\mathring{\domain}$.
\begin{proof}[Proof of Lemma~\ref{lem:Hormander}]
Consider the vector fields
\begin{align*}
X_0=-\gamma x\partial_x-\frac{hx^2-y^2}{z}\partial_x&- \gamma y\partial_y -(1+h)\frac{xy}{z}\partial_y+(1-h)x\partial_z,\\
X_1=\sqrt{2\kappa_1}\partial_x,&\qquad X_2=\sqrt{2\kappa_2}\partial_y.
\end{align*}
By Corollary~7.2 of~\cite{bellet2006ergodic}, it suffices to show that the list of vector fields
\begin{align*}
\{X_i \}_{i=1}^2, \, \{ [X_i, X_j]\}_{i,j=0}^2, \, \{ [[X_i, X_j], X_k]\}_{i,j,k=0}^2, \ldots 
\end{align*}
has rank $3$ at every point $x\in \mathring{\domain}$.  In the above, $[A,B]$ denotes the commutator of vector fields $A$ and $B$.  Note \begin{align*}
[X_1,X_0]= \big(- \gamma -\frac{2hx}{z} \big)\partial_x-(1+h)\frac{y}{z}\partial_y+(1-h)\partial_z.
\end{align*}
Hence for $\xbf\in\mathring{\domain}$, $\{X_1(\xbf),X_2 (\xbf), [X_1, X_0](\xbf)\}$ spans $\rbb^3$, completing the proof.
\end{proof}

Next, we turn to the uniform bounds for the density in Lemma~\ref{lem:inf.p(x,y)>0} and the probabilities in Lemma \ref{lem:inf.P(x,C.epsilon)>0} whose arguments will make use of the auxiliary result in Lemma \ref{lem:support.theorem} below. In turn, Lemma \ref{lem:support.theorem} is a consequence of \cite[Theorem 1]{doss1982support}, which is an analogue of the Support Theorems \cite{stroock1972degenerate, stroock1972support} adapted to normally reflected diffusions. See also \cite{hopfner2016ergodicity,
ren2016approximate,wu2018limit} for related results on the support of SDEs with reflections.

\begin{lemma} \label{lem:support.theorem}  Consider the control problem
\begin{equation} \label{eqn:control.problem}
\begin{aligned}
\emph{d} x_t &= -\gamma x_t\, \emph{d} t-\frac{hx_t^2-y_t^2}{z_t}\, \emph{d} t+\sqrt{2\kappa_1}\,\emph{d} U^1_t,\\
\emph{d} y_t&= -\gamma y_t\, \emph{d}t-(1+h)\frac{x_ty_t}{z_t}\, \emph{d} t+\sqrt{2\kappa_2}\,\emph{d} U^2_t,\\
\emph{d} z_t&=(1-h)x_t\,\emph{d} t-\emph{d}k_t,
\end{aligned}
\end{equation}
where $U_t=(U^1_t,U^2_t)\in C^1(\rbb)\times C^1(\rbb)$ is a control function with $U^i_0=0$, $i=1,2$. Suppose that $(\xbftd_t,\ktil_t)$ and $(\xbf_t,k_t)$ respectively solve ~\eqref{eqn:control.problem} and \eqref{eqn:xyz} with initial condition $\xbftd_0=\xbftd$ and $\xbf_0=\xbf$ on the deterministic time interval $[0, T]$. Then, for all $\varepsilon>0$, there exists $\widetilde{\varepsilon}$ such that for all $\xbf\in B(\xbftd,\widetilde{\varepsilon}\,)$,
\begin{align*}
\P\bigg\{ \sup_{0\leq t\leq T}|\xbf_t-\xbftd_t|\leq \varepsilon\bigg\}>0.
\end{align*}

\end{lemma}
\begin{proof}
See \cite[Theorem 1]{doss1982support}.
\end{proof}

With Lemma~\ref{lem:support.theorem} in hand, we are now ready to give the proof of Lemma~\ref{lem:inf.p(x,y)>0}.

\begin{proof}[Proof of Lemma~\ref{lem:inf.p(x,y)>0}] Considering the deterministic control problem \eqref{eqn:control.problem},
we pick the constant processes 
\begin{align*}
\xtil_t=0,\qquad\ytil_t=0,\qquad\ztil_t=\tfrac{1}{2},\qquad \ktil=0,
\end{align*}
together with
\begin{equation*} 
U^1_t=0,\qquad U^2_t=0,
\end{equation*}
it is clear that with $U_t=(U^1_t,U^2_t)$ defined above, $(\xbftd_t,\ktil_t)$ solves~\eqref{eqn:control.problem}, and drives $\xbf_*=(0,0,1/2)$ at time 0 to the same $\xbf_*$ at time $T$. In view Lemma \ref{lem:support.theorem}, we obtain for all $\varepsilon>0$
\begin{align*}
\P\bigg\{\sup_{0\leq t\leq T}|\xbf_t-\xbftd_t|\leq\varepsilon\bigg\}>0.
\end{align*}
In the above, $\xbf_t$ is the solution of \eqref{eqn:xyz} with initial condition $\xbf_0=\xbf_*$. Together with Lemma~\ref{lem:Hormander}, for all $\varepsilon <1/2$ we have 
\begin{align*}
\PP_T(\xbf_*,B(\xbf_*,\varepsilon))=  \int_{B(\xbf_*,\varepsilon)}\hspace{-0.7cm} p(T,\xbf_*,\ybf)d\ybf\geq \P\bigg\{\sup_{0\leq t\leq T}|\xbf_t-\xbftd_t|\leq\varepsilon\bigg\}>0,
\end{align*}
implying there exists $\ybf_*\in B(\xbf_*,\varepsilon)$ such that $p(T,\xbf_*,\ybf_*)>0$. By continuity of $p(T,\cdot,\cdot)$, we therefore obtain the result. 
\end{proof}

Similar to the proof of Lemma~\ref{lem:inf.p(x,y)>0}, the proof of Lemma~\ref{lem:inf.P(x,C.epsilon)>0} will use a control argument which shows that we can drive our process from $\xbf_0 \in \domain_R$ to $\xbf_*$ at time $T$.
\begin{proof}[Proof of Lemma~\ref{lem:inf.P(x,C.epsilon)>0}] There are two cases to consider depending on how close $z$ is to the boundary $\{ z=1\}$.  

\emph{Case $1$}: $z\leq 3/4$.  Since $z<1$, we have by Lemma~\ref{lem:Hormander} 
\begin{align*}
\PP_T(\xbf,B(\xbf_*,\varepsilon_1))=\int_{B(\xbf_*,\varepsilon_1)}\hspace{-0.7cm}  p(T,\xbf,\ybf)d\ybf .  
\end{align*}
By continuity of the density $p(T,\cdot,\cdot)$ on $\mathring{\domain}\times\mathring{\domain}$, the mapping $\xbf\mapsto \PP_T(\xbf,B(\xbf_*,\varepsilon_1))$ restricted to $\mathring{\mathcal{O}}$ is continuous by the bounded convergence theorem. For a given time $T>0$, choose a deterministic  smooth function $F_t$ satisfying
\begin{equation} \label{eqn:int.x.dt:z<3/4}
F_0=0,\quad F'_0=(1-h)x,\quad F_T=\tfrac{1}{2}-z,\quad F'_T=0,\quad\text{and}\quad \forall t\in[0,T]: \, 1-z> F_t >-z,
\end{equation} 
and set 
\begin{equation} \label{eqn:choice.xt.zt.kt:z<3/4}
\xtil_t=\tfrac{1}{1-h} F'_t,\qquad \ztil_t=z+(1-h)\int_0^t\xtil_s\d s=z+F_t,\quad \ktil_t=0.
\end{equation}
Then, choose $\ytil_t$ arbitrarily such that
\begin{align} \label{eqn:choice.yt:z<3/4}
\ytil_0=y,\qquad\ytil_T=0,
\end{align}
We note that with this choice of $F_t$ in~\eqref{eqn:int.x.dt:z<3/4}, for $t\in[0,T]$, 
\begin{align*}
0<\ztil_t<1,\quad\text{and}\quad z_T=\tfrac{1}{2}.
\end{align*}
Observe that $(\xbftd_t,\ktil_t)$ defined in~\eqref{eqn:choice.xt.zt.kt:z<3/4} and~\eqref{eqn:choice.yt:z<3/4} solves the control problem ~\eqref{eqn:control.problem} with $U^{\xbf}_t=(U^1_t,U^2_t)$ given by
\begin{equation} \label{eqn:control.process:U1U2}
U^1_t=\frac{1}{\sqrt{2\kappa_1}}\int_0^t \xtil_s'+ \gamma\xtil_s+\frac{h\xtil_s^2-\ytil_s^2}{\ztil_s}\d s,\qquad U^2_t=\frac{1}{\sqrt{2\kappa_2}}\int_0^t \ytil_s'+\gamma \ytil_s+(1+h)\frac{\xtil_s\ytil_s}{\ztil_s}\d s,
\end{equation}
and drives the trajectory from $\xbf$ to $\xbf_*=(0,0,1/2)$ at time $T$. 
We invoke Lemma~\ref{lem:support.theorem} to infer that for every $\xbf\in \domain_R\cap \{z\leq 3/4\}$,
\begin{align*}
\PP_T(\xbf,B(\xbf_*,\varepsilon_1))\geq \P\bigg\{\sup_{0\leq t\leq T}|\xbf_t-\xbftd_t|\leq \varepsilon_1\bigg\}>0.
\end{align*}
By compactness of $\domain_R\cap \{z\leq 3/4\}$ and continuity of $ \PP_T(\xbf,B(\xbf_*,\varepsilon_1))$ on $\mathring{\domain}$ with respect to $\xbf $, we conclude that
\begin{align} \label{ineq:inf.x.in.C_R.and.z<3/4}
\inf_{\xbf\in\domain_R\cap\{z\leq 3/4\}}\PP_T\big(\xbf, B(\xbf_*,\varepsilon)\big)>0.
\end{align}

\emph{Case $2$:} $3/4\le z\leq 1$.  Here, we first aim to show that for every $\xbf\in \domain_R\cap\{z\ge3/4\}$, there exists $\varepsilon=\varepsilon(\xbf)>0$ such that
\begin{align} \label{ineq:inf_y.in.B(x).P(y,B(x_*))>0}
\inf_{\ybf\in B(\xbf,\varepsilon(\xbf))}\PP_T(\ybf,B(\xbf_*,\varepsilon_1))\geq c(\xbf)>0.
\end{align}
To do so, similar to case 1, we first choose a deterministic smooth function $G$ satisfying
\begin{equation} \label{eqn:int.x.dt:z>3/4}
G_0=0,\quad G'_0=(1-h)x,\quad G_T=-\frac{1}{4},\quad G'_T=0,\quad\max_{0\leq t\leq T} G_t=\frac{1}{4},\qquad \min_{0\leq t\leq T}G_t=- \frac{1}{4},
\end{equation} 
and set 
\begin{equation*} 
\xtil_t=\tfrac{1}{1-h} G'_t.
\end{equation*}
With the choice of $\xtil_t$ above, we define 
\begin{align}\label{eqn:control.z>3/4.ztilde}
\ztil_t=z+(1-h)\int_0^t\close\xtil_s\d s-\ktil_t=z+G_t-\ktil_t,
\end{align}
where $\ktil_t$, in view of \cite[Formula (1.2)]{pilipenko2014introduction}, is explicitly given by
\begin{equation}\label{eqn:control.z>3/4.ktilde}
\ktil_t=-\min_{0\leq s\leq t}\bigg\{\bigg(1-z-(1-h)\int_0^s\close\xtil_\ell\d \ell \bigg) \wedge 0\bigg\}=-\min_{0\leq s\leq t}\big\{\big(1-z-G_s \big) \wedge 0\big\}.
\end{equation}
We finally choose $\ytil_t$ arbitrarily satisfying $\ytil_0=y,\,\ytil_T=0$ and the control $U^{\xbf}_t=(U^1_t, U^2_T)$ as in~\eqref{eqn:control.process:U1U2}.
To show that $(\xbftd_t,\ktil_t)$ satisfies~\eqref{eqn:control.problem} and drives $\xbf$ to $\xbf_*$ at time $T$, it suffices to prove that $\ztil_t>0$ for all time $t\in[0,T]$ and $\ztil_T=1/2$. To this end, we now claim that 
\begin{align}\label{ineq:min.ztilde=ztilde_T=1/2}
\min_{0\leq t\leq T}\ztil_t=\ztil_T=\tfrac{1}{2}. 
\end{align}
To show the minimum part, we recast~\eqref{eqn:control.z>3/4.ztilde} together with~\eqref{eqn:control.z>3/4.ktilde} as
\begin{align*}
\ztil_t=z+G_t+\min_{0\leq s\leq t}\{(1-z-G_s) \wedge 0\}\geq\tfrac{1}{2},
\end{align*}
which can be rewritten as
\begin{align*}
\min_{0\leq s\leq t}\{(1-z-G_s) \wedge 0\}\geq (1-z-G_t)-\tfrac{1}{2},
\end{align*}
and equivalently,
\begin{align}\label{ineq:ztilde>1/2}
\min_{0\leq s\leq t}\{(1-G_s) \wedge z\}\geq (1-G_t)-\tfrac{1}{2}.
\end{align}
By the choice of $G_t$ in~\eqref{eqn:int.x.dt:z>3/4}, namely $\min_{[0,T]}G_t=-1/4$ and $\max_{[0,T]}G_t=1/4$, on the one hand
\begin{align*}
(1-G_t)-\tfrac{1}{2}\leq \tfrac{3}{4}.
\end{align*}
On the other hand, since $z\geq\tfrac{3}{4}$ and for $s\in[0,T]$, $1-G_s\in[3/4,5/4]$ and
\begin{align*}
\min_{0\leq s\leq t}\{(1-G_s) \wedge z\}\geq \tfrac{3}{4},
\end{align*}
which proves~\eqref{ineq:ztilde>1/2}, and thus establishes the first equality in~\eqref{ineq:min.ztilde=ztilde_T=1/2}. To prove the second equality,
namely $\ztil_T=1/2$, we have to show that equality holds in~\eqref{ineq:ztilde>1/2} at time $T$, i.e.,
\begin{align} \label{eqn:ztilde_T=1/2}
\min_{0\leq t\leq T}\{(1-G_t) \wedge z\}=(1-G_T)-\tfrac{1}{2}.
\end{align}
We invoke~\eqref{eqn:int.x.dt:z>3/4} again to see that 
\begin{align*}
(1-G_T)-\tfrac{1}{2}=\tfrac{3}{4},
\end{align*}
and recall that $\max_{0\leq t\leq T}G_t=1/4$ and $z\geq 3/4$, implying
\begin{align*}
\min_{0\leq t\leq T}\{(1-G_s) \wedge z\}=\tfrac{3}{4},
\end{align*}
which proves~\eqref{eqn:ztilde_T=1/2} and thus completes ~\eqref{ineq:min.ztilde=ztilde_T=1/2}.

Turning to~\eqref{ineq:inf_y.in.B(x).P(y,B(x_*))>0}, since $(\xbftd_t,\ktil_t)$ satisfies~\eqref{eqn:control.problem} and drives $\xbf$ to $\xbf_*$ at time $T$, we invoke Lemma \ref{lem:support.theorem} to infer the existence of an $\varepsilon(\xbf)<\varepsilon_1$ such that for all $\ybf\in B(\xbf,\varepsilon(\xbf))$
\begin{align*} 
\PP_T(\ybf,B(\xbf_*,\varepsilon_1))\geq\P\big\{\sup_{0\leq t\leq T}|\ybf_t-\widetilde{\xbf}_t|\leq \varepsilon_1\big\} \geq c(\xbf)>0.
\end{align*}
This establishes~\eqref{ineq:inf_y.in.B(x).P(y,B(x_*))>0}.

Now, since $\domain_R\cap\{z\geq 3/4\}$ is compact, there exists a finite union of $B(\xbf,\varepsilon(\xbf))$ that covers $\domain_R\cap\{z\geq 3/4\}$. It then follows from~\eqref{ineq:inf_y.in.B(x).P(y,B(x_*))>0} that
\begin{align*} 
\inf_{\ybf\in \domain_R\cap\{z\geq 3/4\} }\PP_T\{\ybf,B(\xbf_*,\varepsilon_1)\}>0,
\end{align*}
which together with~\eqref{ineq:inf.x.in.C_R.and.z<3/4} proves the claim~\eqref{ineq:inf.x.in.C_R}. The proof is thus complete.
\end{proof}
\begin{remark} For general domains, the control $\ktil_t$ as in~\eqref{eqn:control.z>3/4.ktilde} is not known explicitly. In our context, expression~\eqref{eqn:control.z>3/4.ktilde} shows that $\ktil_t$ as a function of $\xtil_t$ is continuous with respect to sup norm. Thus, we are able to derive a uniform bound lower bound for $\mathcal{P}_T(\xbf,B(\xbf_*,\varepsilon_1))$.
\end{remark}

\section{Lyapunov definition} \label{sec:Lyapunov}

 In this section, we use the framework in~\cite{athreya2012propagating,herzog2015noise,herzog2015noiseII,
 herzog2017ergodicity} to define the Lyapunov function $\Psi$ as in the statement of Proposition~~\ref{prop:Lyapunov:xyz}. We first note some:

\subsection{Differences in the reflected setting}
Introduce the operators  
\begin{align}
\L&= -\gamma x\partial_x -\frac{hx^2-y^2}{z}\partial_x-\gamma y\partial_y-(1+h)\frac{xy}{z}\partial_y+\kappa_1\partial_{x}^{2}+\kappa_2\partial_{y}^{2}+(1-h)x\partial_z, \label{generator:L}\\
\Q&=-\partial_z, \label{generator:Q}
\end{align}
and observe that Dynkin's formula for the dynamics~\eqref{eqn:xyz} reads   
\begin{align}
\label{eqn:DF}
\E_{\xbf} \f(\xbf_{t\wedge \tau_R})=\f(\xbf)+\E_{\xbf} \int_0^{t\wedge \tau_R}\close \L\f(\xbf_s)\d s&+\E_{\xbf}\int_0^{t\wedge \tau_R}\close\Q\f(\xbf_s)\d k_s,
\end{align}
for any $\f\in C^2(\domain; \rbb)$, $t\geq 0$ and $R>0$.  Thus the most notable difference between a typical SDE without reflection and one with reflection is the appearance of the boundary operator $\Q$.   In particular, in order to construct a Lyapunov function $\Psi$ as in Proposition~\ref{prop:Lyapunov:xyz}, we have to simultaneously control the interior dynamics driven by $\L$ as well as the boundary effects contained in $\Q$.  

Since $k_t$ is a non-decreasing function, following the common approaches in the standard SDE setting~\cite{khasminskii2011stochastic, meyn1993stability}, this means we should look for a function $\Psi\in C^2(\domain; [1, \infty))$ satisfying the following conditions:
\begin{align*}
\L \Psi \leq -c \Psi +D \,\, \text{ on } \,\, \domain \qquad \text{ and } \qquad \Q \Psi(x,y,1) \leq 0 \,\, \text{ for } (x,y) \in \RR^2  ,
\end{align*}   
for some constants $c, D>0$. If there existed such a $\Psi$ which moreover satisfies part (a) of Proposition~\ref{prop:Lyapunov:xyz}, then parts (b) and (c) of the same result would follow immediately in the usual way by using Dynkin's formula as in~\cite{khasminskii2011stochastic,meyn1992stochastic}.  Following the approach outlined out in~\cite{herzog2015noise,herzog2015noiseII}, however, we find it more convenient to construct a Lyapunov function $\Psi$ which is not globally $C^2$.  In particular, $\Psi$ will be $C^2$ everywhere in $\domain$ except for finitely many, non-intersecting two-dimensional interfaces in the domain on which the function is only continuous.  Intuitively, this is because the dynamics is different in different regions of the phase space so that, when building the function locally on each subregion, it is easier to glue the local functions together continuously (rather than in a $C^2$ manner).  When doing this, however, one issue that needs to be addressed is the absence of an It\^{o}/Tanaka formula for such functions.  While such a formula exists in the standard SDE setting~\cite{peskir2007change,herzog2015noise, herzog2015noiseII}, to the best of the authors' knowledge, one does not exist in the reflected SDE setting.  However, one follows almost immediately from the argumentation in~\cite{herzog2015noise, herzog2015noiseII} (see Appendix~\ref{sec:Tanaka}).

\subsection{Change of coordinates and decomposition of the phase space}\label{sec:space.decomposition}
We find it convenient to construct a Lyapunov function $\Phi=\Phi(u,v,z)$ for the equation~\eqref{eqn:uvz} and then change back to the original system~\eqref{eqn:xyz} to produce a Lyapunov function $\Psi(x,y,z) = \Phi(\tfrac{x}{z^{1/3}}, \tfrac{y}{z^{1/3}}, z)$ afterwards.  Note that, after changing coordinates as in~\eqref{eqn:cocoord}, the respective interior and boundary operators for the reflected system~\eqref{eqn:uvz} are given by
\begin{align}
\L&= -\gamma u\partial_u- \gamma v\partial_v-  \frac{\alpha_h u^2-v^2}{z^{2/3}}\partial_u-(\alphah+1)\frac{uv}{z^{2/3}}\partial_v+(1-h)uz^{1/3}\partial_z \nonumber\\
&\hspace{1cm}+\frac{\kappa_1}{z^{2/3}}\partial_{u}^2+\frac{\kappa_2}{z^{2/3}}\partial_{v}^2, \label{generator:L:uvz}\\
\Q&=\frac{u}{3z}\partial_u+\frac{v}{3z}\partial_v-\partial_z \label{generator:Q:uvz},
\end{align}
where $\alphah=\frac{1}{3}+\frac{2}{3}h$ as in~\eqref{def:alpha}. Our Lyapunov function $\Phi$ will be built piece-by-piece locally on various subregions of the domain making use of these operators.  We first provide some heuristics, however, as to why one should expect the specific subregions utilized in the construction.  Following the approach outlined and developed in~\cite{athreya2012propagating,gawedzki2011ergodic, herzog2015noise,herzog2015noiseII,herzog2017ergodicity}, this will be done by performing a scaling analysis of the operator $\L$ above which affords a better understand of the dynamics near the boundary in space; that is, where either $r=r(u,v)=\sqrt{u^2+v^2} \gg1 $ or $z\approx 0$.  The specific scalings chosen represent different paths near the boundary in space.

 Following the heuristics of Section~\ref{sec:intuition}, it is helpful to express $\L$ as 
 \begin{align*}
 \L = z^{-2/3} \M,
 \end{align*} 
 where $\M$ is as in~\eqref{def:M}.
 In what follows, we will actually do the scaling analysis applied to $\M$ as opposed to $\L$.  Since the factor $z^{-2/3}$ is positive on the phase space $\domain$, the analysis done for $\M$ is easily transferrable to $\L$.  Although we do not try to make precise sense of it, this factorization translates to a random time change applied to~\eqref{eqn:uvz} where the dynamics slows down or speeds up according to the $z$-process.        
  
 \subsubsection{Dynamics along a fixed large ray in the $(u,v)$-plane} \label{sec:Lyapunov:large(u,v)}
 We first introduce the scaling transformation 
\begin{align*}
S_{a}^{\lambda,1}: (u,v,z)\mapsto (\lambda u,\lambda v, \lambda^{-a}z),
\end{align*}
for a scaling parameter $\lambda\gg 1$ and a fixed constant $a\geq 0$.  This allows us to realize the dynamics along a fixed ray away from the negative $u$-axis, while also considering the possibility that $z$ tends to $0$ simultaneously, i.e., if $a>0$.     

Considering $\M$ as in~\eqref{def:M}, under the transformation $S_a^{\lambda}$ we find that 
\begin{align}
\label{eqn:scaling1}
\M\circ S_{a}^{\lambda,1}(u,v,z)&=-\lambda^{-\tfrac{2a}{3}}\gamma u z^{2/3} \partial_u-\lambda^{-\tfrac{2a}{3}}\gamma v z^{2/3}\partial_v -\lambda (\alpha_h u^2-v^2) \partial_u-\lambda (\alpha_ h+1) uv\partial_v\\
&\qquad+\lambda (1-h) uz \partial_z+\lambda^{-2} \kappa_1\partial_{u}^{2}+ \lambda^{-2} \kappa_2\partial_{v}^{2}.\nonumber
\end{align}
When $\lambda \gg 1$ and $a\geq 0$, the dominant balance of terms in $\M$ is contained in the operator $\T_1$ defined by 
\begin{align}
\label{eqn:R1:T1}
\T_1 = -  (\alpha_h u^2-v^2)\partial_u-(\alphah+1) uv\partial_v+(1-h)uz\partial_z  .
\end{align}
Thus the noise terms and the term $  -\gamma u z^{2/3}\partial_u- \gamma vz^{2/3} \partial_v$ are negligible in a region of the form 
\begin{align}
\label{def:R0R1}
\mathcal{R}_0\cup \mathcal{R}_1:= \{(u,v,z) \in \domain \, : \, r\geq r_*, u\geq -C|v|\} ,
\end{align}  
where $r_*, C>0$ are sufficiently large constants to be determined later.  The individual subregions $\mathcal{R}_0$ and $\mathcal{R}_1$ that define the union above will also be defined later.  Essentially, though, the region $\R_0$ is defined by our basic initial guess of a Lyapunov function, and is hence the region in which this function is actually a Lyapunov function locally on that region.

\subsubsection{Dynamics at large radius near the negative $u$-axis}
For the case when $r\gg1$ and the process is near the negative $u$-axis, we consider the scaling transformation  
\begin{align*}
S_{a,b}^{\lambda,2}: (u,v,z) \mapsto ( \lambda u, \lambda^{-b} v, \lambda^{-a} z),
\end{align*}
where $\lambda \gg 1$ is the scaling parameter and $b>0$ and $a\geq 0$ are fixed constants.  Then it follows that 
\begin{align*}
\M\circ S_{a,b}^{\lambda,2}(u,v,z)&=-\lambda^{-\tfrac{2a}{3}}\gamma u z^{2/3} \partial_u-\lambda^{-\tfrac{2a}{3}}\gamma v z^{2/3}\partial_v -(\lambda\alpha_h u^2-\lambda^{-2b-1}v^2) \partial_u-\lambda (\alpha_ h+1) uv\partial_v\\
&\qquad+\lambda (1-h) uz \partial_z+\lambda^{-2} \kappa_1\partial_{u}^{2}+ \lambda^{2b} \kappa_2\partial_{v}^{2} \nonumber \\
&\approx  -\lambda \alpha_h  u^2 \partial_u - \lambda (\alpha_{h}+1)uv \partial_v + \lambda (1-h) u z \partial_z  + \lambda^{2b} \kappa_2 \partial_v^2.  
\end{align*}  
Thus the dominant balance of terms is contained in the operator \begin{align}
\label{eqn:T3def}
|u|\T_3:=-\alpha_h u^2\partial_u-(\alpha_h+1)uv\partial_v+\kappa_2\partial_{v}^2+(1-h)zu\partial_z+\kappa_2\partial_{v}^2,
\end{align}
and moreover the region in which the noise becomes dominant is defined precisely when $b=1/2$ above.  In particular, $|u|\T_3$ is a good approximation of the dynamics in the region $\mathcal{R}_2\cup \mathcal{R}_3$ where 
\begin{align}
\label{def:R2}
\R_2&=\{r\geq r_*, u \leq -C |v| , |u|^{1/2}v\geq \eta_*\}\cap\{0< z\leq 1\},\\
\R_3&=\{r\geq r_*, u \leq - C |v|, |u|^{1/2}v\leq \eta_*\}\cap\{0< z\leq 1\} \label{def:R3}.
\end{align} 
Observe that the boundary between $\R_2$ and $\R_3$ where noise becomes dominant is defined precisely by the threshold $b=1/2$.

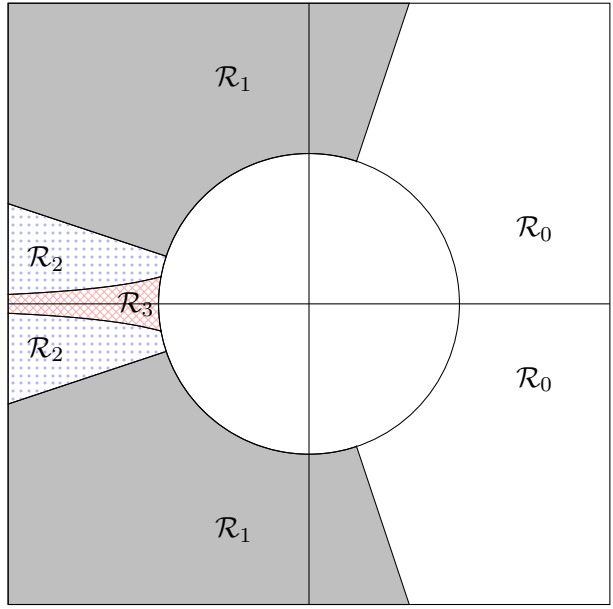
\begin{figure}[H]

\begin{tikzpicture}
\draw[fill=gray!50!white] 
(0.632,1.897) -- 
(1.333,4) -- (-4,4) -- (-4,1.333) -- (-1.897,0.632) --
plot[smooth,samples=100,domain=-1.897:0.632] (\x,{sqrt(4-\x*\x)}) -- cycle;
\draw[fill=gray!50!white] 
(0.632,-1.897) -- 
(1.333,-4) -- (-4,-4) -- (-4,-1.333) -- (-1.897,-0.632) --
plot[smooth,samples=100,domain=-1.897:0.632] (\x,{-sqrt(4-\x*\x)}) -- cycle;

\draw[ pattern=dots, pattern color = blue!30] 
(-4,0.125) -- 
plot[smooth,samples=100,domain=-4:-1.966] (\x,{1/abs(\x)^(1.5)})
-- 
plot[smooth,samples=100,domain=-1.966:-1.897] (\x,{sqrt(4-\x*\x)})
 -- (-4,1.333)  -- cycle;
 
 \draw[ pattern=dots, pattern color = blue!30] 
(-4,-0.125) -- 
plot[smooth,samples=100,domain=-4:-1.966] (\x,{-1/abs(\x)^(1.5)})
-- 
plot[smooth,samples=100,domain=-1.966:-1.897] (\x,{-sqrt(4-\x*\x)})
 -- (-4,-1.333)  -- cycle;
 
 \draw[ pattern=crosshatch, pattern color = red!30] 
(-4,0.125) -- 
plot[smooth,samples=100,domain=-4:-1.966] (\x,{1/abs(\x)^(1.5)})
-- 
plot[smooth,samples=100,domain=-1.966:-2] (\x,{sqrt(4-\x*\x)}) --
plot[smooth,samples=100,domain=-2:-1.966] (\x,{-sqrt(4-\x*\x)})
 -- 
 plot[smooth,samples=100,domain=-1.966:-4] (\x,{-1/abs(\x)^(1.5)})  -- cycle;

\draw (-4,-4) -- (-4,4) -- (4,4) -- (4,-4)-- (-4,-4);
      \draw[-] (-4,0) -- (4,0) ;
      \draw[-] (0,-4) -- (0,4); 

\draw (0,0) circle (2cm);

\draw[color=black] (3,1)  node{$\R_0$};
\draw[color=black] (3,-1)  node{$\R_0$};

\draw[color=black] (-1,3)  node{$\R_1$};
\draw[color=black] (-1,-3)  node{$\R_1$};

\draw[color=black] (-3.5,0.6)  node{$\R_2$};
\draw[color=black] (-3.5,-0.6)  node{$\R_2$};

\draw[color=black] (-2.3,0)  node{$\R_3$};

\end{tikzpicture}
 \caption{The outer subregions $\R_i$, $i=0,1,2,3$ in $(u,v)-$plane.}
\end{figure}

\begin{remark}
While we introduced the operator $|u| \T_3$ above, we have yet to introduce $\T_2$.  This is done in~\eqref{def:T2}.  This operator serves to better transition the dynamics from $\R_1$ to $\R_3$ in the region $\R_2$.  
\end{remark}

\subsubsection{Dynamics at bounded radius and small $z$}
\label{subsub:smallz}
Lastly, in the case when $r$ is bounded and $z$ is near $0$, we consider the scaling 
\begin{align*}
S^{\lambda,3}_a: (u,v,z) \mapsto (u,v, \lambda^{-a} z) ,
\end{align*}
where $\lambda \gg 1$ is our scaling parameter and $a>0$ is a fixed constant.  In this case, when $\lambda \gg1$
\begin{align*}
\M\circ S_{a}^{\lambda,3}(u,v,z)&=-\lambda^{-\tfrac{2a}{3}}\gamma u z^{2/3} \partial_u-\lambda^{-\tfrac{2a}{3}}\gamma v z^{2/3}\partial_v -(\alpha_h u^2-v^2) \partial_u- (\alpha_ h+1) uv\partial_v\nonumber\\ 
&\qquad+(1-h) uz \partial_z+ \kappa_1\partial_{u}^{2}+  \kappa_2\partial_{v}^{2}\nonumber  \\ 
&\approx  -(\alpha_h  u^2-v^2) \partial_u -  (\alpha_{h}+1)uv \partial_v + (1-h) u z \partial_z  +  \kappa_1 \partial_u^2 + \kappa_2 \partial_v^2=\N.     
\end{align*}   
Unfortunately, in this case the only terms that can be neglected are $-\gamma u z^{2/3} \partial_u$ and $-\gamma v z^{2/3} \partial_v$.  This naturally translates to difficulties in constructing the appropriate Lyapunov functional in the domain
\begin{equation}
\label{def:Rinner} 
\RI := \{ r \leq 2r_* \}\cap \{z<\varepsilon_0\},
\end{equation}
where $\varepsilon_0>0$ is small and to be determined later, since almost all terms in $\M$ must be included in the PDE we wish to solve.  Nevertheless, we have successfully decomposed the domain $\domain$ outside of a compact set according to the dynamics near the boundaries in space.  

We thus define the \textbf{outer} $\RO$ and \textbf{inner} $\RI$ regions of $\domain$ as 
\begin{align}
\RO= \R_0 \cup \R_1 \cup \R_2 \cup \R_3,
\end{align}   
and $\RI$ is as in~\eqref{def:Rinner}.

\subsection{The construction in $\RO$} \label{sec:outer-region} In this subsection, we define local functions $\f_{i, j}$ on $\R_i$, $i=0,1,2,3$ and $j=1,2,3$, that will eventually make up the part of the Lyapunov function $\Phi$ on $\RO$. 

\subsubsection{Parameter choices}

Before beginning the construction, we make some choices in parameters.  All of these choices will become apparent later.  The parameters either correspond to boundary parameters in the definitions of the regions $\R_i$ or constants that appear in the definition of the local functions $\f_{i,j}$.      
\begin{choices}
\label{cond:pq}  We let constants $p_i,q_i,\alpha_i,\beta_i$, $i=1,2$, satisfy $p_1>0,\,p_2>0$, $q_1=0$, $q_2>0$, 
\begin{align}
 0&<\alpha_h p_i-(1-h)q_i=\beta_i<\alpha_i<1,
\label{cond:alphai.betai}\\
\text{and}\quad q_2&>\tfrac{1}{3}p_2+\tfrac{1}{2}\alpha_2,\quad p_2>p_1,\quad p_2+\tfrac{3}{2}\alpha_2>p_1+\tfrac{3}{2}\alpha_1.\label{cond:p2.q2} 
\end{align}
\end{choices} 

\begin{remark} \label{rem:p_1} We note that the choice of parameters $(p_i,q_i,\alpha_i,\beta_i)$, $i=1,2$, above is merely to establish the well-posedness and geometric ergodicity of~\eqref{eqn:uvz} and thus that of~\eqref{eqn:xyz}. To achieve the optimal moment bound~\eqref{ineq:int.r^(p+1).pi(dx)} in Theorem~\ref{thm:geometric.ergodicity} (c), however, we will have to pick another  set of parameters  $(p_i,q_i,\alpha_i,\beta_i)$, $i=3,4,$ later. See the proof of Theorem~\ref{thm:geometric.ergodicity} in Section~\ref{sec:proof-thm-ergodicity}.
\end{remark}

\begin{choices} For a sufficiently large constant $C>0$, the parameters $\eta_*$ and $r_*$ satisfy
\begin{align} \label{cond:r*.eta*}
\eta_*=C\sqrt{\kappa_2} \qquad \text{ and } \qquad r_*\gg \max\{C,\eta_*,\gamma\}. 
\end{align}
\end{choices}
In the above, we recall that $r_*$ is the parameter for the boundary $r=r_*$ in $\RO$ and $\eta_*$ is the parameter for the boundary between $\R_2$ and $\R_3$ as in~\eqref{def:R2} and~\eqref{def:R3}.

\subsubsection{Initializing the procedure in $\R_0$}  We first consider an initial region $\R_0$ given by 
\begin{align}
\R_0&=\{ r\geq r_*,C u\geq |v|\}\cap\{0<  z\leq 1\}.
\end{align} 
This region serves to start the construction procedure by providing an initial guess that will be propagated through the rest of the outer domain $\RO$ using boundary-valued PDEs.  

In $\R_0$, we let $\f_{0,i}$, $i=1,2$, be given by 
\begin{equation} \label{form:R0:phi0}
\f_{0,i} = r^{p_i}z^{q_i},\quad i=1,2,
\end{equation}
where $p_i,\,q_i$ are as in a Choice~\ref{cond:pq}.  Recalling that the operator $\T_1$ in~\eqref{eqn:R1:T1} contains the dominant balance of terms in $\M$ in this region, we observe that 
\begin{align}
\T_1 \f_{0, i}  = -(p_i \alpha_h -q_i(1-h) )u \f_{0, i}.  
\end{align}
Hence we see why we want at least part of Choice~\ref{cond:pq} so that the above is large, negative in $\R_0$.

\subsubsection{The construction in $\R_1$} 
In this region, the dynamics is again asymptotically determined by the operator $\T_1$ as in~\eqref{eqn:R1:T1} in the region $\R_1$ given by 
\begin{align}
\R_1  =\{ (u,v,z) \, : \, r\geq r_*, \, C^{-1} |v| \geq u \geq -C|v|\}. 
\end{align}
Here, we propagate the initial guesses $\f_{0,i}$, $i=1,2$, through the region $\R_1$ using the following PDEs
\begin{equation}\label{eqn:R1:T1:boundaryproblem}
\begin{cases}
\T_1\f_{1,i} = -c_{1,i} r^{p_i+1}z^{q_i}\Big|\dfrac{r}{v}\Big|^{\alpha_i},  \\ 
\f_{1,i}\big|_{\{u=C^{-1}|v|\}} = \f_{0,i}, \qquad i=1,2.  
\end{cases}
\end{equation}
In the above, $\alpha_i$ is as in Choice~\ref{cond:pq} and $c_{1,i}$ satisfies
\begin{equation}\label{cond:R1:c1}
 c_{1,i}=\frac{\beta_i}{2 C}.
\end{equation}
\begin{remark} The choices of the positive constants $c_{1,i}$ are to ensure that the local-time contributions from $\f_{0,i}$ and $\f_{1, i}$ on the boundary $\{u=C^{-1} |v|\}$ between $\R_0$ and $\R_1$ have the appropriate sign (see the Tanaka formula~\eqref{eqn:Tanaka} in Appendix~\ref{sec:Tanaka}).

\end{remark}
In order to solve the PDE~\eqref{eqn:R1:T1:boundaryproblem}, we use the method of characteristics to obtain the explicit formula for $\f_{1,i}$.  First, however, it is helpful to transform the equation~\eqref{eqn:R1:T1:boundaryproblem} to cylindrical coordinates $(r, \theta, z)$ to see that, as a function of $(r, \theta, z)$, $\f_{1,i}$ satisfies the following PDE
\begin{align}
\label{eqn:transR1}
\begin{cases}
\widetilde{\T_1} \f_{1, i}(r,\theta, z) = -c_{1, i} \frac{r^{p_i} }{|\sin \theta|^{\alpha_i}} z^{q_i}\,\, \text{ on } \,\, \R_1,&\\
\f_{1,i}\big|_{\{\cos \theta =C^{-1}|\sin \theta|\}} = \f_{0,i}, \qquad i=1,2,& \end{cases} 
\end{align} 
where 
\begin{align}
\widetilde{\T_1}= -\alpha_h r \cos \theta \partial_r - \sin \theta \partial_\theta + (1-h) z \cos \theta \partial_z.  
\end{align}
Indeed, after changing coordinates, one can cancel one power of $r$ from both sides of the resulting equation.  This allows us to work with the simpler dynamics drived by $\widetilde{\T_1}$ above which can be explicitly solved.  

Solving the equation~\eqref{eqn:transR1} and then changing back to the original coordinates $(u,v,z)$ produces the formula
\begin{equation}\label{eqn:R1:phi1}
\begin{aligned}
\f_{1,i}(u,v,z)&= \frac{r^{p_i}z^{q_i}}{(C^{-2}+1)^{\beta_i /2}}\Big|\frac{r}{v}\Big|^{ \beta_i }+c_{1,i}r^{p_i}z^{q_i}\Big|\frac{r}{v}\Big|^{\beta_i}\int_{u/|v|}^{1/C}(t^2+1)^{\frac{\alpha_i-\beta_i -1}{2}}\d t.
\end{aligned}
\end{equation}
Alternatively, a routine but tedious calculation shows that $\f_{1,i}$ solves~\eqref{eqn:R1:T1:boundaryproblem}.

\subsubsection{The construction in $\R_2$}
We next propagate the previous functions $\f_{1,i}$, $i=1,2$, through the region $\R_2$ as in~\eqref{def:R2}. In this transition region, the dominant balance of terms in $\M$ is contained in the first-order part of $|u| \T_3$ as in~\eqref{eqn:T3def}, which is given by
\begin{align}
\label{def:T2}
\T_2:= -\alpha_h u^2\partial_u-(\alpha_h+1)uv\partial_v+(1-h)zu\partial_z. 
\end{align}
Indeed, the restriction $|u|^{1/2} |v| \geq \eta_*$ for $\eta_*>0$ sufficiently large ensures that the second-order part in $|u| \T_3$ does not interfere with $\T_2$.  Consequently, we define $\f_{2,i}$, $i=1,2$, as the solutions of the following PDEs:
\begin{equation}\label{eqn:R2:T2:boundaryproblem}
\begin{cases}
 \T_2\f_{2,i}=-c_{2,i} |u|^{p_i+1}z^{q_i}\Big|\dfrac{u}{v}\Big|^{\alpha_i} \,\, \text{ on } \,\, \R_2,&\\
\f_{2,i}\big|_{\{u=-C|v|\}} =\f_{1,i}, \,\,\, i=1,2.&
\end{cases}
\end{equation}
In the above, we choose the positive constants $c_{2,i}$, $i=1,2$, according to 
\begin{equation} \label{cond:R2:c2}
c_{2,i}=\tfrac{1}{2} c_{1,i}.
\end{equation}
Similar to the construction in region $\R_1$, we can use the method of characteristics 
to arrive at an explicit formula for $\f_{2, i}$
\begin{equation}\label{eqn:R2:phi2}
\begin{aligned}
&\f_{2,i}(u,v,z)= A_{2,i}|u|^{p_i}z^{q_i}\Big|\frac{u}{v}\Big|^{\beta_i }+B_{2,i}|u|^{p_i}z^{q_i}\Big|\frac{u}{v}\Big|^{\alpha_i},
\end{aligned}
\end{equation}
where the constants $A_{2, i}$ and $B_{2, i}$ satisfy
\begin{equation}\label{eqn:R2:constantA2}
\begin{aligned}
A_{2,i}&=(C^{-2}+1)^{p_i/2}+c_{1,i}(C^{-2}+1)^{\frac{p_i+\beta_i}{2}}\int_{-C}^{C^{-1}}\close (t^2+1 )^{\frac{\alpha_i-\beta_i-1 }{2}}\d t-B_{2,i}C^{\alpha_i-\beta_i },
\end{aligned}
\end{equation} 
and 
\begin{equation}\label{eqn:R2:constantB2}
B_{2,i}=\frac{ c_{2,i}}{ \alpha_i-\beta_i  }.
\end{equation}
In this case, applying the method of characteristics is slightly less complicated than in $\R_1$ after factoring out a power of $-u=|u|$ from both sides of~\eqref{eqn:R2:T2:boundaryproblem} since the resulting dynamics is simpler.

\subsubsection{The construction in $\R_3$}
We now consider the last subdomain $\R_3$ in the outer region $\RO$. Recall that the dominant balance of terms in $\M$ is contained in the operator $|u| \T_3$ as in~\eqref{eqn:T3def}.  As suggested by the expression $|u| \T_3$, we will actually propagate $\f_{i,2}$ through this region using $\T_3$ instead of $|u| \T_3$.  Note that this is equivalent to making a random time change since $|u|>0$ is large in $\R_3$.  Different from the previous regions, however, the first exit time $T_3$ from the region $\R_3$ for any diffusion defined by $\T_3$ is now random. Nevertheless, we will see that this exit time $T_3$ can be associated to one of an OU process by making the change of coordinates 
\begin{equation}
\eta=|u|^{1/2}v.
\end{equation}
This will allow us to solve the associated second-order PDEs.  

To this end, offering an abuse of notation, consider the process $(u_t, v_t, z_t)$ solving 
\begin{align}
\label{eqn:T3d}
\d u_t &=  \alpha_h u \, \dt, \\
\d v_t&= (\alpha_h +1) v \, \d t + \frac{\sqrt{2\kappa_2}}{ |u_t|^{1/2}} \, \dW_t, \nonumber \\
\d z_t&=-(1-h) z \, \dt ,\nonumber 
\end{align}
where $W_t$ is a standard Brownian motion on $\RR$.  
Then It\^{o}'s formula implies that $\eta_t = |u_t|^{1/2} v_t$ satisfies
\begin{align*}
\d\eta_t=\big(\tfrac{3}{2}+h\big)\eta_t\d t+\sqrt{2\kappa_2}\, \d W_t .
\end{align*}
Hence, provided $u_0 \neq 0$,
\begin{align*}
T_3:=\inf_{t\geq 0}\{|u_t|^{1/2}|v_t|\geq \eta_*, (u_0,v_0,z_0)\in \R_3\}=\inf_{t\geq 0}\{  |\eta_t|\geq \eta_*, |\eta_0|\leq \eta_*\}.
\end{align*}
Furthermore, the operator $\T_3$ given by~\eqref{eqn:T3def} in the coordinates $(u, \eta, z)$ becomes 
\begin{equation}\label{eqn:R3:T:hat3}
\widetilde{\T}_3=\alpha_h u\partial_u+\big(\tfrac{3}{2}+h\big)\eta\partial_\eta+\kappa_2\partial_{\eta}^2-(1-h)z\partial_z.
\end{equation}

With the above observations, we define $\f_{3, i}$, $i=1,2$ as the solutions of the following PDEs:
\begin{equation}\label{eqn:R3:T.hatphi3:boundary.problem}
\begin{cases}
\widetilde{\T}_3\f_{3,i}=-c_{3,i}|u|^{p_i+\frac{3}{2}\alpha_i}z^{q_i} \,\, \text{ on } \,\, \R_3,&\\
\f_{3,i}\big|_{\{|\eta|=\eta_*\}}=\f_{2,i}, \,\,\, i=1,2.&
\end{cases}\end{equation}
where we offer the abuse of notation that $\f_{3,i}$ denotes both the function in $(u,\eta, z)$ and $(u,v,z)$.  The constants $c_{3,i}$ are chosen to satisfy
\begin{equation} \label{cond:R3:c3}
c_{3,i}=\frac{c_{2,i}}{2\eta_*^{\alpha_i}}.
\end{equation}
Formally solving the above equation, we obtain
\begin{equation}\label{eqn:R3:phi3}
\begin{aligned}
\f_{3,i}(u,v,z)&=A_{3,i}|u|^{p_i+\frac{3}{2}\beta_i}z^{q_i}\E_\eta \big[e^{\gamma_i T_3}\big]+ B_{3,i}|u|^{p_i+\frac{3}{2}\alpha_i}z^{q_i}\E_\eta\big[e^{\gammah_i T_3}\big]\\
&\qquad+C_{3,i} |u|^{p_i+\frac{3}{2}\alpha_i}z^{q_i}\E_\eta\big[e^{\gammah_i T_3}-1\big].
\end{aligned}
\end{equation}
In the above, we have set
\begin{equation} \label{cond:R3:gamma_i.gammahat_i}
\gamma_i=(h+\tfrac{3}{2})\beta_i,\qquad \gammah_i=\beta_i+(h+\tfrac{1}{2})\alpha_i,
\end{equation}
and
\begin{equation}\label{cond:R3:A3.B3}
A_{3,i}=\frac{A_{2,i}}{\eta_*^{\beta_i }},\qquad B_{3,i}= \frac{B_{2,i}}{\eta_*^{\alpha_i}},\qquad C_{3,i}=\frac{c_{3,i}}{\gammah_i},
\end{equation}
where $A_{2,i}$ and $B_{2,i}$ are as in~\eqref{eqn:R2:constantA2} and~\eqref{eqn:R2:constantB2}, respectively. Making use of an analysis on the exiting time $T_3$, cf. Lemma~\ref{lem:R3:G}, we will prove that $\f_{3,i}$ given by~\eqref{eqn:R3:phi3} is well-defined and satisfies the PDE above.  This involves showing, in particular, that the Laplace transforms of the exit time $T_3$ in~\eqref{eqn:R3:phi3} make sense and are smooth functions of $\eta$ for $\eta\in [-\eta_*, \eta_*]$.

\subsubsection{Definition of $\Phi_\emph{O}$ on $\R_\emph{O}$}
Using the functions $\f_{j,i}$, we can now define the outer part $\fO$ of the Lyapunov function $\Phi$ on $\R_\emph{O}$.  \begin{definition} \label{def:f_O} Let $\f_i(\ubf)=\f_i(u,v,z)$, $i=1,2$, be the functions on $\R_\emph{O}$ given by
$$\f_i(\ubf)=\f_{j,i}(\ubf),\qquad\ubf\in\R_j, \quad j=0,1,2,3.$$
For $\ubf\in\R_\emph{O}$, we define the \textbf{outer part} $\Phi_\emph{O}$ of $\Phi$ by 
\begin{equation*}
\Phi_\emph{O}(\ubf)=\f_1(\ubf)+\f_2(\ubf).
\end{equation*}

\end{definition}

\begin{remark} \label{rem:f_1}
Recalling that $q_1=0$ in Choice~\ref{cond:pq}, in Definition~\ref{def:f_O}, $\f_1$ is indeed a function of $(u,v)$ only.
\end{remark}

\subsection{The construction in $\RI$} \label{sec:inner-region}  We next turn to the construction in the inner region $\RI$ as in~\eqref{def:Rinner}.  While dissipation in the system in $\RO$ is generated by transport phenomena; that is, the process exits `bad' parts of space, e.g., $\R_3$, and is transported to a radially contractive region $\R_0$, the situation in $\RI$ is different.  In the inner region, dissipation near the boundary in space, in this case $z\approx 0$, occurs due to global averaging effects of the system as discussed in Section~\ref{sec:intuition}.  Our Lyapunov function $\Phi$ in this region, therefore, must reflect these effects.   

  To this end, we recall the process $(U_t, V_t)$ as in~\eqref{eqn:UVZ} is Markov with generator $\A$ given by~\eqref{def:A}.  We also recall by the heuristics of Section~\ref{subsub:smallz} that $(U_t, V_t)$ determines the approximate dynamics in $\RI$ projected onto the first two coordinates when the initial $z$-coordinate is small.  For $n\in \N$, let 
  \begin{align}
  \tau_n = \inf\{ t\geq 0 \,: \, |(U_t, V_t)| \geq n \}.  
  \end{align}       
Recalling that the function $\f_1$ in Definition~\ref{def:f_O} does not depend on $z$ since $q_1=0$ (cf. Remark~\ref{rem:f_1}), we introduce the following terminology.  
\begin{definition} \label{def:averaging.Lyapunov}
Suppose that there exists a function $\psi \in C(\RR^2; \RR)$ satisfying the following two properties:
\begin{itemize}
\item[(i)] $\psi(u, v ) =o(\f_1( u , v ))$ as $ |(u,v)| \rightarrow \infty$;
\item[(ii)] There exist constants $C_1, \epsilon >0$ such that the following estimate holds for all $n\in \mathbf{N}$, $t\geq 0$ and $(u , v ) \in \RR^2$
\begin{align}
\label{eqn:avgb}
\E_{( u , v )} \psi(U_{t\wedge \tau_n},  V _{t\wedge \tau_n}) \leq \psi(u, v ) +\E_{(u, v )}\int_0^{t\wedge \tau_n}\close\close C_1 U_s - \epsilon \, \emph{d} s.   
\end{align}
\end{itemize}  
We call such a function $\psi$ satisfying the above an \textbf{averaging Lyapunov function} corresponding to $\A$ with constants $C_1, \epsilon >0$.      
\end{definition}

\begin{remark} \label{rem:averaging-Lyapunov}
Note that any positive scalar multiple of an averaging Lyapunov function corresponding to $\A$ is also an averaging Lyapunov function corresponding to $\A$, although with different constants.  
\end{remark}

It turns out that if such an averaging Lyapunov function exists then the mean of $U$ with respect to the invariant probability measure $\mu$ for $\A$ is positive.  This is precisely stated in the following
lemma whose proof will be deferred to Appendix~\ref{sec:auxiliary-result}.
\begin{lemma} \label{lem:mu_h>0}
Suppose that there exists an averaging Lyapunov function $\psi$ corresponding to $\A$ as in~\eqref{def:A}.  Then $\mu(U)=\int_{\rbb^2}U\,\mu(\emph{d} U,\emph{d} V )>0$ where $\mu$ is the unique invariant measure for~$\A$.    
\end{lemma}

\begin{remark}
In order to prove Proposition~\ref{prop:Lyapunov:xyz} for our constructed Lyapunov function $\Phi$, we do not actually need to employ Lemma~
\ref{lem:mu_h>0}.  This result is merely an affirmation of the heuristics of Section~\ref{sec:intuition}.     
\end{remark}
We next define an averaging Lyapunov function $\g(u, v )$ under the critical assumption that $\kappa_1\approx\kappa_2$. This is the only place where this assumption is employed. In particular, everything done previously applies so long as $\kappa_2>0$ and $\kappa_1> 0$.  Proving the existence of the relevant subsolutions of the appropriate PDE to yield an averaging Lyapunov function for general $\kappa_1>0$ and $\kappa_2>0$ is beyond our current understanding of the problem.  However, the numerics of Section~\ref{sec:intuition} suggest that it should exist given that the mean of $U$ with respect to $\mu$ should be positive for any choice of $\kappa_1>0$ and $\kappa_2 >0$.   

To define the averaging Lyapunov function, we first fix two constants $J$ and $m$ as follows:
\begin{choices}  The parameter $J$ and $m$ are given by
\begin{equation}\label{cond:J.m}
J=\Big(\frac{\kappa_1+\kappa_2}{2\alpha_h}\Big)^{2/3} \quad\text{and}\quad m=\frac{\kappa_1c_*}{\alpha_h\big(\tfrac{1}{2}-\frac{12c_*}{2-c_*}\big)},
\end{equation}
where $c_*>0$ is a sufficiently small constant to be determined later.  
\end{choices}
Next, let $\lambda_1:\RR\rightarrow \RR$ be a non-decreasing, infinitely differentiable function satisfying
\begin{equation} \label{def:lambda_1}
\lambda_1(x)=\begin{cases}
0,&x\le 1,\\
1,&x\ge 2,
\end{cases}
\end{equation}
and set
\begin{equation} \label{def:psi_1}
\begin{aligned}
\g_1(u, v )&=\begin{cases} \frac{J}{2}-\frac{J}{2}\log(J)-\tfrac{ 1}{2}r^2,&  r ^{2}\le J,\\
-\frac{J}{2}\log( r ^{2}),&  r ^{2}\ge J,
\end{cases}\\
&=:\begin{cases} \psi_{11}(u, v ), & r ^{2} \leq J,\\
\g_{12}(u, v ), &   r ^{2} \geq J,\end{cases}
\end{aligned}
\end{equation}
and
\begin{equation} \label{def:psi_2}
\g_2(u, v )=-m\frac{u}{ r ^{2}}\lambda_1\Big(\frac{ r ^{2}}{J/2}\Big).  
\end{equation}
The averaging Lyapunov function $\g(u, v ,z)$ is then given by the sum of two above functions, i.e.,
\begin{equation} \label{def:psi:average.Lyapunov}
\g=\g_1+\g_2.
\end{equation}
Later in Lemma~\ref{lem:averaging.Lyapunov}, we shall prove that under the crucial condition $\kappa_1\approx\kappa_2$, $\g$ defined above is indeed an averaging Lyapunov function as in Definition~\ref{def:averaging.Lyapunov}.  Furthermore, we will show that when combined with the outer function $\fO$ in the appropriate manner, we can exhibit a Lyapunov function $\Phi$, which when transformed back to the coordinates $(x,y,z)$, will satisfy the conclusions of Proposition~\ref{prop:Lyapunov:xyz}.

\subsection{Definition of $\Phi$}

Having now defined an averaging Lyapunov function $\g$, we are now in a position to define a Lyapunov candidate $\Phi$ in the variables $(u,v,z)$.  
\begin{definition} \label{def:Phi}
Let $\Phi_\emph{O}$ and $\g$ be respectively as in Definition~\ref{def:f_O} and~\eqref{def:psi:average.Lyapunov}. The function $ \Phi$ is given by
\begin{equation} \label{eqn:Phi}
\Phi(u, v ,z)=\lambda_1(\tfrac{r}{r_*})\Phi_\emph{O}(u, v ,z)+D\g(u, v )-\frac{D\alpha_hJ}{1-h}\log(z) + A z,
\end{equation}
where $D>0$ and $A>0$ will be determined later (cf.~\eqref{cond:D} and \eqref{cond:A}, respectively) and $\alpha_h$ and $J$ are, respectively, the constants in~\eqref{def:alpha} and~\eqref{cond:J.m}, and $\lambda_1$ is the smooth cut-off function as in~\eqref{def:lambda_1}. In the above, whenever $\lambda_1=0$, we assume that the corresponding function $\lambda_1(\tfrac{r}{2r_*})\Phi_\emph{O}(u, v ,z)=0$.  \end{definition}

Having now defined $\Phi$, we next turn to establishing various Lyapunov-type bounds which will be used to prove Proposition~\ref{prop:Lyapunov:xyz} later.  
\section{Local Lyapunov bounds} \label{sec:Lyapunov:proof}
In this section, we start working towards a proof of Proposition~\ref{prop:Lyapunov:xyz} by establishing various local Lyapunov-type estimates for each locally-defined function introduced in the previous section.  This means that we need to establish a number of bounds with respect to both the interior operator $\L$ as in~\eqref{generator:L:uvz} and the boundary operator $\Q$ as in~\eqref{generator:Q:uvz}.  We proceed region-by-region starting with each subregion in the outer domain $\RO$ in Section~\ref{sec:outer:lyapprop} and then proceed to do the same for the inner region $\RI$ in Section~\ref{sec:inner:lyapprop}.

 \subsection{Local Lyapunov bounds in the outer region $\RO$} \label{sec:outer:lyapprop}

Here we proceed region by region, starting with $\f_{0,i}$, $i=1,2$, in the initial outer region $\R_0$.

\begin{lemma}\label{lem:R0:Lyapunov} Let $\f_{0,i}$, $i=1,2$ be given as in~\eqref{form:R0:phi0}. Then for all $r_*>0$ large enough, we have
\begin{align}
\label{eqn:R0est1}
\L\f_{0,i}(u,v,z)\le -\gamma p_i r^{p_i}z^{q_i} -\frac{\beta_i}{2\sqrt{1+C^2}}r^{p_i+1}z^{q_i-2/3},
\end{align}
for all $(u,v,z) \in \R_0$ and 
\begin{align}
\label{eqn:R0est2}
\Q\big(\f_{0,1}+\f_{0,2}\big)(u,v,1)\leq  -(\tfrac{q_2}{2}-\tfrac{p_2}{6}) r^{p_2},\end{align} 
for all $(u,v,1) \in \R_0$.  
\end{lemma}

\begin{proof} 
We apply $\L$ as in~\eqref{generator:L:uvz} to $\f_{0,i}$ on $\R_0$ to obtain 
\begin{align*}
\L\f_{0,i}(u,v,z) &= -\gamma p_i r^{p_i}z^{q_i} - \beta_i u r^{p_i}z^{q_i-2/3} + (\kappa_1+\kappa_2)p_i r^{p_i-2}z^{q_i-2/3}\\
&\qquad+ p_i(p_i-2)(\kappa_1 u^2+\kappa_2 v^2)r^{p_i-4}z^{q_i-2/3}\\
&\leq -\gamma p_i r^{p_i}z^{q_i} - \beta_i u r^{p_i}z^{q_i-2/3} + (\kappa_1+\kappa_2)p_i(1+|p_i-2|)r^{p_i-2}z^{q_i-2/3}.
\end{align*}
Since $Cu\geq |v|$ and $r\geq r_*$ on $\R_0$, we have that 
\begin{align*}
-  u r^{p_i} \leq -  \frac{ r^{p_i+1}}{\sqrt{1+C^2}} \qquad \text{ and } \qquad
(\kappa_1+\kappa_2)p_i(1+|p_i-2|)r^{p_i-2}\leq  \frac{c}{r_*^3} r^{p_i+1},
\end{align*}
on $\R_0$.  
In the above $c=c(\kappa_1,\kappa_2, p_i)>0$. By picking $r_*$ large enough, the estimate~\eqref{eqn:R0est1} readily follows.

To check the claimed boundary property in $\R_0$, note that for any $(u,v,1) \in \R_0$ we have 
\begin{align*}
\Q(\f_{0,1}+ \f_{0,2})(u,v,1)=\big(\frac{u}{3z}\partial_u+\frac{v}{3z}\partial_v-\partial_z\big)\big(\f_{0,1}+\f_{0,2}\big)(u,v,1)&=-(q_2-\tfrac{1}{3} p_2)r^{p_2}+\tfrac{1}{3} p_1r^{p_1}\\
&\leq -(q_2 -\tfrac{1}{3}p_2 -\tfrac{1}{3 r_*^{p_2-p_1}}p_1) r^{p_2}.  
\end{align*}
Recalling that $q_2>\frac{1}{3}p_2$ and $p_2>p_1$ by Choice~\ref{cond:pq}, picking $r_*$ large enough, the second estimate~\eqref{eqn:R0est2} follows. 
\end{proof}

Next, we establish similar estimates for $\f_{1,i}$ in region $\R_1$.  
\begin{lemma} \label{lem:R1:Lyapunov} Let $\f_{1,i}$, $i=1,2,$ be given as in~\eqref{eqn:R1:phi1}. Then for all $r_*$ large enough we have
\begin{align}
\label{eqn:R1est1}
\L \f_{1,i}(u,v,z) \leq - \tfrac{c_{1,i}}{2}r^{p_i+1} z^{q_i} \Big|\frac{r}{v}\Big|^{\alpha_i},
\end{align}
for all $(u,v,z) \in \R_1$ and
\begin{align}
\label{eqn:R1est2}
\Q(\f_{1,1}+ \f_{1,2})(u,v,1) \leq -(\tfrac{q_2}{2}-\tfrac{p_2}{6})r^{p_2},\end{align}
for all $(u,v,1) \in \R_1$.
\end{lemma}
\begin{proof}

We first establish~\eqref{eqn:R1est1}.  In view of~\eqref{generator:L:uvz} and \eqref{eqn:R1:T1}, we see that on $\R_1$
\begin{align*}
\L \f_{1,i}& = z^{-2/3}\T_1\f_{1,i}+(\L-z^{-2/3}\T_1)\f_{1,i}\\
&=z^{-2/3}\big(\T_1\f_{1,i}+\kappa_1\partial_{u}^2\f_{1,i}+\kappa_2\partial_{v}^2\f_{1,i}\big)-\gamma u\partial_u\f_{1,i}-\gamma v\partial_v\f_{1,i}.
\end{align*}
Recalling the boundary problem~\eqref{eqn:R1:T1:boundaryproblem}, we readily have 
\begin{align} \label{eqn:R1:T1:RHS}
\T_1\f_{1,i}=-c_{1,i}r^{p_i+1}z^{q_i}\Big|\frac{r}{v}\Big|^{\alpha_i},
\end{align} 
on $\R_1$. 
Concerning the first term on the righthand side of~\eqref{eqn:R1:phi1}, we see that on $\R_1$
\begin{equation} \label{eqn:R1:u.partial_u+v.partial_v.r^p.z^q}
\begin{aligned}
\MoveEqLeft[10](-\gamma u\partial_u-\gamma v\partial_v)\frac{r^{p_i}z^{q_i}}{(C^{-2}+1)^{\beta_i/2}}\Big|\frac{r}{v}\Big|^{\beta_i}=-\frac{\gamma\, p_i}{(C^{-2}+1)^{\beta_i/2}}r^{p_i}z^{q_i}\Big|\frac{r}{v}\Big|^{\beta_i},
\end{aligned}
\end{equation}
and
\begin{align*}
\partial_{u}^2\big(r^{p_i}\Big|\frac{r}{v}\Big|^{\beta_i}\big)& = (p_i+\beta_i) r^{p_i-2}\Big|\frac{r}{v}\Big|^{\beta_i} +(p_i+\beta_i)(p_i+\beta_i-2) r^{p_i-4}\Big|\frac{r}{v}\Big|^{\beta_i}u^2,\\
\partial_{v}^2\big(r^{p_i}\Big|\frac{r}{v}\Big|^{\beta_i}\big)
& = (p_i+\beta_i) r^{p_i-2}\Big|\frac{r}{v}\Big|^{\beta_i} +(p_i+\beta_i )(p_i+\beta_i -2) r^{p_i-4}\Big|\frac{r}{v}\Big|^{\beta_i}v^2\\
&\qquad +\beta_i(\beta_i+1) r^{p_i}\Big|\frac{r}{v}\Big|^{\beta_i}\frac{1}{v^2}  -2 (p_i+\beta_i )\beta_i r^{p_i-2}\Big|\frac{r}{v}\Big|^{\beta_i}.
\end{align*}
Combining the previous two identities, we infer the existence of a constant $c=c(h)>0$ such that on $\R_1$
\begin{align*}
\big(\kappa_1\partial_{u}^2+\kappa_2\partial_{v}^2\big)r^{p_i}\Big|\frac{r}{v}\Big|^{\beta_i}\leq c (\kappa_1+\kappa_2)\Big(r^{p_i-2}+\frac{r^{p_i}}{v^2}\Big)\Big|\frac{r}{v}\Big|^{\beta_i}.
\end{align*}
We note that in $\R_1$, $|u|\leq C|v|$ implying $v^2\geq r^2/(C^2+1)$. Together with $r \geq r_*$, it follows from the above inequality that
\begin{align*}
\big(\frac{\kappa_1}{z^{2/3}}\partial_{u}^2+\frac{\kappa_2}{z^{2/3}}\partial_{v}^2\big)\frac{r^{p_i}z^{q_i}}{(C^{-2}+1)^{\beta_i/2}}\Big|\frac{r}{v}\Big|^{\beta_i} &\leq c(\kappa_1+\kappa_2)C^2 r^{p_i-2}z^{q_i-2/3}\Big|\frac{r}{v}\Big|^{\beta_i}\\
& \leq \frac{c}{r_*^3}r^{p_i+1}z^{q_i-2/3}\Big|\frac{r}{v}\Big|^{\beta_i}.
\end{align*}
which is clearly subsumed to the right-hand side of~\eqref{eqn:R1:T1:RHS} thanks to condition $\alpha_i>\beta_i$, cf.~\eqref{cond:alphai.betai}, and the facts that $|r/v|\geq 1$ and $r_*$ is large.

For the second term on the righthand side of~\eqref{eqn:R1:phi1}, a short calculation shows that on $\R_1$
\begin{equation} \label{eqn:R1:u.partial_u+v.partial_v.r^p.z^q.int_0^T_1dt}
\begin{aligned}
\MoveEqLeft[5](-\gamma u\partial_u-\gamma v\partial_v)c_{1,i} r^{p_i}z^{q_i}\Big|\frac{r}{v}\Big|^{\beta_i}\int_{u/|v|}^{1/C}(t^2+1)^{\frac{\alpha_i-\beta_i-1}{2}}\d t \\&=- \gamma\, c_{1,i}\, p_i\, r^{p_i}z^{q_i}\Big|\frac{r}{v}\Big|^{\beta_i}\int_{u/|v|}^{1/C}(t^2+1)^{\frac{\alpha_i-\beta_i-1}{2}}\d t.
\end{aligned}
\end{equation}
Concerning the second-derivative terms, we also have the following estimate on $\R_1$ for $c=c(\kappa_1,\kappa_2,C,p_i)$
\begin{align*}
\MoveEqLeft[5]z^{-2/3}\big(\kappa_1\partial_{u}^2+\kappa_2\partial_{v}^2\big)c_{1,i}\, r^{p_i}z^{q_i}\Big|\frac{r}{v}\Big|^{\beta_i}\int_{u/|v|}^{1/C}(t^2+1)^{\frac{\alpha_i-\beta_i-1}{2}}\d t\\
&\le c\,c_{1,i} r^{p_i-2}z^{q_i-2/3}\Big|\frac{r}{v}\Big|^{\alpha_i} ,
\end{align*}
where in the above estimate, we have employed the facts that for every $(u,v,z)\in \R_1$, $r^2/|v|^2\le C^2+1$ and that 
\begin{equation} \label{ineq:int_(-C)^(1/C)(t^2+1)dt}
\int_{u/|v|}^{1/C}(t^2+1)^{\frac{\alpha_i-\beta_i-1}{2}}\d t\leq 2C.
\end{equation}
Using a similar reasoning as before, since $r\geq r_*$, it holds that
\begin{align*}
z^{-2/3}\big(\kappa_1\partial_{u}^2+\kappa_2\partial_{v}^2\big)c_{1,i}\, r^{p_i}z^{q_i}\Big|\frac{r}{v}\Big|^{\beta_i}\int_{u/|v|}^{1/C}(t^2+1)^{\frac{\alpha_i-\beta_i-1}{2}}\d t\le \frac{c}{r_*^3} r^{p_i+1}z^{q_i-2/3}\Big|\frac{r}{v}\Big|^{\alpha_i}.
\end{align*}
Collecting these estimates we arrive at~\eqref{eqn:R1est1} for all $r_*$ large enough. 

To obtain the second estimate~\eqref{eqn:R1est2}, using the identities~ \eqref{eqn:R1:u.partial_u+v.partial_v.r^p.z^q}, \eqref{eqn:R1:u.partial_u+v.partial_v.r^p.z^q.int_0^T_1dt}, we find that for all $(u,v,1) \in \R_1$: 
\begin{align*}
&\Q(\f_{1,1}+ \f_{1,2})(u,v,1)\\
 &=\big(\frac{u}{3z}\partial_u+\frac{v}{3z}\partial_v-\partial_z\big)\big(\f_{1,1}+\f_{1,2}\big) \big|_{\{z=1\}}\\
&= -(q_2-\tfrac{1}{3}p_2) \frac{ r^{p_2}}{(C^{-2}+1)^{\beta_2/2}}\Big|\frac{r}{v}\Big|^{\beta_2}  -(q_2-\tfrac{1}{3}p_2)c_{1,2}\, r^{p_2}\Big|\frac{r}{v}\Big|^{\beta_2}\int_{u/|v|}^{1/C}(t^2+1)^{\frac{\alpha_2-\beta_2-1}{2}}\d t\\
&\qquad+ \tfrac{1}{3}p_1\frac{r^{p_1}}{(C^{-2}+1)^{\beta_1/2}}\Big|\frac{r}{v}\Big|^{\beta_1}  +\tfrac{1}{3}p_1c_{1,1}\, r^{p_1}\Big|\frac{r}{v}\Big|^{\beta_1}\int_{u/|v|}^{1/C}(t^2+1)^{\frac{\alpha_1-\beta_1-1}{2}}\d t\\
&\leq -\tfrac{1}{2}(q_2-\tfrac{1}{3} p_2)r^{p_2}.
\end{align*}
In the last implication above, we employed the fact that in $\R_1$, $1<r/|v|<C+1$ and $r\ge r_*$ which is taken to be sufficiently large. This establishes~\eqref{eqn:R1est2}, thus concluding the proof.
 
\end{proof}

We now turn to the region $\R_2$.

\begin{lemma} \label{lem:R2:Lyapunov} 
Let $\f_{2,i}$, $i=1,2,$ be given as in~\eqref{eqn:R2:phi2}. Then for all $r_*\gg C=C(h)>0$ large enough, \begin{align}
\label{eqn:R2est1}
\L \f_{2,i}(u,v,z) \leq - \tfrac{c_{2,i}}{2}|u|^{p_i+1}z^{q_i-2/3}\Big|\frac{u}{v}\Big|^{\alpha_i},
\end{align} 
for all $(u,v,z) \in \R_2$ and
\begin{align}
\label{eqn:R2est2}
\Q(\f_{2,1}+ \f_{2, i})(u,v,1) \leq - (\tfrac{q_2}{2}-\tfrac{p_2}{6} )|u|^{p_2},
\end{align}  
for all $(u,v,1) \in \R_2$.  
\end{lemma}
\begin{proof}

Recalling~\eqref{generator:L:uvz} and \eqref{def:T2}, we note that
\begin{align*}
\L=z^{-2/3}\big(\T_2+v^2\partial_u+\kappa_1\partial_{u}^2+\kappa_2\partial_{v}^2\big)-\gamma u \partial_u-\gamma v\partial_v.
\end{align*}
By~\eqref{eqn:R2:T2:boundaryproblem}, we readily see
\begin{align} \label{ineq:R2:z^(-2/3)T_2.phi_(2,i)=-c_(2,i)|u|^(p_i+1)z^(q_i-2/3)|u/v|^(alpha_i)}
z^{-2/3}\T_2\f_{2,i}=-c_{2,i}|u|^{p_i+1}z^{q_i-2/3}\Big|\frac{u}{v}\Big|^{\alpha_i}.
\end{align}
Note first that if $\alpha=\alpha_i$ or $\alpha=\beta_i$ we have on $\R_2$ 
\begin{align*}
(\L-z^{-2/3}\T_2)|u|^{p_i}z^{q_i}\Big|\frac{u}{v}\Big|^{\alpha}
&=-\gamma p_i|u|^{p_i}z^{q_i}\Big|\frac{u}{v}\Big|^{\alpha}+(p_i+\alpha)\text{sgn}(u)|u|^{p_i-1}z^{q_i-2/3}\Big|\frac{u}{v}\Big|^{\alpha}v^2\\
&\qquad +\kappa_1(p_i+\alpha)(p_i+\alpha-1)|u|^{p_i-2}z^{q_i-2/3}\Big|\frac{u}{v}\Big|^{\alpha}\\
&\qquad+\kappa_2\alpha(\alpha+1)\frac{|u|^{p_i}}{v^2}z^{q_i-2/3}\Big|\frac{u}{v}\Big|^{\alpha}\\
&=:-I_0+I_1+I_2+I_3.
\end{align*}
To bound $I_1$, we use the fact that in $\R_2$ we have $|v|\le |u|/C$ and $|u/v|>1$ so that since $\alpha_i > \beta_i$
\begin{align*}
I_1\leq (p_i+\alpha)|u|^{p_i-1}z^{q_i-2/3}\Big|\frac{u}{v}\Big|^{\alpha}v^2\le \frac{p_i+\alpha}{C^2}|u|^{p_i+1}z^{q_i-2/3}\Big|\frac{u}{v}\Big|^{\alpha_i}.
\end{align*}
Also, $\sqrt{1+C^{-2}} |u|\ge r\ge r_*$ implies that
\begin{align*}
I_2&=\kappa_1(p_i+\alpha)(p_i+\alpha-1)|u|^{p_i-2}z^{q_i-2/3}\Big|\frac{u}{v}\Big|^{\alpha}\le \frac{c}{r_*^3}|u|^{p_i+1}z^{q_i-2/3}\Big|\frac{u}{v}\Big|^{\alpha_i}.
\end{align*}
Concerning $I_3$, we note that $|u|v^2\ge \eta_*^2=C^2\kappa_2$, cf.~\eqref{cond:r*.eta*}, which yields the bound
\begin{align*}
I_3=\kappa_2\alpha(\alpha+1)\frac{|u|^{p_i}}{v^2}z^{q_i-2/3}\Big|\frac{u}{v}\Big|^{\alpha}\le \frac{\alpha(\alpha+1)}{C^2}|u|^{p_i+1}z^{q_i-2/3}\Big|\frac{u}{v}\Big|^{\alpha_i}.
\end{align*}
Recalling the form of $\f_{2,i}$ in~\eqref{eqn:R2:phi2} and the choice of constants~\eqref{eqn:R2:constantB2}, we collect these estimates to find that 
\begin{align*}
(\L-z^{-2/3}\T_2)\f_{2,i} \leq \frac{c (A_{2,i}+ c_{2,i})}{C^2} |u|^{p_i+1} z^{q_i-2/3} \bigg|\frac{u}{v}\bigg|^{\alpha_i},
\end{align*}
where $c>0$ is a constant that only depends on $h$.  Recalling $A_{2,i}$ as in~\eqref{eqn:R2:constantA2} and  $c_{2,i}=c_{1,i}/2$ by~\eqref{cond:R2:c2}, we note that for all $C=C(h)>0$ sufficiently large
\begin{align}\label{ineq:R2:A_2<1+c_1.C^(alpha-beta)}
A_{2,i}\le 2 +\frac{2 c_{2,i}}{(\alpha_i - \beta_i)}C^{\alpha_i-\beta_i}. 
\end{align}
Thus, since $\alpha_i, \beta_i \in (0,1)$, for all $C=C(h)>0$ large enough, we obtain the first claimed estimate~\eqref{eqn:R2est1}. 

To check the second estimate~\eqref{eqn:R2est2}, note that in $\R_2$, recall $|u|\geq C|v|$ and thus $(1+C^{-2})|u|^2\geq r^2\ge r_*^2$. It follows that on $\{z=1\}$, we also have the asymptotic bound on $\R_2$ for all $r_*>0$ large enough
\begin{align*}
\Q\f_2(u,v,1)=\big(\frac{u}{3z}\partial_u+\frac{v}{3z}\partial_v-\partial_z\big)\big(\f_{2,1}+\f_{2,2}\big) \big|_{\{z=1\}}\leq  -\tfrac{1}{2}(q_2-\tfrac{1}{3} p_2)|u|^{p_2}.
\end{align*}
This gives~\eqref{eqn:R2est2}.

\end{proof}

The final subregion in the outer region $\RO$ is the most challenging since, compared with the previous regions, the formulas for the local functions $\f_{3,i}$ are not explicit due to the presence of the Laplace transforms in~\eqref{eqn:R3:phi3}.  To deal with these terms, however, we appeal to~\cite[Lemma 7.4]{herzog2015noise} which analyzes various properties of these functions.           

To help simplify notation, let 
\begin{equation} \label{eq:R3:G:mean}
G_s(\eta)=\E_\eta\big[e^{sT_3}\big].
\end{equation}
By~\cite[Lemma 7.4]{herzog2015noise}, so long as $0<s<h+\tfrac{3}{2}$, $\eta \mapsto G_s(\eta)\in C^{\infty}([-\eta_*,\eta_*])$ satisfies the following boundary-valued equation
\begin{equation} \label{eqn:G}
\begin{cases}\kappa_2 G''_s(\eta)+(h+\tfrac{3}{2})\eta G'_s(\eta)+s\, G_s(\eta) = 0,\\
G_s(-\eta_*) = G_s(\eta_*) =1.
\end{cases}
\end{equation}
In view of the choice~\eqref{cond:alphai.betai}, we see that $\alpha_i$ and $\beta_i$ satisfy $0<\beta_i<\alpha_i<1$, implying $\gamma_i$ and $\gammah_i$ as in~\eqref{cond:R3:gamma_i.gammahat_i} are strictly less than $h+\frac{3}{2}$. It follows that $G_{\gamma_i}(\eta)=\E_\eta\big[e^{\gamma_i T_3}\big]$ and $G_{\gammah_i}(\eta)=\E_\eta\big[e^{\gammah_i T_3}\big]$ are both smooth and satisfy~\eqref{eqn:G}.  We can thus express the formula~\eqref{eqn:R3:phi3} as
\begin{equation}\label{eqn:R3:phi3:simplified}
\begin{aligned}
\f_{3,i}(u,v,z)&=A_{3,i}|u|^{p_i+\frac{3}{2}\beta_i}z^{q_i}G_{\gamma_i}(\eta)+|u|^{p_i+\frac{3}{2}\alpha_i}z^{q_i}\Big[B_{3,i}G_{\gammah_i}(\eta)+ C_{3,i}\big(G_{\gammah_i}(\eta)-1\big)\Big].
\end{aligned}
\end{equation}
In order to show that $\f_{3,i}$ satisfies the requisite bounds in $\R_3$ with respect to $\L$ and $\Q$, we will make use of the following result concerning the analysis of the function $G_s$ defined above, the proof of which is deferred until Appendix~\ref{sec:auxiliary-result}.
\begin{lemma} \label{lem:R3:G} Let $G_s(\eta)$ be defined as in~\eqref{eq:R3:G:mean} and $0<s<h+3/2$.  Then for every $\eta\in[-\eta_*,\eta_*]$ 
\begin{align} \label{ineq:R3:eta.G'<0}
\eta G'_s(\eta) \leq 0.
\end{align}
Moreover, if $\eta_*$ satisfies condition~\eqref{cond:r*.eta*}; that is, $\eta_*^2=\kappa_2 C^2$, then there exists a constant $c(s,h)>0$ such that for all $C$ large enough
\begin{align}
\sup_{\eta\in[-\eta_*,\eta_*]}G_s(\eta)&=c(s,h)\, C^{\frac{s}{h+3/2}+1},\label{ineq:R3:sup.G}\\
\sup_{\eta\in[-\eta_*,\eta_*]}|G'_s(\eta)|&= \frac{c(s,h)}{\sqrt{\kappa_2}} C^{\frac{s}{h+3/2}+1},\label{ineq:R3:sup.G'(eta)}
\\
\sup_{\eta\in[-\eta_*,\eta_*]}| G''_s(\eta)|&= \frac{c(s,h)}{\kappa_2} C^{\frac{s}{h+3/2}+1}, \label{ineq:R3:sup.G''(eta)}
\\
\text{and}\qquad G'_s(\pm\eta_*)&= \mp\frac{s}{h+3/2}\eta_*^{-1}+o\big(\eta_*^{-1}\big).\label{ineq:R3:G'(eta_*)}
\end{align}
\end{lemma}

We now assert the Lypaunov properties in $\R_3$ in the following result.
\begin{lemma}\label{lem:R3:Lyapunov} Let $\f_{3,i}$, $i=1,2,$ be defined as in~\eqref{eqn:R3:phi3}. Then for all $r_*\gg C=C(h)>0$ large enough
\begin{align}
\label{eqn:R3est1}
\L \f_{3,i}(u,v,z) \leq -\tfrac{c_{3,i}}{2}|u|^{p_i+1+\tfrac{3}{2}\alpha_i}z^{q_i-2/3},
\end{align}
for all $(u,v,z) \in \R_3$ and
\begin{align}
\label{ineq:R3est2}
\Q(\f_{3,1}+ \f_{3,2})(u,v,1)\le - \tfrac{1}{2}( q_2-\tfrac{1}{3}p_2-\tfrac{1}{2}\alpha_2)B_{3,2}|u|^{p_2+\frac{3}{2}\alpha_2},
  \end{align}
  for all $(u,v,1)\in\R_3$. 
\end{lemma}
\begin{proof}
We begin by showing~\eqref{eqn:R3est1}. 
In view of~\eqref{generator:L:uvz} and~\eqref{eqn:R3:T:hat3}, we can rewrite $\L$ as 
\begin{align*}
\L=z^{-2/3}\big(|u|\widetilde{\T}_3+v^2\partial_u+\kappa_1\partial_{u}^2\big)-\gamma u\partial_u-\gamma v\partial_v.
\end{align*}
Recalling~\eqref{eqn:R3:T.hatphi3:boundary.problem}, we find
\begin{align}\label{eqn:R3:T.hat3phi3}
z^{-2/3}|u|\widetilde{\T}_3\f_{3,i}=-c_{3,i}|u|^{p_i+1+\frac{3}{2}\alpha_i}z^{q_i-2/3}.
\end{align}
For the second and third terms on the right-hand side of~\eqref{eqn:R3:phi3:simplified}, it follows that
\begin{align*}
\MoveEqLeft[1](\L-z^{-2/3}|u|\widetilde{\T}_3)\big( |u|^{p_i+\frac{3}{2}\alpha_i}z^{q_i}\big[B_{3,i}G_{\gammah_i} (\eta)+ C_{3,i}\big(G_{\gammah_i} (\eta)-1\big)\big]\big)&\\
&=-(p_i+\tfrac{3}{2}\alpha_i)\big(\gamma|u|^{p_i+\frac{3}{2}\alpha_i}+|u|^{p_i-1+\frac{3}{2}\alpha_i}v^2z^{-2/3}\big)z^{q_i}\big[B_{3,i}G_{\gammah_i} (\eta)+ C_{3,i}\big(G_{\gammah_i} (\eta)-1\big)\big]\\
&\qquad -\tfrac{3}{2}\gamma(B_{3,i}+C_{3,i})|u|^{p_i+\frac{3}{2}\alpha_i}z^{q_i}G'_{\gammah_i} (\eta)\eta\\
&\qquad -\tfrac{1}{2}(B_{3,i}+C_{3,i})|u|^{p_i+\frac{3}{2}\alpha_i}z^{q_i-2/3}G'_{\gammah_i} (\eta)|u|^{-1/2}v^3\\
&\qquad +\kappa_1(p_i+\tfrac{3}{2}\alpha_i)(p_i+\tfrac{3}{2}\alpha_i-1)|u|^{p_i-2+\frac{3}{2}\alpha_i}z^{q_i-2/3}\big[B_{3,i}G_{\gammah_i} (\eta)+ C_{3,i}\big(G_{\gammah_i} (\eta)-1\big)\big]\\
&\qquad +\kappa_1(p_i+\tfrac{3}{2}\alpha_i)(B_{3,i}+C_{3,i})|u|^{p_i-1+\frac{3}{2}\alpha_i}z^{q_i-2/3}G'_{\gammah_i} (\eta)|u|^{-1/2}v\\
&\qquad +\tfrac{1}{4}\kappa_1(B_{3,i}+C_{3,i})|u|^{p_i+\frac{3}{2}\alpha_i}z^{q_i-2/3}G''_{\gammah_i} (\eta)|u|^{-1}v^2\\
&\qquad -\tfrac{\kappa_1}{4}(B_{3,i}+C_{3,i})|u|^{p_i+\frac{3}{2}\alpha_i}z^{q_i-2/3}G'_{\gammah_i} (\eta)|u|^{-3/2}v\\
&=-I_0-I_1-I_2+I_3+I_4+I_5-I_6.
\end{align*}
We proceed to show that each $I_k$, $k=1,\dots,6$ is dominated by the right-hand side of~\eqref{eqn:R3:T.hat3phi3}.  To see this, we first note that by Lemma~\ref{lem:R3:G}, cf.~\eqref{ineq:R3:eta.G'<0}, $\eta G'_{\gammah_i} (\eta)\le 0$. It follows that $I_4\le 0$, which we can neglect.  Also note that $-I_0\leq 0$. since $G_s\geq 1$. To estimate $I_1$, we invoke~\eqref{ineq:R3:sup.G'(eta)} together with $\sqrt{1+C^{-2}} |u|\ge r\ge r_*$ to infer the bound
\begin{align*}
|I_1|&\le c(\gammah_i ,h)\gamma(B_{3,i}+C_{3,i}) C^{\frac{\gammah_i}{h+3/2}+2} |u|^{p_i+\frac{3}{2}\alpha_i}z^{q_i}\\
&\le c(\gammah_i ,h)\gamma(B_{3,i}+C_{3,i})C^{\frac{\gammah_i}{h+3/2}+2}\frac{1}{r_*}|u|^{p_i+\frac{3}{2}\alpha_i+1}z^{q_i-2/3}.
\end{align*}
Likewise, 
\begin{align*}
|I_2|&=\tfrac{1}{2}(B_{3,i}+C_{3,i})|u|^{p_i+\frac{3}{2}\alpha_i-2}z^{q_i-2/3}G'_{\gammah_i}(\eta)(|u|^{1/2}|v|)^3\\
&\le  c(\gammah_i ,h)\kappa_2(B_{3,i}+C_{3,i}) C^{\frac{\gammah_i}{h+3/2}+4} |u|^{p_i+\frac{3}{2}\alpha_i-2}z^{q_i-2/3}\\
&\le  c(\gammah_i ,h)\kappa_2(B_{3,i}+C_{3,i})C^{\frac{\gammah_i}{h+3/2}+4}\frac{1}{r_*^3}|u|^{p_i+\frac{3}{2}\alpha_i+1}z^{q_i-2/3},
\end{align*}
and
\begin{align*}
|I_6|&\le c(\gammah_i ,h)\kappa_1(B_{3,i}+C_{3,i}) C^{\frac{\gammah_i}{h+3/2}+2} |u|^{p_i+\frac{3}{2}\alpha_i-2}z^{q_i-2/3}\\
&\le c(\gammah_i ,h)\kappa_1(B_{3,i}+C_{3,i}) C^{\frac{\gammah_i}{h+3/2}+2}\frac{1}{r_*^3}|u|^{p_i+\frac{3}{2}\alpha_i+1}z^{q_i-2/3}.
\end{align*}
Similarly, to estimate $I_3$ and $I_5$, we employ~\eqref{ineq:R3:sup.G} and~\eqref{ineq:R3:sup.G''(eta)}, respectively, to obtain the bounds
\begin{align*}
|I_3|&\le c(p_i,q_i,\alpha_i)(B_{3,i}+C_{3,i})C^{\frac{\gammah_i}{h+3/2}+1}\frac{1}{r_*^3} |u|^{p_i+\frac{3}{2}\alpha_i+1}z^{q_i-2/3},\\
\text{and}\qquad|I_5|&\le c(p_i,q_i,\alpha_i)(B_{3,i}+C_{3,i})C^{\frac{\gammah_i}{h+3/2}+2} \frac{1}{r_*^3} |u|^{p_i+\frac{3}{2}\alpha_i+1}z^{q_i-2/3}.
\end{align*}
Recalling $B_{3,i}$ and $C_{3,i}$ as in~\eqref{cond:R3:A3.B3} and $c_{3,i}$ in~\eqref{cond:R3:c3}, we observe that
\begin{align*}
B_{3,i}=\frac{B_{2,i}}{\eta_*^{\alpha_i}}=\frac{c_{2,i}}{(\alpha_i-\beta_i)\eta_*^{\alpha_i}}=\frac{2}{\alpha_i-\beta_i}c_{3,i}\quad\text{and}\quad C_{3,i}=\frac{1}{\gammah_i}c_{3,i}.
\end{align*}
It follows that all $I_k,$ $k=1,\dots,6$ above are clearly dominated by $-c_{3,i}|u|^{p_i+1+\frac{3}{2}\alpha_i}z^{q_i-2/3}$ for $r_*$ large enough. This finishes the estimate for the second and third terms on the right-hand side of~\eqref{eqn:R3:phi3:simplified}.

Considering the first term on the right-hand side of~\eqref{eqn:R3:phi3:simplified}, a similar argument as above also yields the bound
\begin{align*}
\MoveEqLeft[1](\L-z^{-2/3}|u|\widehat{\T}_3)\big( A_{3,i} |u|^{p_i+\tfrac{3}{2}\beta_i}z^{q_i}G_{\gamma_i}(\eta)\big)=O\Big(\frac{1}{r_*}\Big)|u|^{p_i+1+\frac{3}{2}\beta_i}z^{q_i-2/3}.
\end{align*}
Since $|u|\ge r_*/2$ and $\beta_i<\alpha_i$, we readily have that $|u|^{p_i+1+\frac{3}{2}\beta_i}$ is dominated by $|u|^{p_i+1+\frac{3}{2}\alpha_i}$ for all $r_*$ large enough.  This finishes the proof of~\eqref{eqn:R3est1}.

We next turn to checking that the boundary terms have the requisite bound.  Note that 
\begin{align*}
\big(\frac{u}{3z}\partial_u+\frac{v}{3z}\partial_v-\partial_z\big)\f_{3,2}\big|_{\{z=1\}}&=-(q_2-\tfrac{1}{3}p_2-\tfrac{1}{2}\alpha_2)|u|^{p_2+\frac{3}{2}\alpha_2}\big[B_{3,2}G_{\gammah_2}(\eta)+ C_{3,2}\big(G_{\gammah_2}(\eta)-1\big)\big]\\
&\qquad -(q_2-\tfrac{1}{3}p_2-\tfrac{1}{2}\beta_2)|u|^{p_2+\frac{3}{2}\beta_2}A_{3,2}G_{\gamma_2}(\eta)\\
&\qquad+\tfrac{1}{2}(B_{3,2}+ C_{3,2})|u|^{p_2+\frac{3}{2}\alpha_2}G'_{\gammah_2}(\eta)\eta+\tfrac{1}{2}A_{3,2}|u|^{p_2+\frac{3}{2}\alpha_2}G'_{\gamma_2}(\eta)\eta.
\end{align*}
We recall from~\eqref{ineq:R3:eta.G'<0} that $G'_{\gammah_2}(\eta)\eta$ and $G'_{\gamma_2}(\eta)\eta$ are negative. Concerning the first two terms, by Choice~\ref{cond:pq}, we see that 
\begin{align*}
q_2-\tfrac{1}{3}p_2-\tfrac{1}{2}\beta_2>q_2-\tfrac{1}{3}p_2-\tfrac{1}{2}\alpha_2>0,\quad
\text{and}\quad
\alpha_2>\beta_2>0.
\end{align*}
Hence, for $|u|\ge r_*/2$ in $\R_3$ and $r_*$ large enough 
\begin{align*}
\Q\f_{3,2}(u,v,1)\le -(q_2-\tfrac{1}{3}p_2-\tfrac{1}{2}\alpha_2)B_{3,2}|u|^{p_2+\frac{3}{2}\alpha_2}.
\end{align*}
Likewise, we have
\begin{align*}
\Q\f_{3,1}(u,v,1)=O(|u|^{p_1+\frac{3}{2}\alpha_1}).
\end{align*}
Since $p_1+\frac{3}{2}\alpha_1<p_2+\frac{3}{2}\alpha_2$, we readily obtain the inequality~\eqref{ineq:R3est2}. This finishes the proof.

\end{proof}

\subsection{Lyapunov bounds for $\RI$} \label{sec:inner:lyapprop} In this subsection, we first show that $\g$ given by~\eqref{def:psi:average.Lyapunov} is indeed an averaging Lyapunov function for $\A$.  This result is at least a partial confirmation of how dissipation is generated through averaging effects when $z\approx 0$.  Although we do not need the result below in order to prove Proposition~\ref{prop:Lyapunov:xyz}, we will use some of the estimates contained in its proof to obtain the needed Lyapunov-type estimates involving $\L$ and $\Q$ afterwards.        
\begin{lemma} \label{lem:averaging.Lyapunov}
There exists a constant $c_*>0$ sufficiently small such that for all $|\kappa_1/\kappa_2-1|<c_*$, $\g$ given by~\eqref{def:psi:average.Lyapunov} is an averaging Lyapunov function as in Definition~\ref{def:averaging.Lyapunov} for $\A$, the operator as in~\eqref{def:A}.
\end{lemma}
\begin{proof} We first consider $\g_1$ as in~\eqref{def:psi_1} and note that for any $J>0$, $\f_4 \in C^1(\RR^2;\RR)$.  Furthermore, $\g_{1i}$, $i=1,2$, are $C^2$ on their closed domains of definition. Also, note that 
 \begin{align*}
 &\A \g_{11}(u, v )= \alpha_h u  r ^{2} -(\kappa_1+\kappa_2),  \quad  r ^{2} \leq J, \\
  &\A \g_{12}(u, v ) =\alpha_h J u + J (\kappa_2-\kappa_1) \frac{ v ^{2}-u^{2}}{ r ^{4}}, \quad  r ^{2}\geq J.  
 \end{align*}  
Note that on the set $\{ r ^{2} \geq J\}$ we have
 \begin{align*}
 \A \g_{12}(u, v ) \leq \alpha_h J u + |\kappa_2-\kappa_1|\le \alpha_h J u + \kappa_1c_*. 
 \end{align*}  
On the other hand on $\{ r^2 \leq J\}$, we have
\begin{align}
\label{eqn:avest1}
\A \g_{11} \leq \alpha_h J u - \frac{\kappa_1+\kappa_2}{2}.
\end{align}  
Indeed, the inequality claimed above is equivalent to
$$\alpha_h u( r ^{2}-J) \le \frac{\kappa_1+\kappa_2}{2},$$
which in turn is a consequence of the following estimate
$$\alpha_h  r (J- r ^{2}) \le \alpha_h J^{3/2}\le \frac{\kappa_1+\kappa_2}{2},$$
by the choice of $J$ as in~\eqref{cond:J.m}.

To obtain the needed estimate for $\g_2$ as in~\eqref{def:psi_2}, note that 
\begin{align*}
-m^{-1} \A \g_2 &=\frac{\alpha_hu^{2}+ v ^{2}}{ r ^{2}}\lambda_1+2(\kappa_2-\kappa_1)\frac{u(3 v ^{2}-u^2)}{ r ^6}\lambda_1\\
&\qquad-\frac{4\alpha_h}{J}u^{2}\lambda_1'+\frac{4}{J}(\kappa_1-\kappa_2)\frac{u(3 v ^{2}-u^2)}{ r ^4}\lambda_1' +\frac{16}{J^2}\frac{u(\kappa_1u^{2}+\kappa_2 v ^{2})}{ r ^{2}}\lambda_1''.
\end{align*}
On the set $ r ^{2}\ge J$, observe that $\lambda_1=1$ and $\lambda_1'=\lambda_1''=0$. Hence multiplying the above by $-1$ and using the choice of $J$ we find that when $r^2\ge J$ and $c_*<2$
\begin{align*}
m^{-1}\A\g_2&=-\frac{\alpha_hu^{2}+ v ^{2}}{ r ^{2}}-2(\kappa_2-\kappa_1)\frac{u(3 v ^{2}-u^2)}{ r ^6}\\
&\le -\alpha_h + \frac{6}{J^{3/2}}|\kappa_2-\kappa_1|= -\alpha_h + 12\alpha_h \frac{|\kappa_2-\kappa_1|}{\kappa_1+\kappa_2}\le -\alpha_h + 12\alpha_h \frac{c_*}{2-c_*}.
\end{align*} 
In view of the choice of $m$ in~\eqref{cond:J.m}, we obtain
$$ \A\g_2\le -\frac{m\alpha_h}{2}-\kappa_1c_*.$$
On the other hand, using the above calculations we can estimate $\A \g_2$ on the set $\{r^2\leq J\}$ by choosing $c_*>0$ small enough as follows: 
\begin{align*}
\A\g_2&\le m\Big(\frac{6 }{J^{3/2}}|\kappa_2-\kappa_1|+4\alpha_h |\lambda_1'|+\frac{12}{ J^{3/2}}|\kappa_2-\kappa_1||\lambda_1'|+\frac{16(\kappa_1+\kappa_2)}{J^{3/2}}|\lambda_1''|\Big)\le \frac{\kappa_1+\kappa_2}{4}.
\end{align*}  
Note that the size of the derivatives of $\lambda_1$ is independent of the choice of any parameters.

Collecting the above estimates we find that on $\{ r^2\geq J\}$ or on $\{ r^2 \leq J\}$
\begin{align}
\label{eqn:intlyap0}
\A\g=\A\g_1+\A\g_2\le \alpha_h Ju-\min\big\{\frac{m\alpha_h}{2},\frac{\kappa_1+\kappa_2}{4}\big\}.
\end{align} 
Applying the It\^{o}/Tanaka formula~\eqref{eqn:Tanaka} in Appendix~\ref{sec:Tanaka}, this in turn implies~\eqref{eqn:avgb} for
\begin{equation*} 
C_1=\alpha_h J,\quad\text{and}\quad\epsilon=\min\big\{\frac{m\alpha_h}{2},\frac{\kappa_1+\kappa_2}{4}\big\},
\end{equation*}
thus finishing the proof.  
\end{proof}
\begin{remark} \label{rem:psi.not.C^2}
Since the function $\psi$ is globally $C^1$ on $\RR^2$ (as opposed to globally continuous), note that there are no additional boundary-flux terms that arise, via the It\^{o}/Tanaka formula~\eqref{eqn:Tanaka}, on the interfaces where $\psi$ is not $C^2$.  
\end{remark}

As a consequence of the previous proof, we collect the needed Lyapnov bounds for the interior function $\psi$.  

\begin{corollary} \label{corollary:psi}
There exist $c_*>0$ sufficiently small so that on either $\{ r^2 \leq J\}$ or on $\{ r^2 \geq J\}$
\begin{align}
\label{eqn:intlyap1}
\L \psi (u,v,z) &\leq \frac{u \alpha_h J}{z^{2/3}} - \min\bigg\{ \frac{m \alpha_h}{2z^{2/3}}, \frac{\kappa_1 + \kappa_2}{4z^{2/3}}\bigg\} + D_1,\\
\Q \psi(u,v,1) & \leq D_2,
\label{eqn:intlyap2} 
\end{align}
for some constants $D_1, D_2>0$.  
\end{corollary}

\begin{proof}
To obtain~\eqref{eqn:intlyap1}, note that by~\eqref{eqn:intlyap0} on $\{ r^2 \leq J \}$ for $c_*>0$ small enough
\begin{align*}
\L \psi &= \frac{\A\psi} {z^{2/3}} - \gamma u \partial_u \psi - \gamma v \partial_v \psi\\
& \leq \frac{\alpha_h J u}{z^{2/3}} -  \min\bigg\{ \frac{m \alpha_h}{2z^{2/3}}, \frac{\kappa_1 + \kappa_2}{4z^{2/3}}\bigg\}+ \gamma r^2 -\frac{\gamma m u }{(u^2+ v^2)}\lambda_1 + \frac{4m\gamma u}{J}\lambda_1' \\
& \leq \frac{\alpha_h J u}{z^{2/3}} -  \min\bigg\{ \frac{m \alpha_h}{2z^{2/3}}, \frac{\kappa_1 + \kappa_2}{4z^{2/3}}\bigg\} + D_{11},
\end{align*}
for some constant $D_{11}>0$.  A similar estimate holds on $\{ r^2 \geq J \}$ for some constant $D_{12}>0$.  Taking $D_1=\max\{ D_{11}, D_{12} \}$ produces~\eqref{eqn:intlyap1}.  The second estimate~\eqref{eqn:intlyap2} follows by direct calculation.       
\end{proof}

\section{Boundary-Flux Estimates} \label{sec:flux}

In order to conclude the proof of Proposition~\ref{prop:Lyapunov:xyz}, essentially all we have left to check is that the boundary-flux terms that arise in the It\^{o}/Tanaka formula~\eqref{eqn:Tanaka} applied to the outer function $\Phi$ have the appropriate sign.  Since the local functions $\f_{i,j}$ are $C^\infty$ on their closed domains of definition, the interfaces where this condition needs to be checked are $\R_0\cap \R_1$, $\R_1\cap \R_2$ and $\R_2\cap \R_3$ for a $C=C(h)>0$ and $r_*>0$ large enough.

 We first begin with the:
\subsection{ Boundary-flux on $\R_0\cap \R_1$ }
 
\begin{lemma} \label{lem:flux:R0R1}
Let $\f_{0,i}$ and $\f_{1,i}$, $i=1,2,$ be given as in~\eqref{form:R0:phi0} and \eqref{eqn:R1:phi1}, respectively. Then for all $C=C(h)>0$ large enough 
\begin{align} \label{ineq:flux:R0R1}
\big(\partial_u \f_{0,i}- \partial_u\f_{1,i}\big)\big|_{\R_0\cap \R_1} \le 0.
\end{align}
\end{lemma}
\begin{proof} 
We first see that 
\begin{align*}
z^{-q_i}\partial_u \f_{0,i}=p_iu r^{p_i-2}.
\end{align*}
On the the other hand, on $\R_0 \cap \R_1 =\{Cu=|v|\}$ we have
\begin{align*}
z^{-q_i}\partial_u \f_{1,i}&= (p_i +\beta_i)u r^{p_i-2}- c_{1,i}r^{p_i-1}\Big|\frac{r}{v}\Big|^{\alpha_i}.
\end{align*}
The flux condition~\eqref{ineq:flux:R0R1} thus holds on $\{Cu=|v|\}$ provided we have
\begin{align*}
\beta_i \frac{u}{r}\ge c_{1,i}\Big|\frac{r}{v}\Big|^{\alpha_i},
\end{align*}
on the set $\{ Cu =|v|\}$.  Note that this is equivalent to
\begin{align*}
c_{1,i}\le \frac{\beta_i}{(C^{-2}+1)^{\alpha_i/2}\sqrt{C^2+1}}.
\end{align*}
Recalling $c_{1,i}$ as in~\eqref{cond:R1:c1}, the above inequality is clearly satisfied for all $C=C(h)>0$ large enough. This finishes the proof. 
\end{proof}
\subsection{Boundary-flux on $\R_1 \cap \R_2$} \label{subsec:flux:R1R2} 
\begin{lemma} \label{lem:flux:R1R2}
Let $\f_{1,i}$ and $\f_{2,i}$, $i=1,2,$ be given as in~\eqref{eqn:R1:phi1} and \eqref{eqn:R2:phi2}, respectively. Then for all $C=C(h)>0$ large enough
\begin{align} \label{ineq:flux:R1R2}
\big(\partial_u \f_{1,i}- \partial_u\f_{2,i}\big)\big|_{\R_1\cap \R_2}\le 0.
\end{align}
\end{lemma}
\begin{proof} We have
\begin{align*}
z^{-q_i}\partial_u \f_{2,i}&=-A_{2,i}(p_i+\beta_i)|u|^{p_i-1}\Big|\frac{u}{v}\Big|^{\beta_i}-B_{2,i}(p_i+\alpha_i)|u|^{p_i-1}\Big|\frac{u}{v}\Big|^{\alpha_i}.
\end{align*}
Also, recalling $A_{2,i}$ and $B_{2,i}$ as in~\eqref{eqn:R2:constantA2} and~\eqref{eqn:R2:constantB2}, respectively, we see that  on the boundary $u=-C|v|$
\begin{align*}
|u|^{1-p_i}z^{-q_i}\partial_u \f_{2,i}\big|_{\{u=-C|v|\}}
&= -(p_i+\beta_i)C^{\beta_i}(C^{-2}+1)^{p_i/2}\\
&\qquad -c_{1,i}(p_i+\beta_i)C^{\beta_i}(C^{-2}+1)^{\frac{p_i+\beta_i}{2} }\int_{-C}^{1/C}\close(t^2+1 )^{\frac{\alpha_i-\beta_i-1}{2}}\d t\\
&\qquad-(\alpha_i-\beta_i)C^{\alpha_i}B_{2,i}.
\end{align*}

Similarly, we can calculate the partial derivative with respect to $u$ of $\f_{1,i}$ to obtain 
\begin{align*}
|u|^{1-p_i}z^{-q_i}\partial_u\f_{1,i}\big|_{\{u=-C|v|\}}
&=-(p_i+\beta_i)C^{\beta_i}(C^{-2}+1)^{p_i/2-1}\\
&\qquad-c_{1,i}(p_i+\beta_i)C^{\beta_i}(C^{-2}+1)^{\frac{p_i+\beta_i}{2}-1}\int_{-C}^{1/C}\close (t^2+1 )^{\frac{\alpha_i-\beta_i-1}{2}}\d t\\
&\qquad-c_{1,i}(C^{-2}+1)^{\frac{\alpha_i+p_i-1}{2}}C^{\alpha_i}.
\end{align*}
Recalling that $B_{2,i}=c_{2,i}/(\alpha_i-\beta_i)$, we have
\begin{align*}
|u|^{1-p_i}z^{-q_i}\big(\partial_u\f_{1,i}-\partial_u\f_{2,i}\big)\big|_{\{u=-C|v|\}}&=\frac{p_i+\beta_i}{C^2+1}C^{\beta_i}(C^{-2}+1)^{p_i/2}\\
&\qquad+c_{1,i}\frac{p_i+\beta_i}{C^2+1}C^{\beta_i}(C^{-2}+1)^{\frac{p_i+\beta_i}{2}}\int_{-C}^{1/C}\close (t^2+1 )^{\frac{\alpha_i-\beta_i-1}{2}}\d t\\
&\qquad-\Big[c_{1,i}(C^{-2}+1)^{\frac{\alpha_i+p_i-1}{2}}-c_{2,i}\Big]C^{\alpha_i}\\
&=I_1+I_2-I_3.
\end{align*}
For $C>0$ large enough, we recall $c_{1,i}$ as in~\eqref{cond:R1:c1} and observe that
\begin{align*}
I_1=\frac{1}{C^{2-\beta_i}}O(p_i+\beta_i)\quad \text{and}\quad I_2\le\frac{1}{C^{3-\alpha_i}}O(\beta_i(p_i+\beta_i)).
\end{align*}
On the other hand, since $c_{2,i}=c_{1,i}/2$, it is straightforward to see that
\begin{align*}
I_3\ge \tfrac{1}{4}c_{1,i}C^{\alpha_i}=\tfrac{1}{8}\beta_iC^{\alpha_i-1},
\end{align*}
which together with the above asymptotic estimates imply~\eqref{ineq:flux:R1R2}. The proof is thus finished.

\end{proof}

\subsection{Boundary-flux on $\R_2\cap \R_3$} \label{subsec:flux:R2R3} Finally, we check the last boundary-flux contribution has the appropriate sign.  
\begin{lemma} \label{lem:flux:R2R3}
Let $\f_{2,i}$ and $\f_{3,i}$, $i=1,2,$ be given as in~\eqref{eqn:R2:phi2} and \eqref{eqn:R3:phi3}, respectively. Then for all $C=C(h)>0, \eta_*=\eta_*(C)>0, r_*=r_*(C,\eta_*)>0$ large enough
\begin{align} \label{ineq:flux:R2R3}
\big(\partial_v\f_{2,i}-\partial_ v \f_{3,i}\big)\big|_{\{|u|^{1/2} v = \eta_*\}}\le 0\quad\text{and}\quad\big(\partial_v\f_{3,i}-\partial_ v \f_{2,i}\big)\big|_{\{|u|^{1/2}v=-\eta_*\}}\le 0.
\end{align}
\end{lemma}
\begin{proof}
We shall prove only the claimed bound for the part of the boundary $\R_2\cap \R_3$ where $|u|^{1/2} v=\eta_*$.  By symmetry, an analogous argument can be carried out when $|u|^{1/2} v=-\eta_*$.  First observe that 
\begin{align*}
z^{-q_i}\partial_v\f_{2,i}\big|_{\{|u|^{1/2}v=\eta_*\}}&=-\frac{\beta_i}{\eta_*^{\beta_i+1}}A_{2,i}|u|^{p_i+\frac{3}{2}\beta_i+\frac{1}{2}}-\frac{\alpha_i}{\eta_*^{\alpha_i+1}}B_{2,i}|u|^{p_i+\frac{3}{2}\alpha_i+\frac{1}{2}}.
\end{align*} 
On the other hand, recalling the expression for $\f_{3,i}$ in~\eqref{eqn:R3:phi3:simplified}, we have
\begin{align*}
z^{-q_i}\partial_v\f_{3,i}\big|_{\{|u|^{1/2}v=\eta_*\}}&=(B_{3,i}+C_{3,i})|u|^{p_i+\frac{3}{2}\alpha_i+\frac{1}{2}}G'_{\gammah_i} (\eta_*)+A_{3,i}|u|^{p_i+\frac{3}{2}\beta_i+\frac{1}{2}}G'_{\gamma_i} (\eta_*).
\end{align*}
In view of~\eqref{ineq:R3:G'(eta_*)} together with $\gamma_i=(h+\frac{3}{2})\beta_i$ and $A_{3,i}=A_{2,i}\eta_*^{-\beta_i}$, we have the following asymptotic behavior as $\eta_*\to\infty$
\begin{align*}
A_{3,i}|u|^{p_i+\frac{3}{2}\beta_i+\frac{1}{2}}G'_{\gamma_i} (\eta_*)=-\frac{\beta_i}{\eta_*^{\beta_i+1}}A_{2,i}|u|^{p_i+\frac{3}{2}\beta_i+\frac{1}{2}}+o\big(\eta_*^{-\beta_i-1}\big)|u|^{p_i+\frac{3}{2}\beta_i+\frac{1}{2}}.
\end{align*}
Also, recalling $\gammah_i=\beta_i+(h+\frac{1}{2})\alpha_i$,
\begin{align*}
(B_{3,i}+C_{3,i})|u|^{p_i+\frac{3}{2}\alpha_i+\frac{1}{2}}G'_{\gammah_i} (\eta_*)&=-\frac{\beta_i+(h+\frac{1}{2})\alpha_i}{(h+\frac{3}{2})\eta_*}(B_{3,i}+C_{3,i})|u|^{p_i+\frac{3}{2}\alpha_i+\frac{1}{2}}\\
&\qquad+o\Big(\frac{1}{\eta_*}\Big)(B_{3,i}+C_{3,i})|u|^{p_i+\frac{3}{2}\alpha_i+\frac{1}{2}}.
\end{align*}
Combining the above estimates, we observe that the claimed flux condition holds on $\{|u|^{1/2}v=\eta_*\}$ if 
\begin{align*}
\frac{\alpha_i}{\eta_*^{\alpha_i}}B_{2,i}-\frac{\beta_i+(h+\frac{1}{2})\alpha_i}{h+\frac{3}{2}}(B_{3,i}+C_{3,i})\ge c(B_{3,i}+C_{3,i}),
\end{align*}
for some positive constant $c$ independent of $\eta_*$. To see this, we recall that
$$B_{3,i}=\frac{B_{2,i}}{\eta_*^{\alpha_i}}=\frac{c_{2,i}}{(\alpha_i-\beta_i)\eta_*^{\alpha_i}}\quad\text{and}\quad C_{3,i}=\frac{c_{3,i}}{\beta_i+(h+\frac{1}{2})\alpha_i}=\frac{c_{2,i}}{2(\beta_i+(h+\frac{1}{2})\alpha_i)\eta_*^{\alpha_i}},$$
whence
\begin{align*}
B_{3,i}+C_{3,i}=\Big[\frac{1}{\alpha_i-\beta_i}+\frac{1}{2(\beta_i+(h+\frac{1}{2})\alpha_i)}\Big]\frac{c_{2,i}}{\eta_*^{\alpha_i}}.
\end{align*}
Also,
\begin{align*}
\frac{\alpha_i}{\eta_*^{\alpha_i}}B_{2,i}-\frac{\beta_i+(h+\frac{1}{2})\alpha_i}{h+\frac{3}{2}}(B_{3,i}+C_{3,i})&=\frac{c_{2,i}}{2(h+\frac{3}{2})\eta_*^{\alpha_i}}.
\end{align*}
We therefore conclude the flux condition~\eqref{ineq:flux:R2R3} on the boundary $|u|^{1/2} v = \eta_*$. This finishes the proof.  

\end{proof}

\section{The Proof of Theorem~\ref{thm:geometric.ergodicity}} \label{sec:proof-thm-ergodicity}

In Proposition~\ref{prop:Lyapunov:uvz}, stated and proven next, we establish a more global Lyapunov-type estimate for the function $\Phi$ in the variables $(u,v,z)$. As a consequence, upon changing back to $(x,y,z)-$coordinates, we will be able to establish Proposition~\ref{prop:Lyapunov:xyz}.
\begin{proposition} \label{prop:Lyapunov:uvz} Let $ \Phi$ be given by Definition~\ref{def:Phi} and $\tau_R$ be the stopping time given by
\begin{align*}
\tau_R=\inf\{t\ge 0: \ubf_t\not\in\domain_R\}.
\end{align*}
Then, there exists a constant $c_*$ sufficiently small such that for all $\kappa_1$, $\kappa_2>0$ satisfying
$$\Big|\frac{\kappa_1}{\kappa_2}-1\Big|<c_*,$$
there exists positive constants $c_1$, $c_2$ and $c_3$ and $R_*$ such that for all $t$ and $R$ 
\begin{equation}\label{ineq:E.Phi:uvz}
\begin{aligned}
\E_{\ubf}\Phi(\ubf_{t\mi\tau_R})
&\le \Phi(\ubf)-c_1\E_{\ubf} \int_0^{t\mi\tau_R}\close\frac{1}{z_s^{2/3}}\big(r_s  ^{p_1+1} +r_s  ^{p_2+1}z_s^{q_2}\big)\mathbf{1}\{ r_s \ge R_*\}\emph{d} s\\
&\qquad -c_2\E_{\ubf} \int_0^{t\mi\tau_R}\close \frac{1}{z_s^{2/3}}\emph{d} s+c_3t.
\end{aligned}
\end{equation}

\end{proposition}
\begin{proof}

We aim to combine the estimates in Section~\ref{sec:Lyapunov:proof} and Section~\ref{sec:flux} with the Tanaka formula~\eqref{eqn:T3def} to establish~\eqref{ineq:E.Phi:uvz}. Recalling $\Phi$ in \eqref{eqn:Phi}, considering $\lambda_1\fO$, in view of the estimates from Section~\ref{sec:outer:lyapprop} and Section~\ref{sec:flux}, together with Tanaka formula~\eqref{eqn:Tanaka}, there exist positive constants $c_1$, $c_2$ such that
\begin{equation} \label{eqn:E.Phi_O}
\begin{aligned}
\E_{\ubf}\lambda_1\fO(\ubf_{t\mi\tau_R})
&\le \lambda_1\fO(\ubf) -c_1\E_{\ubf} \int_0^{t\mi\tau_R}\close\frac{1}{z_s^{2/3}}\big( r_s  ^{p_1+1}z^{q_1}_s+ r_s  ^{p_2+1}z_s^{q_2}\big)\mathbf{1}\{ r_s \ge r_*\}\d s\\
&\qquad+c_2\E_{\ubf} \int_0^{t\mi\tau_R}\close \frac{1}{z_s^{2/3}}\mathbf{1}\{ r_s \le 2r_*\}\d s+c_2\E_{\ubf} \int_0^{t\mi\tau_R}\close \mathbf{1}\{ r_s \le 2r_*, z_s=1\}\d k_s.
\end{aligned}
\end{equation}
Next considering $\psi$, we set
\begin{align} \label{cond:D}
D=\frac{2c_2}{ \min\{ \frac{m \alpha_h}{2}, \frac{\kappa_1 + \kappa_2}{4}\}}.
\end{align}
In light of Corollary~\ref{corollary:psi}, we obtain \begin{equation} \label{eqn:E.D.psi}
\begin{aligned}
&\E_{\ubf}D\psi(\ubf_{t\mi\tau_R})\\
&\le D\psi(\ubf) +D\alphah J\E_{\ubf} \int_0^{t\mi\tau_R}\close \frac{ u_s }{z_s^{2/3}}\d s- \E_{\ubf} \int_0^{t\mi\tau_R}\!\frac{2c_2}{z_s^{2/3}}\d s+DD_1t+DD_2\E_{\ubf} \int_0^{t\mi\tau_R}\close \d k_s.
\end{aligned}
\end{equation}
Concerning the $z$-terms on the right-hand side of~\eqref{eqn:Phi}, we set
\begin{align}\label{cond:A}
A=\frac{D\alpha_h J}{1-h}+DD_2+c_2.
\end{align}
A short computation gives
\begin{align} \label{eqn:Lz}
\L \Big(-\frac{D\alpha_h J}{1-h} \log(z)+Az\Big)&=-D\alpha_h J\frac{u}{z^{2/3}}+(1-h)Auz^{1/3},\\
\Q\Big(-\frac{D\alpha_h J}{1-h} \log(z)+Az\Big)&=\frac{D\alpha_h J}{1-h} -A=-DD_2-c_2\label{eqn:Qz}.
\end{align}
We now collect~\eqref{eqn:E.Phi_O},~\eqref{eqn:E.D.psi},~\eqref{eqn:Lz} and~\eqref{eqn:Qz} to arrive at the estimate
\begin{align*}
\E_{\ubf}\Phi(\ubf_{t\mi\tau_R})
&\le \Phi(\ubf)-c_1\E_{\ubf} \int_0^{t\mi\tau_R}\close\frac{1}{z_s^{2/3}}\big( r_s  ^{p_1+1}z^{q_1}_s+ r_s  ^{p_2+1}z_s^{q_2}\big)\mathbf{1}\{ r_s \ge r_*\}\d s\\
&\qquad -c_2\E_{\ubf} \int_0^{t\mi\tau_R}\close \frac{1}{z^{2/3}_s}\d s+(1-h)A\E_{\ubf} \int_0^{t\mi\tau_R}\close \close u_sz^{1/3}_s\d s+DD_1t.
\end{align*}
In the above, the $\d k_t$ terms were canceled due to the choice of $A$ in~\eqref{cond:A}. Finally, we note that $uz^{1/3}$ is clearly dominated by $r^{p_1+1}z^{-2/3}$ when $r$ is sufficiently large, say, for all $r\ge R_*$ whence
\begin{align*}
\E_{\ubf}\Phi(\ubf_{t\mi\tau_R})
&\le \Phi(\ubf)-c_1\E_{\ubf} \int_0^{t\mi\tau_R}\close\frac{1}{z_s^{2/3}}\big( r_s  ^{p_1+1}z^{q_1}_s+ r_s  ^{p_2+1}z_s^{q_2}\big)\mathbf{1}\{ r_s \ge R_*\}\d s\\
&\qquad -c_2\E_{\ubf} \int_0^{t\mi\tau_R}\close \frac{1}{z^{2/3}_s}\d s+DD_1t.
\end{align*}
Recalling $q_1=0$ by Choice~\ref{cond:pq}, this establishes~\eqref{ineq:E.Phi:uvz}, hence finishing the proof.\end{proof}

We are now in a position to prove Theorem~\ref{thm:geometric.ergodicity} and Proposition~\ref{prop:Lyapunov:xyz}.
\begin{proof}[Proof of Theorem~\ref{thm:geometric.ergodicity} and Proposition~\ref{prop:Lyapunov:xyz}]
As discussed briefly in Section~\ref{sec:mainresult}, upon changing back to the variables $(x,y,z)$, namely, by setting 
\begin{align*}
\Psi(x,y,z):=\Phi\Big(\frac{x}{z^{1/3}},\frac{y}{z^{1/3}} ,z\Big),
\end{align*}
where $\Phi$ is the Lyapunov function from Proposition~\ref{prop:Lyapunov:uvz}, Proposition~\ref{prop:Lyapunov:xyz} (a) immediately follows after noting that $\Phi(u,v,z)\rightarrow \infty$ as $|(u,v)|+ z^{-1}\rightarrow \infty$. Also, we observe that by~\eqref{ineq:E.Phi:uvz}, setting $r=r(x,y)$, we have
\begin{equation}\label{ineq:E.Phi:xyz}
\begin{aligned}
\E_{\xbf}\Psi(\xbf_{t\mi\tau_R})
&\le \Psi(\xbf)-c_1\E_{\xbf} \int_0^{t\mi\tau_R}\!\frac{1}{z_s}\big(r_s  ^{p_1+1}z^{-\frac{1}{3}p_1} +r_s  ^{p_2+1}z_s^{q_2-\frac{1}{3}p_2}\big)\mathbf{1}\{ r_s \ge R_*\}\d s\\
&\qquad -c_2\E_{\xbf} \int_0^{t\mi\tau_R}\close \frac{1}{z_s^{2/3}}\d s+c_3t.
\end{aligned}
\end{equation}
We note that as $r\to\infty$ and/or $z\to 0$, $\Psi(x,y,z)$ is dominated by $r^{p_1}z^{-p_1/3}+r^{p_2}z^{q_2-p_1/3}-\log(z)$. Proposition \ref{prop:Lyapunov:xyz} parts (b) and (c) then readily follow from~\eqref{ineq:E.Phi:xyz}. By combining with Proposition~\ref{prop:minorization}, this proves Theorem~\ref{thm:geometric.ergodicity} (a) and (b). It remains to establish the moment bound~\eqref{ineq:int.r^(p+1).pi(dx)} in Theorem~\ref{thm:geometric.ergodicity} (c).

Let $\pi$ be the unique invariant measure for~\eqref{eqn:xyz}.  We first show that
\begin{align} \label{ineq:int.z^(-2/3).pi(dx)}
\int_\domain \frac{1}{z^{2/3}}\pi(\d \xbf)<\infty.
\end{align}
To see this,  in light of~\eqref{ineq:E.Phi:xyz}, we infer that
\begin{align*}
\frac{1}{t}\int_0^t \E_\xbf\frac{1}{z^{2/3}_s}\d s\le \frac{1}{c_2 t}\Psi(\xbf)+\frac{c_3}{c_2},
\end{align*}
whence
\begin{align*}
\limsup_{t\to\infty} \frac{1}{t}\int_0^t \E_\xbf\frac{1}{z^{2/3}_s}\d s\le \frac{c_3}{c_2}.
\end{align*}
On the other hand, for every $R>0$, by Birkhoff's Ergodic Theorem, it holds that
\begin{align*}
\frac{1}{t}\int_0^t\E_\xbf \Big(\frac{1}{z^{2/3}_s}\mi R\Big) \d s\rightarrow \int_\domain \Big(\frac{1}{z^{2/3}}\mi R\Big)\pi(\d\xbf),\quad t\to\infty.
\end{align*}
By the Monotone Convergence Theorem ($R\rightarrow \infty$),~\eqref{ineq:int.z^(-2/3).pi(dx)} readily follows.

Next, consider the parameters $(p_i,q_i,\alpha_i,\beta_i)$, $i=3,4$ such that $p_i,q_i>0$ and conditions~\eqref{cond:alphai.betai},~\eqref{cond:p2.q2} are satisfied, namely,
\begin{equation} \label{cond:p3q3p4q4}
\begin{aligned}
&0<\alphah p_i-(1-h)q_i=\beta_i<\alpha_i<1,\\
&q_4>\tfrac{1}{3}p_4+\tfrac{1}{2}\alpha_4,\qquad p_4+\tfrac{3}{2}\alpha_4>p_3+\tfrac{3}{2}\alpha_3,
\end{aligned}
\end{equation}
and that
\begin{equation} \label{cond:q3-p_3/3=1}
q_3-\tfrac{1}{3}p_3=1.
\end{equation}
We then consider $\lambda_1\Phitil_\text{O}(u,v,z)$ where $\Phitil_\text{O}$  essentially has the same expression as $\fO$ with $(p_i,q_i,\alpha_i,\beta_i)$, $i=1,2$, replaced by $(p_i,q_i,\alpha_i,\beta_i)$, $i=3,4$, above. Similar to~\eqref{eqn:E.Phi_O}, we have the estimate for all $t,\,R>0$
\begin{align*}
\E_{\ubf}\lambda_1\Phitil_\text{O}(\ubf_{t\mi\tau_R})
&\le \lambda_1\Phitil_\text{O}(\ubf) -c_4\E_{\ubf} \int_0^{t\mi\tau_R}\close\frac{1}{z_s^{2/3}}\big( r_s  ^{p_3+1}z^{q_3}_s+ r_s  ^{p_4+1}z_s^{q_4}\big)\mathbf{1}\{ r_s \ge r_*\}\d s\\
&\qquad+c_5\E_{\ubf} \int_0^{t\mi\tau_R}\close \frac{1}{z_s^{2/3}}\d s+c_5\E_{\ubf} \int_0^{t\mi\tau_R}\close \mathbf{1}\{ r_s \le 2r_*, z_s=1\}\d k_s+c_5t,
\end{align*}
for some positive constants $c_4$ and $c_5$.  By setting 
\begin{align*}
\Phitil(u,v,z)=\lambda_1\ftil_\text{O}(u,v,z)+c_5z,
\end{align*}
we have (after sending $R\to\infty$) for some $R_*$ much larger than $r_*$,
\begin{align*}
\E_{\ubf}\Phitil(\ubf_{t})
&\le \Phitil(\ubf) -c_4\E_{\ubf} \int_0^{t}\close\frac{1}{z_s^{2/3}}\big( r_s  ^{p_3+1}z^{q_3}_s+ r_s  ^{p_4+1}z_s^{q_4}\big)\mathbf{1}\{ r_s \ge R_*\}\d s+c_5\E_{\ubf} \int_0^{t} \frac{1}{z_s^{2/3}}\d s+c_5t.
\end{align*}
Translating back to $(x,y,z)$ and setting $\widetilde{\Psi}(x,y,z)=\Phitil(xz^{-1/3},yz^{-1/3},z)$, we obtain (recalling $q_3-\frac{1}{3}p_3=1$)
\begin{align*}
&\E_{\xbf}\widetilde{\Psi} (\xbf_{t})\\
&\le \widetilde{\Psi} (\xbf) -c_4\E_{\xbf} \int_0^{t}\,\frac{1}{z_s}\big( r_s  ^{p_3+1}z^{q_3-\frac{1}{3}p_3}_s+ r_s  ^{p_4+1}z_s^{q_4-\frac{1}{3}q_4}\big)\mathbf{1}\{ r_s \ge R_*\}\d s+c_5\E_{\xbf} \int_0^{t} \frac{1}{z_s^{2/3}}\d s+c_5t\\
&\le \widetilde{\Psi} (\xbf) -c_4\E_{\xbf} \int_0^{t}\big(r_s  ^{p_3+1}+ r_s  ^{p_4+1}z_s^{q_4-\frac{1}{3}q_4-1}\big)\mathbf{1}\{ r_s \ge R_*\}\d s+c_5\E_{\xbf} \int_0^{t}\frac{1}{z_s^{2/3}}\d s+c_5t.
\end{align*}
It follows that
\begin{align*}
\frac{1}{t}\int_0^t \E_\xbf r_s^{p_3+1}\le \frac{1}{c_4t}\widetilde{\Psi}(\xbf)+\frac{c_5}{c_4t}\int_0^t \E_\xbf\frac{1}{z_s^{2/3}}\d s+\frac{c_5}{c_4}.
\end{align*}
Since $z^{-2/3}$ is integrable with respect to $\pi$ (cf.~\eqref{ineq:int.z^(-2/3).pi(dx)}), we invoke Birkhoff's Ergodic Theorem again to obtain for all $\xbf\in\domain$
\begin{align*}
\limsup_{t\to\infty} \frac{1}{t}\int_0^t \E_\xbf r_s^{p_3+1}\d s\le \frac{c_5}{c_4}\int_\domain\frac{1}{z^{2/3}}\pi(\d\xbf)+\frac{c_5}{c_4}.
\end{align*}
Reasoning as in the argument for~\eqref{ineq:int.z^(-2/3).pi(dx)}, we arrive at
\begin{align*}
\int_\domain r^{p_3+1}\pi(\d\xbf)<\infty.
\end{align*}
In other words, setting $\lambda:=p_3+1$, we see that $(x^2+y^2)^{\lambda/2}=r^{p_3+1}\in L^1(\pi)$. 

To see the bound $\lambda<2/h$, we recall from~\eqref{cond:p3q3p4q4} and~\eqref{cond:q3-p_3/3=1} together with $\alphah=\frac{1+2h}{3}$ as in~\eqref{def:alpha}. It follows that in terms of $\beta_3$ and $h$,
\begin{align*}
p_3=\tfrac{1}{h}(\beta_3+1-h)\quad\text{and}\quad q_3=\tfrac{1}{3h}(\beta_3+1+2h).
\end{align*} 
In particular, since $\beta_3\in(0,1)$, $p_3<2/h-1$, whence $\lambda=p_3+1<2/h$ as claimed. The proof is thus finished.
\end{proof}

\section*{ Acknowlegement} The authors would like to thank David Lipshutz and Jiajun Tong for fruitful discussions on the topic of this paper. All authors are grateful for support from the Department of Mathematics at Iowa State University, where this work started, as well as for support through grants DMS-1612898 and  DMS-1855504 from the National Science Foundation. The authors also would like to thank the anonymous reviewer for their valuable comments and suggestions.

\appendix

\section{Well-posedness} \label{sec:wellposed}

In this section, we briefly recap the framework used to establish well-posedness for SDEs with reflection as in \cite{khasminskii2011stochastic,lions1984stochastic,
pilipenko2014introduction} but tailored to our system~\eqref{eqn:xyz}.

For each $R>0$, we consider a smooth function $\g_R$ such that
\begin{align*}
\g_R(\xbf)=\begin{cases}
1,\quad \xbf\in \domain_{2R},\\
0,\quad \xbf\not\in\domain_{2R+1},
\end{cases}
\end{align*}
where $\domain_R$ is as in~\eqref{eqn:OR}. Next, consider the following truncated version of~\eqref{eqn:xyz}
\begin{equation}\label{eqn:xyz:truncated}
\begin{aligned}
\d x_t &=-\gamma x_t\, \d t-\frac{hx_t^2-y_t^2}{z_t}\g_R(\xbf_t)\, \d t+\sqrt{2\kappa_1}\, \d W^1_t,\\
\d y_t&= -\gamma y_t\, \d t-(1+h)\frac{x_ty_t}{z_t}\g_R(\xbf_t)\, \d t+\sqrt{2\kappa_2}\, \d W^2_t,\\
\d z_t&=(1-h)x_t\,\d t-\d k_t.
\end{aligned}
\end{equation}
Since the above system has globally Lipschitz drift terms, there exists a unique solution pair $(\xbf_t^R,k_t^R)$ for~\eqref{eqn:xyz:truncated} with initial condition $(x,y,z)$ \cite{lions1984stochastic,pilipenko2014introduction}. In particular, thanks to the simple structure of the reflecting boundary $\{z=1\}$, it holds that
\begin{align*}
z^R_t=z+(1-h)\int_0^t\close x^R_s\d s-k^R_t,
\end{align*}
where $k_t^R$ is explicitly given by (cf. \cite[Formula (1.2)]{pilipenko2014introduction})
\begin{align*}
k_t^R=-\min_{0\leq s\leq t}\bigg\{\bigg(1-z-(1-h)\int_0^s\close x^R_\ell\d \ell \bigg) \wedge 0\bigg\}.
\end{align*}
As a consequence, we may define the local solution $(\xbf_t,k_t)$ of~\eqref{eqn:xyz} such that with probability one, for all $0\le t\le \tau_R$,  $(\xbf_t,k_t)=(\xbf_t^R,k^R_t)$ where $\tau_R$ is the stopping time as in~\eqref{eqn:tau_R}. See also \cite[Section 2.4]{pilipenko2014introduction}. Note that nonexplosivity (Proposition~\ref{prop:ne}) follows immediately from Proposition~\ref{prop:Lyapunov:xyz} (b) and the bound
\begin{align*}
\PP_\xbf\{ \tau_R < t \} \leq \frac{ \E_\xbf \Psi(\xbf_{t\wedge \tau_R})}{\Psi_R} \leq \frac{c(x,t)}{\Psi_R},
\end{align*} 
satisfied for all $\xbf, t, R$, since $\Psi_R\rightarrow \infty$ as $R\rightarrow \infty$ by virtue of Proposition~\ref{prop:Lyapunov:xyz} (a).

\section{Generalized Tanaka/It\^{o} formula} \label{sec:Tanaka}

Consider a stochastic process $\xbf_t=(x^1_t,\dots,x^d_t,x^{d+1}_t)\in\rbb^d\times(0,1]$ solving the Skorokhod SDE 
\begin{align} \label{eqn:Skhorohod}
\d x^j_t=f^j(\xbf_t)\d t+\sum_{i=1}^{d+1} \mathrm{A}^{ij}(\xbf_t)\d W^{i}_t+g^{j}(\xbf_t)\d k_t,\quad j=1,\dots,d+1, 
\end{align}
where $\{W^j_t\}_{1\le j\le d+1}$ is a standard $(d+1)$-dimensional Brownian motion, $\mathrm{A}=(\mathrm{A}^{ij})$ is the diffusion matrix, and $k_t$ is the reflective process along the reflective vector field $g=(g^j)$ and satisfies the following properties:
\begin{itemize}
\item $k_0=0$, $k_t$ is an $\mathcal{F}_t-$adapted non-decreasing continuous process, and $\int_0^t |g(\xbf_s)|\d k_s<\infty.$
\item $k_t$ only increases when $x^{d+1}_t=1$, i.e., $\int_0^\infty\boldsymbol{1}\{x^{d+1}_t<1\}\d k_t=0$.
\end{itemize}
Setting $(f^{ij})=\mathrm{A}\mathrm{A}^T$, we consider the following operators associated with $\xbf_t$
\begin{align*}
\L&=\sum_{j=1}^{d+1}f^j\partial_{x^j}+\tfrac{1}{2}\sum_{i,j=1}^{d+1}f^{ij}\partial_{x^ix^j}\quad\text{and}\quad\Q=\sum_{j=1}^{d+1}g^j\partial_{x^j}.
\end{align*}
Let $\f\in C(\rbb^d\times(0,1];\rbb)$ be such that
\begin{align} \label{eqn:phi=phi1.or.phi2}
\f(x^1,\dots,x^d,x^{d+1})=\begin{cases}
\f_{1}(x^1,\dots,x^d,x^{d+1}),& x^d\le b(x^1,\dots,x^{d-1}),\\
\f_2(x^1,\dots,x^d,x^{d+1}),& x^d\ge b(x^1,\dots,x^{d-1}),
\end{cases}
\end{align}
where the $\f_i$'s are $C^2$ on their respective closed domains above and $b\in C^2(\rbb^{d-1};\rbb)$. Our Tanaka formula is given in the following result.
\begin{theorem} \label{thm:Tanaka}
Let $\tau_n=\inf\{t\ge 0: |(x_t^1,\dots,x_t^d)|\ge n \text{ or } x^{d+1}_t<1/n\}$ and $\zeta\in C^1([0,\infty);\rbb^+)$. Then, for all initial conditions $\xbf=(x^1,\dots,x^d,x^{d+1})\in\rbb^d\times(0,1]$,
\begin{equation} \label{eqn:Tanaka}
\begin{aligned}
&\E_{\xbf}[\zeta(t\mi\tau_n)\f(\xbf_{t\mi\tau_n})]\\
&=\zeta(0)\f(\xbf)+\E_\xbf\int_0^{t\mi\tau_n}\close\zeta'(s)\f(\xbf_s)\emph{d} s\\
&\qquad+\E_{\xbf}\int_0^{t\mi \tau_n} \close\zeta(s)\cdot\Big[ \tfrac{1}{2}\L\f(x^1_s,\dots,(x^d_s)^+,x^{d+1}_s)+\tfrac{1}{2}\L\f(x^1_s,\dots,(x^d_s)^-,x^{d+1}_s)\Big]\emph{d}  s\\
&\qquad+\E_{\xbf}\int_0^{t\mi\tau_n}\close\zeta(s)\cdot\Big[ \tfrac{1}{2}\Q\f(x^1_s,\dots,(x^d_s)^+,1)+\tfrac{1}{2}\Q\f(x^1_s,\dots,(x^d_s)^-,1)\Big]\emph{d} k_s\\
&\qquad +\emph{Flux}(\xbf,t,n),
\end{aligned}
\end{equation}
where, for any function $f:\RR^d\times(0,1]\rightarrow \RR$,
\begin{align*}
f(x^1,\dots,(x^d)^+,x^{d+1})&=\lim_{\xi\downarrow x^d}f(x^1,\dots,x^{d-1},\xi,x^{d+1}),\\
f(x^1,\dots,(x^d)^-,x^{d+1})&=\lim_{\xi\uparrow x^d}f(x^1,\dots,x^{d-1},\xi,x^{d+1}),
\end{align*}
provided the limits exist
and \emph{Flux}($\xbf,t,n$) satisfies the following properties:
\begin{itemize}
\item[\emph{(a)}] If $(\partial_{x^d}\f_2-\partial_{x^d}\f_1)\big|_{x^d=b(x^1,\dots,x^{d-1})}\le 0$, then \emph{Flux}($\xbf,t,n)\le 0$ and the maps $t\mapsto\emph{Flux}(\xbf,t,n)$, $n\mapsto\emph{Flux}(\xbf,t,n)$ are decreasing.
\item[\emph{(b)}] If $(\partial_{x^d}\f_2-\partial_{x^d}\f_1)\big|_{x^d=b(x^1,\dots,x^{d-1})}\ge 0$, then \emph{Flux}($\xbf,t,n)\ge 0$ and the maps $t\mapsto\emph{Flux}(\xbf,t,n)$, $n\mapsto\emph{Flux}(\xbf,t,n)$ are increasing.
\end{itemize}
\end{theorem}
The proof of Theorem~\ref{thm:Tanaka} is similar to that of \cite[Theorem 8.1]{herzog2015noiseII}. The only difference is the appearance of the first-order operator $\Q$. Since the argument is relatively short, we include it here for the sake of completeness.  The idea in the proof is to regularize $\f$ using a mollifier, integrate by parts and then pass to the limit. In the case $\f$ is also differentiable across the boundary, i.e., $\f\in C^1$, we refer the reader to \cite{gozzi2006weak} for another version of Tanaka formula.

\begin{proof}[Proof of Theorem~\ref{thm:Tanaka}] 

Since the bivariate process $(\xbf_t,k_t)$ is stopped at time $\tau_n$, without loss of generality, we assume all the functions $\f$, $f^i$, $f^{ij}$ and $g^i$ have compact supports inside a cylinder $\{|(x^1,\dots,x^d)|<R\}\cap \{x^{d+1}>1/R\}$ where $R$ is much larger than $n$.

Let $\chi:\rbb^{d}\times\rbb^+\to\rbb$ be a smooth mollifier, set $\chi_\epsilon(\xbf)=\epsilon^{-d-1}\chi(\epsilon^{-1}\xbf)$ and define
\begin{align*}
\f_{\epsilon}(\xbf)=\chi_\epsilon*\f(\xbf)= \int_{\rbb^d\times\rbb^+}\close \close \chi_\epsilon(\xbf-\ybf)\f(\ybf)\d \ybf.
\end{align*}
Also, by~\eqref{eqn:phi=phi1.or.phi2}, $\f_\epsilon$ can be expressed as
\begin{align} \label{eqn:phi_epsilon=phi1.+.phi2}
\f_{\epsilon}(\xbf)=\chi_\epsilon*(1_{\R^-}\f_1)(\xbf)+\chi_\epsilon*(1_{\R^+}\f_2)(\xbf),
\end{align}
where $\R^-=\{x^d\le b(x^1,\dots,x^{d-1})\}$ and $\R^+=\{x^d\ge b(x^1,\dots,x^{d-1})\}$. Applying Dynkin's formula gives
\begin{align} 
&\E_{\xbf}[\zeta(t\mi\tau_n)\f_\epsilon(\xbf_{t\mi\tau_n})]\label{eqn:LQ.phi_epsilon}\\
&=\zeta(0)\f_\epsilon(\xbf)+\E_\xbf\int_0^{t\mi\tau_n}\close\zeta'(s)\f_\epsilon(\xbf_s)\d s+\E_{\xbf}\int_0^{t\mi \tau_n} \close\zeta(s) \L\f_\epsilon(x^1_s,\dots,x^d_s,x^{d+1}_s)\d  s\nonumber\\
&\qquad+\E_{\xbf}\int_0^{t\mi\tau_n}\close\zeta(s) \Q\f_\epsilon(x^1_s,\dots,x^d_s,1)\d k_s.\nonumber
\end{align}

Concerning $\L\f_\epsilon$, in light of~\cite[Lemma 8.1]{herzog2015noiseII}, we readily have
\begin{align}
&\int_0^{t\mi\tau_n}\close\zeta(s) \L\f_\epsilon(\xbf_s)\d s\label{eqn:L.phi_epsilon}\\
&= \sum_{j=1}^{d+1}\int_0^{t\mi\tau_n}\close\zeta(s) f^j(\xbf_s)(\chi_\epsilon*1_{\R^-}\partial_{x^j}\f_1)(\xbf_s)+f^j(\xbf_s)(\chi_\epsilon*1_{\R^+}\partial_{x^j}\f_2)(\xbf_s)\d s\nonumber \\
&\quad+\tfrac{1}{2}\sum_{i,j=1}^{d+1}\int_0^{t\mi\tau_n}\close\zeta(s) f^{ij}(\xbf_s)(\chi_\epsilon*1_{\R^-}\partial_{x^ix^j}\f_1)(\xbf_s)+f^{ij}(\xbf_s)(\chi_\epsilon*1_{\R^+}\partial_{x^ix^j}\f_2)(\xbf_s)\nonumber \\
&\quad+\tfrac{1}{2} \int_0^{t\mi\tau_n}\close\zeta(s) \int_{\Gamma}(\partial_{x^d}\f_2-\partial_{x^d}\f_1)(\xbf_s)\chi_\epsilon(\xbf_s-\ybf)\sqrt{1+|\nabla b(\ybf)|^2}\sum_{i,j=1}^{d}f^{ij}(\xbf_s)\sigma^i\sigma^j\d S_\Gamma\d s.\nonumber 
\end{align}
In the above, $\Gamma=\{\xbf:x^d=b(x^1,\dots,x^{d-1})\}$ and $\sigma^i$ is the $i$-th component of the unit surface normal vector $\sigma=(-\nabla b(x^1,\dots,x^{d-1}),1)/\sqrt{1+|\nabla b(x^1,\dots,x^{d-1})|^2}$.

Concerning $\Q\f_\epsilon$, we first take the derivatives on both sides of~\eqref{eqn:phi_epsilon=phi1.+.phi2} to obtain
\begin{align*}
\partial_{x^j}\f_\epsilon(\xbf)&=-\partial_{x^j}\chi_\epsilon*(1_{\R^-}\f_1)(\xbf)-\partial_{x^j}\chi_\epsilon*(1_{\R^+}\f_2)(\xbf)\\
&=\chi_\epsilon*(1_{\R^-}\partial_{x^j})\f_1(\xbf)+\chi_\epsilon*(1_{\R^+}\partial_{x^j})\f_2(\xbf_s).
\end{align*}
In the last implication above, we performed an integration by parts using the fact that $\f_1$ and $\f_2$ agree on the boundary surface $\Gamma$. It follows that
\begin{equation}\label{eqn:Q.phi_epsilon}
\begin{aligned}
&\int_0^{t\mi\tau_n}\close\zeta(s) \Q\f_\epsilon(\xbf_s)\d k_s\\
&= \sum_{j=1}^{d+1}\int_0^{t\mi\tau_n}\close\zeta(s)  g^j(\xbf_s)(\chi_\epsilon*1_{\R^-}\partial_{x^j}\f_1)(\xbf_s)+g^j(\xbf_s)(\chi_\epsilon*1_{\R^+}\partial_{x^j}\f_2)(\xbf_s)\d k_s .
\end{aligned}
\end{equation}
Combining~\eqref{eqn:Q.phi_epsilon} with~\eqref{eqn:LQ.phi_epsilon} and~\eqref{eqn:L.phi_epsilon}, since all functions therein are compactly supported and $\d k_s\ge 0$, we obtain by virtue of the Dominated Convergence Theorem as $\epsilon\to0$ 
\begin{align*}
&\E_{\xbf}[\zeta(t\mi\tau_n)\f(\xbf_{t\mi\tau_n})]-\zeta(0)\f(\xbf)-\E_\xbf\int_0^{t\mi\tau_n}\close\zeta'(s)\f(\xbf_s)\d s\\
&\qquad-\E_{\xbf}\int_0^{t\mi \tau_n} \close\zeta(s)\cdot\Big[ \tfrac{1}{2}\L\f(x^1_s,\dots,(x^d_s)^+,z_s)+\tfrac{1}{2}\L\f(x^1_s,\dots,(x^d_s)^-,x^{d+1}_s)\Big]\d  s\\
&\qquad-\E_{\xbf}\int_0^{t\mi\tau_n}\close\zeta(s)\cdot\Big[ \tfrac{1}{2}\Q\f(x^1_s,\dots,(x^d_s)^+,1)+\tfrac{1}{2}\Q\f(x^1_s,\dots,(x^d_s)^-,1)\Big]\d k_s\\
&=\lim_{\epsilon\to 0}\tfrac{1}{2} \int_0^{t\mi\tau_n}\close\zeta(s) \int_{\Gamma}(\partial_{x^d}\f_1-\partial_{x^d}\f_1)(\xbf_s)\chi_\epsilon(\xbf_s-\ybf)\sqrt{1+|\nabla b(\ybf)|^2}\sum_{i,j=1}^{d}f^{ij}(\xbf_s)\sigma^i\sigma^j\d S_\Gamma\d s\\
&=:\text{Flux}(\xbf,t,n),
\end{align*}
which establishes formula~\eqref{eqn:Tanaka}. 

Concerning $\text{Flux}(\xbf,t,n)$, we recall that the matrix $(f^{ij})=\mathrm{A}\mathrm{A}^T$ is non-negative so that
\begin{align*}
\chi_\epsilon(\xbf_s-\ybf)\sqrt{1+|\nabla b(\ybf)|^2}\sum_{i,j=1}^{d}f^{ij}(\xbf_s)\sigma^i\sigma^j\d S_\Gamma\ge 0.
\end{align*}
Combining with $\zeta(s)\ge 0$, it follows that $\text{Flux}(\xbf,t,n)$ satisfies the claimed properties.
\end{proof}

\section{Proof of auxiliary results} \label{sec:auxiliary-result}

\begin{proof}[Proof of Lemma~\ref{lem:mu_h>0}]
Let $\psi \in C(\RR^2)$ be an averaging Lyapunov function corresponding to $\A$ satisfying $\psi=o(\f_1)$ as $ r \rightarrow \infty$  and~\eqref{eqn:avgb} for some $C_1, \epsilon >0$. In view of~\eqref{eqn:E.Phi_O}, $\f_1$ satisfies 
\begin{align*}
\E_{(U, V )} \f_1(U_{t\wedge \tau_n},  V _{t\wedge \tau_n}) \leq \f_1(U, V ) + \E_{(U, V )} \int_0^{t\wedge \tau_n}\close\close  -  c_1  r _s^{p_1+1} +c_2 \, \d s ,
\end{align*}
for all $n\in \mathbf{N}$, $t\geq 0$, $(U, V ) \in \RR^2$. In the above, for a slightly abuse of notation, we set $r_t:=\sqrt{U^2_t+V^2_t}$. By Remark~\ref{rem:averaging-Lyapunov}, any positive scalar multiple of $\psi$ is also an averaging Lyapunov function.  In particular, we note that the function
\begin{align}
\ghat(U, V )= \f_1(U, V ) + \tfrac{ c_2}{\epsilon} \psi(U, V ),
\end{align}
is bounded below by some constant $-K <0$ and satisfies the inequality for some constants $c_1,c_2>0$
\begin{align}
-K \leq \E_{(U, V )} \ghat(U_{t\wedge \tau_n},  V_{t\wedge \tau_n}) \leq \ghat(u, v ) + \E_{(U, V )} \int_0^{t\wedge \tau_n}\close -  c_1 r _s^{1+p_1} + \frac{c_2 C_1}{\epsilon} U_s  \, \d s 
\end{align}
for all $t\geq 0$, $n\in \mathbf{N}$ and $(u, v ) \in \RR^2$.  Note by the monotone and dominated convergence theorems, we can take $n\rightarrow \infty$ to see that, for any $t>0$,
\begin{align*}
-\frac{K}{t} \leq \frac{\ghat(U, V )}{t} + t^{-1} \E_{(U, V )}\int_0^t\close  -  c_1 r_s ^{p_1+1} + \frac{c_2 C_1}{\epsilon} U_s  \, \d s.   
\end{align*}   
Since $r^{p_1+1} \in L^1(\mu)$, it follows that $U\in L^1(\mu)$.  Upon taking $t\rightarrow \infty$ in the above, we find that by Birkhoff's ergodic theorem
\begin{align*}
\int_{\rbb^2}\close -c_1 r ^{p_1+1}+\frac{c_2 C_1}{\epsilon} U \mu(\d U,\d V )\ge 0,
\end{align*}
whence
\begin{align*}
\mu(U)=\int_{\rbb^2}\close U\,\mu(\d U,\d V ) \geq \frac{\epsilon c_1}{c_2C_1 }\int_{\RR^2}\close   r^{p_1+1}  \mu(\d U,\d V )>0.  
\end{align*}
The proof is thus complete.
\end{proof}

\begin{proof}[Proof of Lemma~\ref{lem:R3:G}]
Fix $\eta_*>0$. In view of \cite[Proof of Lemma 7.4]{herzog2015noise}, $G(\eta,s)$ admits the following explicit representation
\begin{align}\label{eqn:R3:G}
G(\eta,s)= \frac{\int_0^\infty t^{a-1} e^{-t^2/2}\cos(\sqrt{b}\eta t)\d t }{\int_0^\infty t^{a-1} e^{-t^2/2}\cos(\sqrt{b}\eta_* t)\d t  },\quad\text{where}\quad a = \frac{s}{h+3/2}\in(0,1),\quad b=\frac{h+3/2}{\kappa_2}>0,
\end{align}
and that its derivatives with respect to $\eta$ can be expressed as
\begin{equation}\label{eqn:R3:lem:G'G''}
\begin{aligned}
G'(\eta,s)&= -\sqrt{b}\frac{\int_0^\infty t^{a} e^{-t^2/2}\sin(\sqrt{b}\eta t)\d t }{\int_0^\infty t^{a-1} e^{-t^2/2}\cos(\sqrt{b}\eta_* t)\d t  },\\
\,\text{and}\,\,\, G''(\eta,s)&= -b\frac{\int_0^\infty t^{a+1} e^{-t^2/2}\cos(\sqrt{b}\eta t)\d t }{\int_0^\infty t^{a-1} e^{-t^2/2}\cos(\sqrt{b}\eta_* t)\d t  }.
\end{aligned}
\end{equation}
We first show that the above denominator is always positive. Indeed, we note that the function $t^{(a-1)/2}e^{-t/2}$ is completely monotone \cite{schilling2012bernstein} and by Hausdorff–Bernstein–Widder theorem, there exists a non-negative measure $\nu$ on $[0,\infty)$ such that for $t>0$,
\begin{align*}
t^{a-1}e^{-t/2} = \int_0^\infty\close  e^{-tx}\nu(\d x).
\end{align*}
We thus obtain the expression
\begin{align*}
\int_0^\infty\close  t^{a-1} e^{-t^2/2}\cos(\omega t)\d t = \int_0^\infty\close \int_0^\infty\close e^{-t^2 x}\nu(\d x)\cos(\omega t)\d t.
\end{align*}
We aim to use Fubini theorem to switch the order of integration. To do so, we note that since $0<a<1$,
\begin{align*}
\int_0^\infty\close \int_0^\infty\close e^{-t^2 x}\nu(\d x) \d t=\int_0^\infty\close  t^{a-1} e^{-t^2/2} \d t <\infty.
\end{align*}
It follows that for every $\omega\in\rbb$,
\begin{equation}\label{eqn:lem:R3:completemonotone}
\begin{aligned}
\int_0^\infty\close t^{a-1} e^{-t^2/2}\cos(\omega t)\d t &= \int_0^\infty\close \int_0^\infty\close e^{-t^2 x}\nu(\d x)\cos(\omega t)\d t\\
&= \int_0^\infty\close \int_0^\infty\close e^{-t^2 x}\cos(\omega t)\d t\,\nu(\d x)\\
&=\int_0^\infty\frac{\sqrt{\pi}}{x}e^{-\omega^2/4x}\nu(\d x)>0.
\end{aligned}
\end{equation}

To prove~\eqref{ineq:R3:eta.G'<0}, we note that~\eqref{eqn:lem:R3:completemonotone} implies that $G(\eta,s)$ is decreasing for $\eta\in[0,\eta_*]$, and by symmetry, is increasing on $\eta\in[-\eta_*,0]$. We thus conclude that $\eta G'(\eta,\beta_i)\leq 0$. 

With regards to the growth rates of $G$ and its derivatives, in light of \cite[Section 12.9]{olver2010nist} and recalling $\eta_*=C\sqrt{\kappa_2}$ as in~\eqref{cond:r*.eta*}, it holds that 
\begin{align*}
\int_0^\infty t^{a-1} e^{-t^2/2}\cos(\sqrt{b}\eta_* t)\d t\sim \Gamma(a)\cos(a\pi/2)\frac{1}{(\sqrt{\beta }\eta_*)^{a+1}},\quad\text{as}\quad \sqrt{\beta}\eta_*=\sqrt{h+3/2}\,C\to\infty.
\end{align*}
The bounds~\eqref{ineq:R3:eta.G'<0}, \eqref{ineq:R3:sup.G'(eta)} and \eqref{ineq:R3:sup.G''(eta)} now follow from the above asymptotics together with formulas~\eqref{eqn:R3:G} and~\eqref{eqn:R3:lem:G'G''}. This finishes the proof.

\end{proof}

\bibliographystyle{abbrv}
{\footnotesize\bibliography{uvz-bib}}

\end{document}